\documentclass[12pt,french]{amsart}
\usepackage{graphics}
\usepackage{cite}
\usepackage{amsmath}
\usepackage{amsfonts}
\usepackage{amssymb}
\usepackage{float}
\usepackage{epic}
\usepackage{eepic}
\usepackage{color}
\usepackage{ae}
\usepackage{mathrsfs,array,graphicx}
\usepackage[all,ps]{xy}
\usepackage{epstopdf}
\usepackage{a4wide}
\usepackage{enumerate}
\usepackage{hyperref} \hypersetup{colorlinks,allcolors=[rgb]{0,0,0}}
\usepackage{arydshln}
\usepackage[a4paper,margin=1in]{geometry}
\usepackage{stmaryrd}
\usepackage{enumitem}
\usepackage{mathtools}
\usepackage[french]{babel}
\usepackage{mathtools}
\usepackage{mathdots}
\usepackage{tikz-cd}

\theoremstyle{theorem}
\newtheorem{thmx}{Théorème}

\newtheorem{theorem}{Théorème}
\newtheorem{proposition}[theorem]{Proposition}

\newtheorem{corollary}[theorem]{Corollaire}
\newtheorem{lemma}[theorem]{Lemme}
\theoremstyle{definition}
\newtheorem{definition}[theorem]{Définition}
\theoremstyle{remark}
\newtheorem{remark}[theorem]{Remarque}
\newtheorem{example}[theorem]{\textbf{Exemple}}
\numberwithin{equation}{section}
\numberwithin{theorem}{section}

\def\hat{\widehat}
\let\li\overline

\newcommand{\tu}[1]{\textup{#1}}

\providecommand{\lim}{\mathop{\mathrm{lim}}\limits}
\newcommand{\Ad}{\tu{Ad}}
\newcommand{\Aut}{\tu{Aut}}
\newcommand{\A}{\mathbb A}

\newcommand{\C}{\mathbb C}

\newcommand{\GL}{{\tu{GL}}}

\newcommand{\GSpin}{\tu{GSpin}}
\newcommand{\GSp}{\tu{GSp}}
\newcommand{\Gal}{{\tu{Gal}}}

\newcommand{\Gm}{\mathbb{G}_{\tu{m}}}
\newcommand{\G}{\mathbb G}

\newcommand{\Hom}{\tu{Hom}}
\newcommand{\Ext}{\tu{Ext}}
\newcommand{\RHom}{\tu{RHom}}
\newcommand{\Tot}{\tu{Tot}}

\newcommand{\Ind}{\tu{Ind}}

\newcommand{\LG}{{}^L G}

\newcommand{\Lie}{\tu{Lie}\,}

\newcommand{\PGL}{\tu{PGL}}

\newcommand{\Q}{\mathbb Q}
\newcommand{\Qbar}{\overline{\mathbb{Q}}}
\newcommand{\Res}{\tu{Res}}
\newcommand{\R}{\mathbb R}

\newcommand{\SL}{\tu{SL}}
\newcommand{\SO}{{\tu{SO}}}
\newcommand{\SU}{{\tu{SU}}}

\newcommand{\Sp}{\tu{Sp}}

\newcommand{\Sym}{{\tu{Sym}}}

\newcommand{\Tr}{\tu{Tr}\,}

\newcommand{\Zhat}{\widehat{\mathbb{Z}}}

\newcommand{\Z}{\mathbb Z}
\newcommand{\Zp}{\mathbb{Z}_p}
\newcommand{\ad}{\tu{ad}}

\newcommand{\cH}{\mathcal H}
\newcommand{\cL}{\mathcal L}
\newcommand{\cN}{\mathcal N}
\newcommand{\cO}{\mathcal O}
\newcommand{\cR}{\mathcal{R}}
\newcommand{\cS}{\mathcal S}

\newcommand{\cusp}{\tu{cusp}}
\newcommand{\der}{\tu{der}}
\newcommand{\sico}{\tu{sc}}
\newcommand{\diag}{{\tu{diag}}}

\newcommand{\eps}{\varepsilon}

\newcommand{\afr}{\mathfrak{a}}
\newcommand{\gfr}{\mathfrak{g}}
\newcommand{\kfr}{\mathfrak{k}}
\newcommand{\lfr}{\mathfrak{l}}
\newcommand{\mfr}{\mathfrak{m}}
\newcommand{\nfr}{\mathfrak{n}}
\newcommand{\pfr}{\mathfrak{p}}
\newcommand{\qfr}{\mathfrak{q}}
\newcommand{\ufr}{\mathfrak{u}}

\newcommand{\iS}{\mathfrak S}
\newcommand{\ig}{\mathfrak g}

\newcommand{\inv}{^{-1}}
\newcommand{\isomto}{\overset \sim \to}

\newcommand{\lbr}{\left\lbrace}

\newcommand{\lhk}{\left(}

\newcommand{\ol}{\overline}
\newcommand{\ul}{\underline}
\newcommand{\one}{\textbf{\tu{1}}}

\newcommand{\rbr}{\right\rbrace}

\newcommand{\rhk}{\right)}

\newcommand{\std}{\tu{std}}

\newcommand{\surjects}{\twoheadrightarrow}
\newcommand{\til}[1]{\tilde{#1}}

\newcommand{\uH}{\tu{H}}
\newcommand{\uM}{\tu{M}}

\newcommand{\vierkant}[4]{{\lhk \begin{smallmatrix} #1 & #2 \cr #3 & #4 \end{smallmatrix} \rhk } }
\newcommand{\vol}{\tu{vol}}
\newcommand{\wh}[1]{\widehat{#1}}

\newcommand{\wt}{\widetilde}

\newcommand{\myotimes}[1]{\mathbin{\mathop{\otimes}\displaylimits_{#1}}}

\newcommand{\Std}{\mathrm{Std}}
\newcommand{\LLC}{\mathrm{LL}}
\newcommand{\lng}{\operatorname{lg}}
\newcommand{\res}{\operatorname{res}}

\usepackage[continuous, page]{pagenote}
\makepagenote
\definecolor{green}{rgb}{0, 0.7, 0}

\newcommand{\disc}{\tu{disc}}
\renewcommand{\sp}{\textup{sp}}
\newcommand{\ind}{\tu{ind}}
\newcommand{\ff}{\textup{f}}
\newcommand{\uL}{\textup{L}}
\newcommand{\ct}{\tu{ct}}
\newcommand{\cT}{{\mathcal T}}
\newcommand{\cB}{{\mathcal B}}
\newcommand{\cA}{{\mathcal A}}
\newcommand{\cZ}{{\mathcal Z}}
\newcommand{\cF}{{\mathcal{F}}}
\newcommand{\cG}{{\mathcal{G}}}
\newcommand{\cX}{{\mathcal{X}}}
\newcommand{\id}{\textup{id}}
\newcommand{\LM}{ {}^L M}
\newcommand{\LH}{ {}^L H}

\newcommand{\ih}{{\mathfrak h}}
\newcommand{\nn}{{\mathfrak n}}
\newcommand{\ti}[1]{\ut #1 \inv}
\newcommand{\autom}[2]{ \Phi_{#1} \ti{#2} \Phi_{#1}\inv}
\newcommand{\LL}{ {}^L }

\newcommand{\ord}{{\tu{ord}}}
\newcommand{\loccit}{\textit{loc. cit.}}
\newcommand{\vierkantdrie}[9]{{\lhk \begin{smallmatrix} #1 & #2 & #3 \cr #4 & #5 & #6 \cr #7 & #8 & #9 \end{smallmatrix}\rhk}}

\makeatletter
\let\@wraptoccontribs\wraptoccontribs
\def\@@and{et}
\makeatother

\begin{document}

\author{Laurent Clozel}
\author{Arno Kret}
\author{Olivier Taïbi}

\contrib[appendice par]{Olivier Taïbi}
\contrib[]{Jean-Loup Waldspurger}

\title{Invariance galoisienne des z\'eros centraux de fonctions L}
\maketitle

\begin{abstract}
  Nous démontrons l'invariance Galoisienne de la propriété d'annulation en \(1/2\) des fonctions \(L\) standard ou de Rankin pour certaines représentations automorphes cuspidales algébriques régulières autoduales ou autoduales conjuguées de groupes linéaires sur un corps de nombres arbitraire.
  Il s'agit de l'analogue côté automorphe d'une conjecture de Deligne portant sur les fonctions \(L\) de motifs purs sur les corps de nombres.
  La démonstration repose sur l'utilisation de la cohomologie pondérée de Goresky-Harder-MacPherson
  et sur la construction de certaines représentations automorphes discrètes pour les groupes classiques comme résidus de séries d'Eisenstein, mais l'abandon de l'hypothèse ``\(F\) totalement réel'' introduit de nouvelles difficultés concernant certains opérateurs d'entrelacement.
  Celles-ci sont résolues grâce à l'appendice, rédigé par J.-L. Waldspurger et l'un d'entre nous, démontrant l'holomorphie et la non-annulation de certains opérateurs d'entrelacement normalisés.
  Nous démontrons également l'invariance Galoisienne des constantes des équations fonctionnelles (``root numbers'') correspondantes, impliquant l'invariance Galoisienne de la parité de l'ordre d'annulation en \(1/2\) de ces fonctions \(L\).
\end{abstract}

\tableofcontents

\section*{Introduction}

Les valeurs des fonctions $L$ aux entiers critiques sont des objets centraux en géométrie arithmétique moderne. Elles sont conjecturalement liées aux invariants arithmétiques globaux des motifs, tels que les régulateurs et l'ordre des groupes de Shafarevich-Tate. Dans son article fondateur de 1979 \cite{Deligne1979}, Deligne a formulé une conjecture précise concernant la nature algébrique de ces valeurs.

Soit $M$ un motif pur sur un corps de nombres $F$ à coefficients dans un corps de nombres $E$.
On peut (au moins conjecturalement) supposer que \(M\) est pur de poids \(w \in \Z\).
Fixons un plongement $\sigma \colon E \hookrightarrow \C$. À $M$, on associe une fonction $L$ relative à $\sigma$, notée $L(s, M, \sigma)$. Deligne a identifié un ensemble d'entiers critiques $m$ pour $L(s, M, \sigma)$, basé sur la structure de Hodge de $M$. Il a également défini des périodes $c^\pm(M, \sigma)$ via l'isomorphisme de comparaison entre les cohomologies de Betti et de de Rham. La conjecture stipule que pour tout entier critique $m$, il existe un entier $k \in \Z$ et un signe $\pm$ appropriés tels que la valeur normalisée
$$
\alpha(M, m, \sigma) = \frac{L(m, M, \sigma)}{(2\pi i)^k c^\pm(M, \sigma)}
$$
appartient à $\sigma(E) \subset \C$.

De plus, Deligne a formulé une propriété d'équivariance Galoisienne \cite[Conjectures 2.7 et 2.8]{Deligne1979}.
Soit $a \in \Aut(\C)$ est un automorphisme du corps des nombres complexes (seule sa restriction à la clôture algébrique de \(\Q\) jouera un rôle).
La conjecture prédit entre autres que la formation de ces valeurs spéciales est équivariante sous cette action, c'est-à-dire qu'on s'attend à avoir $a (\alpha(M, m, \sigma)) = \alpha(M, m, a \circ \sigma)$.
Puisque les périodes et les puissances de $2\pi i$ sont non nulles, une conséquence fondamentale de cette conjecture est l'invariance par le groupe de Galois de l'annulation des valeurs critiques:
\[ L(m, M, \sigma) = 0 \iff L(m, M, a \circ \sigma) = 0. \]
Deligne conjecture que pour \(m\) critique \(L(s, M, \sigma)\) n'a pas de pôle en \(s=m\), et ne peut s'annuler que si \(w=2m-1\) (\(m=(w+1)/2\) est alors le ``centre'' de l'équation fonctionnelle).
Lorsque le poids \(w\) est impair Deligne conjecture en outre, suivant Gross, l'invariance de l'ordre d'annulation:
\[ \ord_{s=(w+1)/2} L(s, M, \sigma) = \ord_{s=(w+1)/2} L(s, M, a \circ \sigma). \]

Il devrait exister une famille finie \((\pi_i)_{i \in I}\), où \(\pi_i\) est une représentation automorphe cuspidale pour \(\GL_{N_i,F}\) (à caractère central unitaire) satisfaisant
\[ L(s, M, \sigma) = \prod_{i \in I} L(s-\tfrac{w}{2},\pi_i). \]
De plus chaque \(\pi_i[(w+N_i-1)/2]\) devrait être algébrique (au sens de \cite[Définition 1.8]{Clozel_AA}), où l'on note \(\pi[s]\) la tordue de \(\pi\) par \(|\det|^s\).
On suppose dans la suite \(w\) impair (donc \(\pi_i[N_i/2]\) algébrique).
On s'attend à avoir
\[ L(s, M, a \circ \sigma) = \prod_{i \in I} L(s-\tfrac{w}{2}, a(\pi_i[N_i/2])[-N_i/2]) \]
où pour une représentation automorphe cuspidale algébrique \(\pi\) pour \(\GL_{N,F}\) sa conjuguée par Galois \(a(\pi)\), dont l'existence est conjecturale en général, est caractérisée par sa partie finie $a(\pi)_{\ff} \simeq \pi_{\ff} \otimes_{\C, a} \C$ (voir \cite[\S\S 3-4]{Clozel_AA}, qui décrit également la partie Archimédienne).
(L'existence de \(a(\pi)\) satisfaisant ces conditions est connue sous des hypothèses de régularité aux places Archimédiennes.)
Ainsi la conjecture de Deligne portant sur les ordres d'annulation se traduit du côté automorphe comme suit.
Pour \(\pi\) une représentation automorphe cuspidale (à caractère centrale unitaire) pour \(\GL_{N,F}\) telle que \(\pi[N/2]\) soit algébrique on devrait avoir
\[ \ord_{s=1/2} L(s, \pi) = \ord_{s=1/2} L(s, a(\pi[N/2])[-N/2]). \]

Dans cet article, nous démontrons l'invariance sous $\Aut(\C)$ de l'annulation et de la parité de l'ordre d'annulation au point central $s = \tfrac12$ des fonctions $L$ standard des représentations cuspidales autoduales, et des fonctions $L$ de Rankin attachées à des paires de représentations cuspidales autoduales ou autoduales conjuguées de groupes linéaires généraux, sous des hypothèses d'algébricité, de régularité et de signe (alternative symplectique/orthogonal).
Sous ces hypothèses nous montrons que les propriétés
$$
L (\tfrac12, \pi\times\rho)\neq 0, \qquad \ord_{s = 1/2}L(s, \pi\times\rho)\equiv 0\pmod 2
$$
sont préservées lorsque $(\pi, \rho)$ est remplacé par $(\tilde{a}(\pi), \tilde{a}(\rho))$, où \(\tilde{a}(\pi) = a(\pi)\) si \(\pi\) est algébrique (resp.\ \(\tilde{a}(\pi) = a(\pi[1/2])[-1/2]\) si \(\pi[1/2]\) est algébrique).
Nous dirons que \(\pi\) est \emph{demi-algébrique} si \(\pi[1/2]\) est algébrique.
Sans entrer dans les détails ici, notre hypothèse d'algébricité implique notamment que le motif pur à coefficients dans \(E \subset \C\) (suffisamment grand) correspondant conjecturalement à \(L(s, \pi \times \rho)\) (à torsion de Tate près) est de poids impair.

Récemment, deux d'entre nous ont appliqué dans \cite{ClozelKret} la méthode de la cohomologie d'Eisenstein pour étudier l'invariance par $\Aut(\C)$ des valeurs centrales. On part d'une représentation cuspidale algébrique superrégulière $\pi$ de $\GL_{2n}(\A_F)$, avec $F$ totalement réel, et l'on suppose $\pi$ symplectique, i.e. $L(s,\pi,\wedge^2)$ a un pôle en $s=1$. On considère alors des séries d'Eisenstein sur $\Sp_{4n}$ induites depuis le parabolique de Siegel (Levi $\GL_{2n}$) : les facteurs de normalisation du terme constant y font apparaître $L(2s,\pi,\wedge^2)$, qui a un pôle en $s=\tfrac12$ sous l'hypothèse symplectique. L'analyse du résidu relie l'apparition d'un pôle résiduel (et de la représentation résiduelle de carré intégrable correspondante) à la non-annulation $L(\tfrac12,\pi)\neq 0$ (à des hypothèses auxiliaires près). Dans \cite{ClozelKret}, l'hypothèse ``$F$ totalement réel'' permet en outre d'utiliser la conjecture de Zucker \cite{Looijenga_Zucker} pour identifier la cohomologie $L^2$ à la cohomologie d'intersection de la compactification de Baily-Borel, fournissant ainsi une $\Q$-structure canonique sur les classes pertinentes. Combinée à la classification d'Arthur \cite{ArthurBook}, cette structure permet de transporter l'existence de classes à $a(\pi)$ et de montrer qu'elles proviennent bien de résidus de séries d'Eisenstein, d'où $L(\tfrac12,a(\pi))\neq 0$ (et, sous des hypothèses supplémentaires, des résultats analogues pour certaines fonctions $L$ de Rankin).

Le présent article généralise de manière significative les méthodes et résultats de \cite{ClozelKret}. Le changement décisif est l'abandon du cadre des variétés de Shimura, ce qui nous affranchit des hypothèses ``corps totalement réel/CM'' et nous permet de traiter des situations non hermitiennes beaucoup plus générales. Sur le plan cohomologique, nous remplaçons la cohomologie d'intersection des compactifications de Baily--Borel par la cohomologie pondérée de \cite{GoreskyHarderMacPherson}, qui est précisément le réceptacle naturel des contributions résiduelles issues des pôles des séries d'Eisenstein et donc des mécanismes de détection des non-annulations de valeurs $L$. Les résultats de Nair et Rai \cite{NairRai}, établissant une structure rationnelle sur la cohomologie pondérée ainsi qu'une décomposition spectrale compatible, fournissent alors le cadre arithmétique permettant de suivre l'action de $\Aut(\C)$ au-delà du monde hermitien.

Ce passage hors du domaine des variétés de Shimura introduit toutefois de nouvelles difficultés. D'une part, on ne dispose plus, dans la généralité visée, des outils géométriques qui rigidifient les degrés cohomologiques et les contributions spectrales. D'autre part, certaines estimations analytiques usuelles (souvent motivées par des avatars de Ramanujan) ne sont plus automatiquement disponibles pour contrôler a priori les phénomènes de croissance et de normalisation intervenant dans les séries d'Eisenstein. Il devient alors indispensable de rendre explicite et robuste la comparaison des normalisations, en particulier via une étude fine des opérateurs d'entrelacement locaux qui gouvernent les termes constants et les résidus. C'est précisément l'objet de l'appendice, où nous analysons ces opérateurs (et leurs facteurs de normalisation) de façon suffisamment précise pour assurer, dans ce contexte non hermitien, la compatibilité rationnelle requise et le transfert des propriétés de non-annulation sous l'action de $\Aut(\C)$.

\bigskip

Nous énonçons maintenant nos résultats principaux. Nous supposons que $F$ est un corps de nombres. Notre premier résultat principal concerne la fonction $L$ standard. Soit $\pi_v$ une représentation admissible de $\GL_{2n}(\C)$ ou $\GL_{2n}(\R)$. Nous disons que $\pi_v$ est \emph{superrégulière} si, à l'action du groupe de Weyl près, son caractère infinitésimal est de la forme .
$$
z \mapsto ((z/\li z)^{p_1}, \ldots (z/\li z)^{p_m}, (z/\li z)^{-p_m}, \ldots, (z/\li z)^{-p_1})
\quad 
\begin{cases} \textup{$p_i \geq p_{i+1} + 2$ ($i < m$)} \cr \textup{$p_m \geq \tfrac 32$} 
\end{cases}
$$ 

\begin{thmx}[Théorème~\ref{thm:Clozel-Standard}]\label{thm:A}
Soit $\pi$ une représentation cuspidale algébrique régulière, autoduale et symplectique de $\GL_{2n}(\A_F)$, telle que $\pi_v$ soit superrégulière pour tout $v | \infty$. Alors pour tout automorphisme $a \in \Aut(\C)$ nous avons
$$
L(1/2, \pi) = 0 \iff L(1/2, a(\pi)) = 0.
$$
\end{thmx}

Notre second résultat principal concerne les fonctions $L$ de Rankin. Soient $\pi$ et $\rho$ des représentations cuspidales algébriques régulières et autoduales de $\GL_{r}(\A_F)$ et $\GL_{t}(\A_F)$ respectivement, avec $r \geq 2$ pair et $t \geq 3$ impair.
Nous supposons que \(\pi\) est symplectique et \(\rho\) est orthogonale (rappelons \cite[Proposition 1.1]{ClozelKret} que cette condition est automatiquement vérifiée si \(F\) n'est pas totalement complexe).
Nous disons que $\pi, \rho$ sont \emph{disjointes} en $v | \infty$ si $p_i \pm 1/2 \neq q_j$ où $p = (p_1 > \ldots -p_1)$ et $q = (q_1 > \ldots > q_m > 0 > \ldots > -q_1)$ sont les caractères infinitésimaux de $\pi$ et $\rho$ respectivement.

\begin{thmx}[Théorème~\ref{thm:Clozel-Rankin}]\label{thm:B}
Soient $\pi$ et $\rho$ comme ci-dessus. Supposons qu'elles soient disjointes pour tout $v | \infty$, et que $\pi_v$ soit superrégulière. Alors pour tout $a \in \Aut(\C)$ nous avons
$$
L(1/2, \pi \times \rho) = 0 \iff L(1/2, a (\pi) \times a(\rho)) = 0.
$$
\end{thmx}

Dans la Section~\ref{sec:FunctionRankin}, nous généralisons le Théorème~\ref{thm:B}.
Nous commençons ici avec des entiers positifs $r, t$ (sans borne inférieure, ni restriction sur la parité) et soient $\pi, \rho$ des représentations automorphes cuspidales autoduales pour $\GL_r(\A_F)$ et $\GL_t(\A_F)$ respectivement. De nouveau, nous supposons que parmi $\pi$ et $\rho$, l'une est orthogonale et l'autre est symplectique. Pour la représentation symplectique, nous supposons qu'elle est algébrique régulière, et l'orthogonale est algébrique régulière (resp.\ demi-algébrique régulière\footnote{Notre hypothèse de régularité dans ce cas est en fait légèrement plus faible: voir la Définition \ref{def:SO_reg}.}) si elle est de dimension impaire (resp. de dimension paire).


\begin{thmx}[Théorème~\ref{thm:equiv_Sp_SO}]\label{thm:C}
Soient $\pi$ et $\rho$ comme ci-dessus. Supposons qu'elles soient disjointes pour tout $v | \infty$, et que $\pi_v$ soit superrégulière. Alors pour tout $a \in \Aut(\C)$, 
$$
L(1/2, \pi \times \rho) = 0 \iff L(1/2, \tilde a (\pi) \times \tilde a(\rho)) = 0.
$$
\end{thmx}

Dans la Section~\ref{sec:Eps}, nous prouvons un résultat concernant l'invariance des facteurs epsilon au point de symétrie $s = 1/2$. Soient $r, t$ des entiers positifs. Soient $\pi, \rho$ des représentations automorphes cuspidales autoduales pour $\GL_r(\A_F)$ et $\GL_t(\A_F)$ respectivement, avec $\pi, \rho$ de types opposés comme précédemment (orthogonale/symplectique). Nous faisons les mêmes hypothèses de régularité et d'algébricité sur $\pi, \rho$ que pour le Théorème~\ref{thm:B}.  

\begin{thmx}[Théorème~\ref{thm:inv_epsilon_symp}]\label{thm:D}
Soient $\pi$ et $\rho$ comme ci-dessus. Alors pour tout $a \in \Aut(\C)$ nous avons
$$
\eps(1/2, a(\pi) \times a(\rho)) = a \lhk \eps(1/2, \pi \times \rho) \rhk.
$$
\end{thmx}

Dans le Corollaire~\ref{cor:inv_pari_ord_sympl}, nous déduisons l'invariance de la parité de l'ordre d'annulation des fonctions \(L\).

Enfin, nous appliquons ces méthodes aux représentations autoduales conjuguées, que nous étudions à l'aide des groupes unitaires. Soit $E/F$ une extension quadratique (nous insistons, toute extension quadratique est permise). On note \(\{1,c\} = \Gal(E/F)\). Soient $n,r$ des entiers positifs. Soient $\pi$ et $\rho$ des représentations automorphes cuspidales autoduales conjuguées (\(\pi^\vee \simeq \pi^c\) et \(\rho^\vee \simeq \rho^c\)) régulières de $\Res_{E/F} \GL_n$ et $\Res_{E/F} \GL_r$ respectivement.
Si $n+r$ est impair, nous supposons que $\pi$ et $\rho$ sont algébriques et si $n+r$ est pair, nous supposons que $\pi[1/2]$ et $\rho$ sont algébriques. Nous supposons que $\pi$ est superrégulière et que les caractères infinitésimaux du couple $\pi, \rho$ sont disjoints pour chaque place archimédienne (Équation~\eqref{eq:SuperregularityDisjointness}).
Enfin nous supposons l'analogue de l'hypothèse dans le cas autodual que $\pi, \rho$ sont de types (orthogonal/symplectique) différents, à savoir que $\eta(\pi) = (-1)^r$ et $\eta(\rho) = (-1)^{r-1}$ (voir Définition~\ref{def:Parite} et Équation~\eqref{eq:SignCondition}).

\begin{thmx}[Théorème~\ref{thm:TheoremeUnitaire}]\label{thm:E}
Soient $\pi$ et $\rho$ comme ci-dessus. Alors pour tout $a \in \Aut(\C)$ nous avons 
$$
L(1/2, \pi \times \rho) \neq 0 \iff L(1/2, \tilde a(\pi) \times a(\rho)) \neq 0.
$$
\end{thmx}

Soit \(E/F\) une extension quadratique de corps de nombres. Soient \(\pi\) et \(\rho\) des représentations automorphes cuspidales conjuguées autoduales pour \(\GL_{r,E}\) et \(\GL_{t,E}\). Supposons que \(\pi\) et \(\rho\) sont de types opposés, qu'elles sont algébriques ou demi-algébriques régulières et qu'en toute place archimédienne \(v\) de \(E\) les poids de \(\LLC(\pi_v) \otimes \LLC(\rho_v)\) sont demi-entiers.

\begin{thmx}[Théorème~\ref{thm:inv_epsilon_conjautodual}]\label{thm:F}
  Alors pour tout \(a \in \Aut(\C)\) on a
  \[ \epsilon(1/2, \tilde{a}(\pi) \times \tilde{a}(\rho)^\vee) = \epsilon(1/2, \pi \times \rho^\vee). \]
\end{thmx}

Dans le corollaire~\ref{cor:inv_pari_ord_sympl}, nous déduisons l'invariance de la parité de l'ordre d'annulation, dans le cas conjugué auto-dual.

\bigskip

La stratégie des preuves repose sur l'interaction entre la théorie analytique des séries d'Eisenstein et la théorie arithmétique de la cohomologie des groupes arithmétiques. Nous construisons un groupe réductif $G$ (par exemple $\Sp_{2N}$) qui contient un sous-groupe de Levi $M$ isomorphe \`a $\GL_r \times G'$, où $G'$ est un groupe classique plus petit. Les représentations $\pi$ (sur $\GL_r$) et $\sigma$ (sur $G'$, associée à $\rho$ via l'endoscopie) donnent lieu à une représentation sur $M$. Nous considérons les séries d'Eisenstein induites $E(f, s)$ attachées à la représentation $\pi \otimes \sigma$ sur $M$. Le comportement analytique de $E(f, s)$ en $s = 1/2$ est contrôlé par la formule du terme constant de Langlands. Le terme constant implique des rapports de fonctions $L$. Dans l'exemple mentionné où $G = \Sp_{2N}$, l'opérateur de terme constant $M(s)$ agit essentiellement comme la multiplication par un facteur impliquant :
$$
\frac{L(s, \pi \times \rho)}{L(s+1, \pi \times \rho)} \frac{L(2s, \pi, \Lambda^2)}{L(2s+1, \pi, \Lambda^2)}. 
$$
En supposant que le pôle approprié existe pour le facteur auxiliaire (par exemple, $L(s, \pi, \Lambda^2)$ ayant un pôle en $s = 1$, ce qui correspond à $2s$ en $s = 1/2$), la non-annulation de la valeur centrale $L(1/2, \pi \times \rho)$ détermine si la série d'Eisenstein a un pôle en $s = 1/2$. Plus précisément, si $L(1/2, \pi \times \rho) \neq 0$, l'opérateur normalisé $M(s)$, figurant dans l'équation fonctionnelle, a un pôle, et la série d'Eisenstein $E(f,s)$ admet un résidu. Ce résidu définit une forme automorphe non nulle de carré intégrable appartenant au spectre résiduel de $G$.

Si la série d'Eisenstein a un pôle, la représentation résiduelle $\Pi$ associ\'ee au r\'esidu $\tu{Res}_{s = 1/2} E(f, s)$ engendre une classe non triviale dans la cohomologie de l'espace localement symétrique attaché à $G$, à coefficients dans un système local $V_\lambda$. C'est ici qu'intervient les conditions de ``superrégularité" (et ``disjonction'' dans le cas Rankin): elles assurent que $\Pi_\infty$ possède le caractère infinitésimal correct pour avoir une $(\mathfrak{g}, K)$-cohomologie non nulle.
L'apport arithmétique clé est le résultat de Nair et Rai \cite{NairRai}, s'appuyant sur les travaux de Franke. Ils ont prouvé que la cohomologie pondérée
$$
W^\pm H^\bullet(G, \mathcal{X}; V) \cong H^\bullet(\mathfrak{g}, K; \mathcal{A}^2(G) \otimes V)
$$
admet une structure $\Q$-rationnelle. L'espace $\mathcal{A}^2(G)$ inclut à la fois les formes cuspidales et résiduelles. L'action de $a \in \Aut(\C)$ sur la partie finie de la cohomologie permute les espaces propres de Hecke. Spécifiquement, si une représentation $\Pi_{\ff}$ apparaît dans la cohomologie, alors sa conjuguée $a(\Pi_{\ff})$ doit également apparaître.

Par conséquent, la chaîne logique est :
\begin{enumerate}
\item $L(1/2, \pi \times \rho) \neq 0 \implies$ La série d'Eisenstein a un pôle.
\item Pôle $\implies$ Existence d'une représentation résiduelle $\Pi$ dans $L^2_{disc}(G)$.
\item $\Pi$ est cohomologique $\implies \Pi_{\ff}$ contribue à $H^\bullet(\mathfrak{g}, K; \mathcal{A}^2(G) \otimes V)$.
\item Rationalité de la cohomologie $\implies a(\Pi_{\ff})$ contribue à la cohomologie (avec coefficients conjugués).
\item $a (\Pi_{\ff})$ doit être la partie finie d'une certaine représentation automorphe $\Pi'$ dans le spectre discret de $G$.
\end{enumerate}

L'étape finale consiste à inverser l'implication pour la représentation conjuguée. Nous savons que $a( \Pi_{\ff})$ existe dans le spectre discret. Nous devons montrer que cette représentation $\Pi'$ correspond à un résidu d'une série d'Eisenstein induite depuis $a(\pi)$ et $a(\rho)$. Cela nécessite d'exclure la possibilité que $\Pi'$ soit cuspidale ou provienne d'un sous-groupe parabolique différent. Nous utilisons la classification endoscopique d'Arthur du spectre discret. La représentation $\Pi$ appartient à un paquet d'Arthur global déterminé par un paramètre $\Psi$. Puisque $\Pi$ est un résidu d'une série d'Eisenstein induite de $\pi \otimes \sigma$, son paramètre d'Arthur est construit à partir des paramètres de $\pi$ et $\sigma$ (impliquant spécifiquement $\pi \otimes \tu{sp}(2)$). Nous soutenons que le paramètre d'Arthur de la représentation conjuguée $\Pi'$ doit être le conjugué de Galois du paramètre de $\Pi$. Cela implique que $\Pi'$ ne peut pas être cuspidale (car son paramètre implique un facteur $\tu{SL}_2$ \'egale \`a $ \tu{sp}(2)$ correspondant au résidu) et doit être supportée dans le sous-groupe parabolique correspondant à $a(\pi)$ et $a(\rho)$. Par conséquent, $\Pi'$ est un résidu de la série d'Eisenstein conjuguée. L'existence de ce pôle implique la non-annulation des facteurs dans le terme constant pour la représentation conjuguée, donnant $L(1/2, a (\pi) \times a(\rho)) \neq 0$.

Enfin, l'Appendice \ref{sec:appendice}, rédigée par l'un d'entre nous et J.-L. Waldspurger, est consacrée aux propriétés analytiques des opérateurs d'entrelacement locaux qui apparaissent dans la formule du terme constant. Pour établir l'équivalence précise entre la non-annulation de la valeur centrale $L(1/2, \pi \times \rho)$ et l'existence d'un pôle pour la série d'Eisenstein, il est impératif de s'assurer que les contributions locales n'introduisent ni zéros ni pôles parasites. Nous démontrons dans cette annexe que, pour les données induisantes considérées, les opérateurs d'entrelacement normalisés sont holomorphes et non nuls au point critique $s=1/2$. Cette régularité locale est indispensable pour garantir que l'existence d'une classe de cohomologie résiduelle est strictement régie par le comportement de la fonction $L$ globale.

\noindent \textbf{Remarque.} Nous avons inclus dans cet article les Sections~1 et~2, qui permettent de démontrer le Thm.~2.3. Celui-ci était nécessaire (et omis à tort) aux démonstrations de notre article précédent~\cite{ClozelKret}.

\bigskip

\noindent \textbf{Mise en garde.}
Nos théorèmes principaux (Thm.~\ref{thm:A}--Thm.~\ref{thm:F}) reposent de manière cruciale sur la classification endoscopique du spectre discret pour les groupes classiques, telle qu'établie dans la monographie d'Arthur~\cite{ArthurBook} et, pour les groupes unitaires quasi-déployés (Théorème~\ref{thm:E}), dans les travaux de Mok~\cite{Mok}. Ces textes fondateurs, parus vers 2013 et 2015, dépendaient de plusieurs résultats anticipés (référencés comme [A24--A27] dans~\cite{ArthurBook}) qui sont à ce moment encore non publiés. De plus, la classification est conditionnelle à la preuve du Lemme Fondamental Pondéré Tordu (LFPT) (voir Conjectures~3.6 et 3.7 dans~\cite{Waldspurger_lemma_fond_pond_tordu}). Des progrès significatifs ont été réalisés pour traiter ces dépendances. Dans une prépublication récente, Atobe, Gan, Ichino, Kaletha, Mínguez et Shin \cite{AtobeGanIchinoKalethaMinguezShin} fournissent des arguments alternatifs pour les références manquantes [A24--A27], établissant ainsi les résultats principaux de~\cite{ArthurBook} et~\cite{Mok} conditionnels seulement à la validité du LFPT.

\section{Caractères infinitésimaux}

Faute de référence, nous rappelons tout d'abord l'expression du caractère infinitésimal d'une représentation en fonction de son paramètre de Langlands. Fixons une clôture algébrique $\C$ de $\R$. Soit $k = \R$ ou $\C$ et $G$ un groupe réductif  connexe défini sur $k$. Si $G$ est un groupe réel ou complexe, soit $\Lie_0 G$ son algèbre de Lie réelle et $\Lie G = \Lie_0 G \otimes_\R \C$. Soit $\cZ(G)$ le centre de l'algèbre enveloppante de $\Lie G$. Dans le cas réel, $\cZ(G) \cong S(\Lie T)^{W}$ où $T$ est un tore maximal réel de $G$ et $W$ le groupe de Weyl complexe. Dans le cas complexe, $T$ est un tore complexe, $\Lie T = \Lie_0 T \times \Lie_0 T$ et $\cZ(G) \cong S(\Lie_0T \times \Lie_0 T)^{W \times W}$, l'ordre sur $\Lie_0 T \times \Lie_0 T$ étant défini par le choix de deux isomorphismes $\sigma, \sigma' \colon k \to \C$. On suppose que $\sigma = \id$ et $\sigma' = c$, la conjugaison complexe.

Soit $\hat G$ le groupe dual (complexe) de $G$; on identifie $\li k$ à $\C$. Soit $(\cB, \cT)$ une paire de Borel dans $\hat G$. On a $X_*T \cong X^* \cT$, donc $\Lie T \cong (\Lie \cT)^*$ et un caractère infinitésimal s'identifie à un élément de $\Lie \cT /W$ (cas réel), de $(\Lie \cT \times \Lie \cT)/W \times W$ (cas complexe).

Soit $W_k$ le groupe de Weil de $k$, $W_k^0 = \C^\times$ sa composante neutre ($= W_k$ dans le cas complexe), $\pi$ une représentation irréductible admissible de $G(k)$ et $\varphi \colon W_k^0 \to \hat G$ la restriction à $W_k^0$ de son paramètre de Langlands. On peut supposer que $\varphi(W_k^0) \subset \cT$. Dans le cas réel, on écrit $\varphi(z) = z^\tau (\li z)^{\tau'}$, $\tau, \tau' \in X_*\cT \otimes \C = \Lie \cT$, $\tau - \tau' \in X_* \cT$.

Dans le cas complexe, on écrit de même
$$
\varphi(z) = z^\tau (\li z)^{\tau'}
$$
avec les mêmes conditions. Noter que la conjugaison complexe opère sur $\hat G$, et à conjugaison près préserve $\cT$; en particulier elle opère sur les exposants $\tau, \tau'$.

\begin{proposition}\label{prop:Clozel1_1}
\begin{enumerate}
\item Si $k = \R$, le caractère infinitésimal de $\pi$ est la $W$-orbite de $\tau$ ou de $c \tau'$.
\item Si $k = \C$, le caractère infinitésimal de $\pi$ est la $W \times W$-orbite de $(\tau, \tau')$.
\end{enumerate}
\end{proposition}

Supposons d'abord $k = \C$. D'après Langlands, $\pi$ est un sous-quotient de l'induite unitaire
$$
\ind_B^G \chi,
$$
o\`{u} $B \subset G$ est un sous-groupe de Borel, $T$ le tore quotient de $B$, et le caractère $\chi$ de $T$ est donné par $t \mapsto t^\tau(t^{\tau'})^-$, $\tau, \tau' \in X_* \cT \otimes \C$. L'induction unitaire étant compatible avec l'isomorphisme d'Harish-Chandra, on en déduit que le caractère infinitésimal est $(\tau, \tau') \in (\Lie \cT)^2$, modulo $W \times W$.

Le cas réel est plus compliqu\'e. Noter que, $\varphi$ s'étendant à $W_\R$, un calcul simple montre que $\tau$ est conjugué par $W$ à $c\tau'$ \footnote{Le fait que le caractère infinitésimal est donné par $\tau$ et non $c \tau$ résulte du choix systématique d'isomorphismes $\C \to \C$.}. Nous allons donner un argument simple de changement de base. Changeant de notation, notons $\ig^0$, $\ig_\C^0$ les algèbres de Lie (réelles) de $G(\R)$, $G(\C)$, et $\ig$, $\ig_{\C}$ leur complexifiées. Il existe un isomorphisme $\C$-linéaire $j \colon \ig_{\C}^0 \to \ig$; alors l'application
\begin{equation}\label{eq:Clozel_1}
\ig_{\C} \to \ig \times \ig,  \quad x \mapsto (jx, cjx)
\end{equation}
définit un isomorphisme. Il en est de même pour un tore réel, d'où un isomorphisme
\begin{equation}\label{eq:Clozel_2}
\cZ(\ig_{\C}) \to \cZ(\ig) \otimes \cZ(\ig).
\end{equation}

Soit $\Pi$ le $L$-paquet de représentations de $G(\R)$ associé à $\varphi$, et $\pi_{\C}$ la représentation de $G(\C)$ associée à $\varphi|_{\C^\times}$. Le paramètre $\varphi$ passe par un sous-groupe $\LM$ de $\LG$ associé à un parabolique réel $P = MN$ de $G$, et définit un $L$-paquet de séries discrètes (modulo le centre) $\pi_M$ pour $M(\R)$~; les éléments de $\Pi$ sont sous-quotients d'induites unitaires des $\pi_M \in \Pi_M$. Par compatibilité avec l'homomorphisme d'Harish-Chandra, il suffit de vérifier le lemme pour $M$.

Le paramètre $\varphi|_{\C^\times}$ définit une représentation $\pi_{M, \C}$ de $M(\C)$. D'après \eqref{eq:Clozel_2} on a un isomorphisme
$$
\cZ(M) \isomto (S(\Lie T) \otimes S(\Lie T))^{W_M \times W_M},
$$
le membre de gauche étant bien sûr défini par le groupe complexe. On en déduit une application
$$
N \colon \cZ(M) \to S(\Lie T), \quad u \otimes v \mapsto uv.
$$

D'après \cite[Proposition 8.6]{ClozelThese}, les caractères infinitésimaux $\omega, \omega_{\C}$ de $\Pi_M$, $\pi_{M, \C}$ sont associés par $\omega_{\C} = \omega \circ N$. En reprenant les isomorphismes ``évidents'' et à l'aide de l'expression \eqref{eq:Clozel_1}, on en déduit enfin que le caractère infinitésimal de $\pi$ est donné par $\tau$ ou $c \tau'$.

Le corollaire suivant ne sera pas utilisé par la suite, mais il est d'un intérêt propre. Il nous a été indiqué par Erez Lapid. Soit $\xi \colon \LG \to \LH$ un homomorphisme de $L$-groupes (sur $k$). Soit $\hat \ig$, $\hat \ih$ les algèbres de Lie des groupes duaux. Dans le cas réel, $\cZ(G) = \C[\Lie \cT]^W$ (polynômes sur $\Lie \cT$) $= \C[\hat \ig]^{\hat G}$ d'après un théorème de Chevalley. De même $\cZ(H) = \C[\hat \ih]^{\hat H}$. On vérifie aisément qu'il y a une application naturelle de restriction $\C[\hat \ih]^{\hat H} \to \C[\hat \ig]^{\hat G}$; par composition on obtient une application $\xi_*$ entre caractères infinitésimaux pour $G$ et $H$.

La même construction, dupliquée, s'applique au cas complexe. On a alors~:

\begin{corollary}
Soit $\xi \colon \LG \to \LH$ un homomorphisme de $L$-groupes sur $k = \R$ ou $\C$. Si $\varphi \colon W_k \to \LG$ définit une représentation $\pi_G$ de $G(k)$ et si $\varphi_H = \xi \circ \varphi$ définit une représentation $\pi_H$, les caractères infinitésimaux, $\omega_G$, $\omega_H$ de $\pi_G$, $\pi_H$ sont associés par $\omega_H = \xi_* \omega_G$.
\end{corollary}

Ceci se déduit aussitôt des identifications précédentes.

\section{Induites paraboliques cohomologiques}

Soit $F$ un corps de nombres, $\li F$ une clôture algébrique de $F$, $\Gamma_F = \Gal(\li F/F)$. Pour un groupe $H$ sur $F$ on note $X^*(H)$ le $\Gamma$-module des homomorphismes $H \to \Gm/\li F$. Soit $G$ un groupe semi-simple connexe sur $F$, $M$ un sous-groupe de Levi de $G$ (sur $F$), $A_M$ son tore déployé central maximal. Donc $A_M \cong \Gm^r /F$, $r$ étant le rang parabolique.

On identifie $X^*(M)^{\Gamma_F} \otimes_\Z \C$ à un groupe de caractères de $M(\A_F)$ par
$$
(\lambda \otimes s)(m) = |\lambda(m)|_F^s.
$$
Soit $A_M(\R) \subset A_M(F_\infty)^\circ$ le sous-groupe donné, sur chaque facteur $\Gm(F_\infty)$, par le plongement diagonal $\R^\times_+ \subset F_\infty^\times$. 
On consid\`ere ici $G$ comme un $\Q$-groupe par restriction des scalaires.
Il est isomorphe à $\R^r$. Le morphisme de restriction $X^*(M)^{\Gamma_F} \otimes \C \to \Hom_\ct(A_M(\R)^\circ, \C^\times)$ est un isomorphisme.

Soit $\pi_M$ une représentation automorphe cuspidale de \(M(\A_{\ff})\).
Elle s'écrit de façon unique $\pi_M \cong \pi_{M, 0} \otimes \chi$, où $\chi \in X^*(M)^{\Gamma_F} \otimes \C$ et $\pi_{M, 0}$, de caractère central trivial sur $A_M(\R)^\circ$, est cuspidale unitaire.

\begin{lemma} \label{lem:chi_reel}
  Supposons que pour toute place archimédienne \(v\) de \(F\) le caractère infinitésimal de \(\pi_{M,v}\) soit réel.
  Alors \(\chi\) est réel, i.e.\ \(\chi \in X^*(M)^{\Gamma_F} \otimes \R\).
\end{lemma}
\begin{proof}
  L'hypothèse implique que la restriction du caractère central de \(\pi_{M,v}\) à \(A_M(\R)^0 \subset S_M(F_v)\) est réelle.
  L'image de \(\chi\) par l'isomorphisme de restriction \(X^*(M)^{\Gamma_F} \otimes \C \to \Hom_\ct(A_M(\R)^\circ, \C^\times)\) est égale au produit (sur toutes les places archimédiennes \(v\)) de ces restrictions, donc elle est également réelle.
\end{proof}

On suppose maintenant que pour toute place archimédienne \(v\) de \(F\) une induite parabolique de $M$ à $G$ (normalisée, pour un des sous-groupes paraboliques de \(G\) de facteur de Lévi \(M\)) de $\pi_{M,v}$ admet un sous-quotient irréductible $\pi_v$ qui est cohomologique en $v$, au sens suivant~:
\begin{itemize}
\item[$\bullet$] Si $v$ est réelle, on suppose qu'il existe une représentation algébrique irréductible $V^\lambda$ de $G(F_v \otimes \C)$ de plus haut poids $\lambda$ telle que $\uH^\bullet(\ig(F_v) \otimes \C, K_v; \pi_v \otimes V^\lambda) \neq 0$, où \(K_v\) est un sous-groupe compact connexe maximal de \(G(F_v)\).
\item[$\bullet$] Si $v$ est complexe, on fait la même hypothèse pour une représentation algébrique de $G(F_v \otimes \C)$ de la forme $V^{\lambda_1} \otimes V^{\lambda_c}$ où $(1, c)$ désignent les deux isomorphismes $\C \to \C$.
\end{itemize}
On va déduire de cette hypothèse une contrainte sur $\chi$. On introduit $\hat G$ et $\LG$ (sur $F$ ou $F_v$), et on utilise la description précédente. L'hypothèse cohomologique implique alors~:
\begin{itemize}
\item[$\bullet$] Si $v$ est réelle, le caractère infinitésimal $\tau$ de $\pi_v$ est l'orbite sous $W$ de $-(\lambda + \rho) \in \Lie \cT$.
\item[$\bullet$] Si $v$ est complexe, il est égal à l'orbite sous $W \times W$ de $(-\lambda_1 + \rho, -\lambda_c + \rho) \in (\Lie \cT)^2$.
\end{itemize}
Par compatibilité du caractère infinitésimal avec l'induction parabolique, la $W$-orbite (resp. $W \times W$-orbite) de l'image du caractère infinitésimal de $\pi_{M,v}$ est égale à cette expression.
Notons $C$ le tore $M/M_\der$, de sorte que \(X^*(C) = X^*(M)\) et son dual $\hat C$ s'identifie à $Z(\hat M)^0$.
On a donc
\[ \chi \in X^*(C)^{\Gamma_F} \otimes \C \subset X^*(C) \otimes \C = X_*(\hat C) \otimes \C, \]
et $\hat C \subset \cT$ ce qui nous permet de voir \(\chi\) comme un élément de \(X_*(\cT) \otimes \C\) invariant par \(W_M\).
Pour plus de clarté, notons, dans le cas complexe, $\tau_1, \tau_c$ les exposants $\tau, \tau' $ de la Proposition 1.1.
Alors
\begin{itemize}
\item[$\bullet$] Si $v$ est réelle, le caractère infinitésimal de $\pi_{M, v}$ est
\begin{equation}\label{eq:Clozel1}
\tau = \chi + \tau(\pi_{M, 0, v})\ (\tu{mod } W_M).
\end{equation}
\item[$\bullet$] Si $v$ est complexe, c'est, modulo $W_M \times W_M$~:
\begin{equation}\label{eq:Clozel2}
(\tau_1, \tau_c) = (\chi + \tau_1(\pi_{M, 0, v}), \chi + \tau_c(\pi_{M,0, v})).
\end{equation}
\end{itemize}
Utilisons maintenant l'unitarité de $\pi_{M, 0}$.

\begin{lemma}\label{lem:Clozel1_3}
Soit $k = \R$ ou $\C$ et $\pi$ une représentation irréductible unitaire de $M(k)$.
\begin{enumerate}[label=(\roman*)]
\item Si $k = \R$, le caractère infinitésimal $\tau$ de $\pi$ vérifie $\tau = -\li{c(\tau)} (\tu{mod } W_M)$.
\item Si $k = \C$, le caractère infinitésimal $(\tau_1, \tau_c)$ de $\pi$ vérifie $(\tau_1, \tau_c) = (-\li \tau_c, - \li \tau_1) (\tu{mod } W_M^2)$.
\end{enumerate}
\end{lemma}

Supposons $k = \C$, et que $\pi$ est un sous-quotient de $\ind_B^M(\chi)$, $\chi(t) = t^{\tau_1} (t^{\tau_c})^-$. On a $\pi \cong \bar {\tilde \pi}$, la conjugaison complexe agissant sur l'espace $V$ de $\pi$. Or $\tilde \pi$ est un sous-quotient de $\ind_B^M(\chi\inv)$, et $\bar{\tilde \pi}$ un sous-quotient de $\ind_B^M(\li \chi\inv)$~; le résultat s'en déduit.

Dans le cas réel, rappelons que la conjugaison complexe $c$ opère sur $X_* \cT$~; on note $\bar{\phantom{x}}$ l'action de la conjugaison complexe sur le second facteur de $X_*(\cT) \otimes \C$.

Soit $H/\R$ un groupe semi-simple ayant une série discrète, et $\pi_H$ appartenant à celle-ci. On rappelle que l'action de $c$ définissant le $L$-groupe fixe un épinglage. Alors $c$ opère par $(-w_0)$ sur $X_*(\cT_H)$ (notation évidente), où $w_0$ est l'élément de plus grande longueur du groupe de Weyl, et le caractère infinitésimal est réel; l'assertion (i) s'en déduit. Si $H$ est un tore réel, l'assertion se déduit de la dualité pour les tores (Langlands~\cite[Lemma 2.8]{LanglandsReal}).

Soit $H$ ayant une série discrète modulo le centre, et $Z^0$ le centre connexe de $H$. On a un morphisme $Z^0 \times H_\der \to H$~; $\pi_H$ (unitaire) se restreint en une somme finie des représentations de $Z^0(\R) \times H_\der(\R)$, et ceci est compatible avec la functorialit\'e (cf. Langlands~\cite[\S 3]{LanglandsReal}). Or avec des notations évidentes, $\cT_H \otimes \C$ est égal à ($\cT_{Z^0} \otimes \C) \times (\cT_{H_\der} \otimes \C)$~; le résultat pour $H$ (et les représentations essentiellement de carré intégrable) s'en déduit. Enfin, on vérifie que le Lemme~\ref{lem:Clozel1_3}(i) est compatible avec l'induction parabolique.

D'après~\eqref{eq:Clozel1} et~\eqref{eq:Clozel2} et les Lemmes~\ref{lem:chi_reel} et~\ref{lem:Clozel1_3}, on obtient alors
\begin{itemize}
\item[$\bullet$] $- \overline{ c(\tau(\pi_{M, v}))} = \tau(\pi_{M, v}) - 2\chi$ si \(F_v \simeq \R\),
\item[$\bullet$] $- \overline{\tau_c(\pi_{M, v})} = \tau_1(\pi_{M, v}) - 2 \chi$ si \(F_v \simeq \C\).
\end{itemize}
Puisque les exposants $\tau$ intervenant dans ces formules sont dans $\rho + X_*(\cT)$, on en déduit enfin~:
\begin{equation*}
2 \chi \in X_*(\cT) \cap (X_*(\hat C)^{\Gamma_F} \otimes \C) \subset X_*(\cT) \cap X_*(\hat C) \otimes \C
\end{equation*}
Or $X_*(\hat C) \subset X_*(\cT)$ est facteur direct (puisque c'est l'orthogonal des racines de $\cT$ dans $\hat M$), d'où enfin
\begin{equation}
2 \chi \in X_*(\hat C)^{\Gamma_F} = X^*(M)^{\Gamma_F}.
\end{equation}
Ainsi~:

\begin{theorem}\label{thm:Clozel2_3}
  Soit $\pi = \pi_0 \otimes \chi$ une représentation cuspidale de $M(\A_F)$, $\pi_0$ étant de caractère central trivial sur $A_M(\R)^\circ$, et $\chi \in X^*(M)^{\Gamma_F} \otimes \C$.
  Supposons que pour toute place archimédienne \(v\) de \(F\) il existe un sous-quotient irréductible de l'induite normalisée $\ind_{P(F_v)}^{G(F_v)} \pi_v$ soit cohomologique.
  Alors $\chi \in \tfrac 12 X^*(M)^{\Gamma_F}$.
\end{theorem}

Dans le cas où $G$ est un groupe semi-simple classique déployé, \textit{i.e.} $\Sp_{2n}$, $\SO_{2n}$ ou $\SO_{2n+1}$ sur $F$, de sorte que $M = \GL_{n_1} \times \ldots \times \GL{n_r} \times G'$ où $G'$ est semi-simple du même type que $G$, on peut écrire $\chi = s = (s_1, \ldots, s_r)$, $\pi = \pi_1 \otimes \ldots \otimes \pi_r \otimes \pi'$, $\pi \otimes \chi = \pi_{1, 0} |\det|^{s_1} \otimes \ldots \otimes \pi_{r, 0} |\det|^{s_r}$ et le Théorème affirme que $s_i \in \tfrac 12 \Z$~: en particulier ceci généralise le Théorème 4.1 de Grbac et Schwermer (Theorem 3.2 dans~\cite{GrbacEisensteinCohAutLfunc}).

\begin{remark}
Cette démonstration semble être essentiellement celle suggérée par Harder (avec des hypothèses plus restrictives) dans~\cite[\S 4.4]{GrbacSchwermerForum}.
\end{remark}

\section{Cohomologie pondérée et \((\gfr,K)\)-cohomologie du spectre automorphe discret}

On aura besoin de la cohomologie pondérée de Goresky, Harder et MacPherson \cite{GoreskyHarderMacPherson}, rappelons brièvement leur construction.
Soit \(G\) un groupe réductif connexe sur \(\Q\), \(A_G\) son centre déployé, et \(K\) un sous-groupe ouvert d'un sous-groupe compact maximal de \(G(\R)\) (autrement dit \(K_\mathrm{max}^0 \subset K \subset K_\mathrm{max}\) où \(K_\mathrm{max}\) est un sous-groupe compact maximal de \(G(\R)\), uniquement déterminé par \(K\)).
On considère l'espace symétrique \(\cX := G(\R)/K A_G(\R)^0\).
On s'intéresse aux variantes adéliques des espaces localement symétriques, les variétés différentielles
\[ X_{K_{\ff}} := G(\Q) \backslash \left( \cX \times G(\A_{\ff})/K_{\ff} \right) \]
pour \(K_{\ff}\) un sous-groupe compact ouvert net de \(G(\A_{\ff})\).
On a des compactifications de Borel-Serre réductives (\S 6 \loccit)
\[ j: X_{K_{\ff}} \hookrightarrow \ol{X}_{K_{\ff}} \]
vivant sous les compactifications de Borel-Serre \cite{BorelSerre_corners}.

Soient \(Q\) un parabolique maximal de \(G\), \(N_Q\) son radical unipotent, \(M_Q := Q/N_Q\) son quotient réductif et \(A_Q\) le centre déployé de \(M_Q\).
Le tore \(A_Q/A_G\) est de dimension \(1\) qui a une unique identification avec \(\GL_1\) telle que les caractères de \(A_Q\) apparaissant dans \(\Lie N_Q\) soient strictement positifs.
Fixons un tore déployé maximal de \(G\) et un ordre sur les racines correspondantes (cela revient à choisir un parabolique minimal \(Q_0\) de \(G\) et un relèvement de \(A_{Q_0}\) le long de \(Q_0 \to M_{Q_0}\)), et notons \(\Delta\) l'ensemble des racines simples.
Pour \(\alpha \in \Delta\) on a un sous-groupe parabolique standard propre maximal \(Q_\alpha\) de \(G\).
Soit \(2 \rho\) la somme des racines positives (pour \(G\)).
Pour \(Q\) un sous-groupe parabolique de \(G\) on note \(\rho_Q\) la restriction de \(\rho\) à \(A_Q/A_G\), vue comme un élément de \(\frac{1}{2} \Z\) via l'identification ci-dessus.
Un profil de poids est une fonction \(\Delta \to \frac{1}{2} + \Z\).
On s'intéresse particulièrement aux profils de poids ``milieu'' \(p_+\) et \(p_-\):
\[ p_+(\alpha) = -\rho_{Q_\alpha} \ \text{ou} \ -\rho_{Q_\alpha} + 1/2 \]
\[ p_-(\alpha) = -\rho_{Q_\alpha} \ \text{ou} \ -\rho_{Q_\alpha} - 1/2 \]
(choisir l'option qui est dans \(\frac{1}{2}+\Z\)).
On s'intéresse aussi au profil de poids ``\(-\infty\)'': \(p(\alpha)\) très négatif (peu importe la valeur précise).

Nous allons considérer des systèmes locaux sur \(X_{K_{\ff}}\) associés à des représentations algébriques (définies sur \(\Q\)) de \(G\).
On peut se ramener au cas d'une représentation irréductible.
L'ensemble des classes d'isomorphismes de représentations irréductibles de \(G\) est décrit par \cite[Théorème 7.2]{Tits_repr_corps_qqc}.
Pour \(G\) quasi-déployé, ce qui sera toujours le cas pour nous, cette description est un peu plus simple grâce au Théorème 3.3 \loccit: si \(\lambda\) est un poids dominant pour \(G_{\Qbar}\), de stabilisateur \(\Gamma_E\) dans \(\Gamma_{\Q}\), il existe une représentation absolument irréductible \(V_\lambda\) de \(G_E\) (définie sur \(E\)) de plus haut poids \(\lambda\), unique à isomorphisme près (l'algèbre des endomorphismes de \(V_\lambda\) étant simplement \(E\)) et on en déduit par restriction des scalaires une représentation de \(G\) sur \(V_\lambda\) vu comme \(\Q\)-espace vectoriel, d'algèbre d'endomorphismes \(E\).
Cette représentation est irréductible et on a \(V_\lambda \simeq V_{\lambda'}\) (comme représentations de \(G\) sur \(\Q\)) si et seulement si \(\lambda\) et \(\lambda'\) sont dans la même orbite sous \(\Gamma_{\Q}\).
Toute représentation irréductible de \(G\) est isomorphe à une telle représentation \(V_\lambda\).
Énonçons une conséquence directe de cette classification sous la forme qui nous sera utile dans la suite.

\begin{proposition} \label{pro:irr_alg_rep_qs}
  Soit \(G\) un groupe réductif connexe quasi-déployé sur \(\Q\).
  Soit \(W\) une représentation irréductible de \(G_\C\).
  Il existe une représentation irréductible \(V\) de \(G\) (sur \(\Q\)), d'algèbre d'endomorphismes une extension finie (de corps commutatifs) \(E/\Q\), et un plongement \(\iota_0: E \hookrightarrow \C\) tels que \(W\) soit isomorphe à \(\C \otimes_{\iota_0,E} V\).
  La paire \((V,\iota_0)\) est unique à isomorphisme près.
\end{proposition}

Soit \(V\) une représentation algébrique (de dimension finie, définie sur \(\Q\)) irréductible de \(G\).
À \(V\) est associé un système local \(\ul{V}\) sur \(X_{K_{\ff}}\).
Goresky, Harder et MacPherson définissent (27.9 \loccit) un complexe \(I^\bullet V\) quasi-isomorphe à \(\ul{V}\).
À un profil de poids \(p\) ils associent un sous-complexe \(W^p I^\bullet V\) de \(j_* I^\bullet V\).
La fonction \(p \mapsto W^p I^\bullet V\) est décroissante, en particulier
\[ W^{p_+} I^\bullet V \,\subseteq\, W^{p_-} I^\bullet V \,\subseteq\, j_* I^\bullet V. \]
Les correspondances de Hecke s'étendent aux compactifications \(\ol{X}_{K_{\ff}}\) (voir \cite[\S 7.7]{GKM_DS_char_Lef_formula}), en particulier on a des \(\Q[G(\A_{\ff})]\)-modules admissibles
\[ W^p H^i(G, \cX; V) := \varinjlim_{K_{\ff}} H^i(\ol{X}_{K_{\ff}}, W^p I^\bullet V) \]
Dans le cas \(p = -\infty\) on a un morphisme naturel \(j_* I^\bullet V \to Rj_* \ul{V}\), qui est un quasi-isomorphisme d'après le Corollaire 32.10 \loccit, donc \(W^{-\infty} H^\bullet(G, \cX; V)\) est la cohomologie ordinaire.
Notons \(W^\pm H^i(G, \cX; V)\) l'image de
\[ W^{p_+} H^i(G, \cX; V) \longrightarrow W^{p_-} H^i(G, \cX; V), \]
c'est donc un \(\Q[G(\A_{\ff})]\)-module admissible et on a un morphisme naturel \(G(\A_{\ff})\)-équivariant
\[ W^\pm H^i(G, \cX; V) \to H^i(G, \cX; V). \]

Notons \(V_{\C} = \C \otimes_{\Q} V\).
Concrètement, pour \(G\) quasi-déployé si (la classe d'isomorphisme de) \(V\) correspond à une \(\Gamma_\Q\)-orbite \(o\) de poids dominants alors \(\Qbar \otimes_{\Q} V \simeq \bigoplus_{\lambda \in o} V_\lambda\).
Le théorème de de Rham se traduit par un isomorphisme de \(\C[G(\A_{\ff})]\)-modules
\begin{equation} \label{eq:deRham}
  \C \otimes_{\Q} H^i(G, \cX; V) \simeq H^i(\gfr,K; C^\infty(G(\Q) \backslash G(\A)) \otimes_{\C} V_{\C})
\end{equation}
où \(C^\infty(G(\Q) \backslash G(\A)) := \varinjlim_{K_{\ff}} C^\infty(G(\Q) \backslash G(\A) / K_{\ff})\).
Franke a démontré la conjecture de Borel \cite[Theorem 18]{Franke_weighted} affirmant que l'inclusion de l'espace des formes automorphes \(\cA(G) = \cA(G(\Q) \backslash G(\A))\) pour \(G\) dans \(C^\infty(G(\Q) \backslash G(\A))\) induit un isomorphisme en \((\gfr,K)\)-cohomologie, et on a donc un isomorphisme \(G(\A_{\ff})\)-équivariant
\begin{equation} \label{eq:deRham_plus_Franke}
  \C \otimes_{\Q} H^i(G, \cX; V) \simeq H^i(\gfr,K; \cA(G) \otimes_{\C} V_{\C}).
\end{equation}
La formulation de Franke est légèrement différente, expliquons la différence.
On a un morphisme continu surjectif
\[ m_G: G(\A) \to \afr_{G,0} := \Lie A_G(\R) = \R \otimes_{\Z} X_*(A_G) \]
défini par la propriété \(|\chi(g)| = \langle m_G(g), \chi \rangle\) pour tout \(\chi \in X^*(G)^{\Gamma_{\Q}}\).
On a une décomposition \cite[\S I.3.2]{MoeglinWaldspurger_bookspec}
\[ \cA(G) = \bigoplus_{\chi \in X^*(G)^{\Gamma_\Q} \otimes \C} \Sym^\bullet \afr_G^* \otimes_{\C} \cA(G, \chi) \]
où
\[ \cA(G,\chi) := \cA(G(\Q) \backslash G(\A) /A_G(\R)^0)(\chi) := \{ \chi f \,|\, f \in \cA(G(\Q) \backslash G(\A) /A_G(\R)^0) \} \]
et \(\Sym^\bullet \afr_G^*\) n'est rien d'autre que l'espace des fonctions \(G(\A) \to \C\) de la forme \(P \circ m_P\) où \(P\) est une fonction polynômiale (complexe) sur \(\afr_{G,0}\).
On a un isomorphisme (voir bas de la p.256 dans \cite{Franke_weighted})
\[ H^i(\gfr,K; \cA(G) \otimes_{\C} V_{\C}) \simeq \bigoplus_{\chi} H^i \left( \gfr/\afr_G,K;  (\cA(G, \chi) \otimes V_{\C})^{A_G(\R)^0} \right) \]
et si on décompose \(V_{\C}\) en somme de représentations irréductibles \(V_\lambda\), le terme correspondant à \(\lambda\) est nul sauf pour \(\chi^{-1} = \lambda|_{A_G(\R)^0}\).

Nair et Rai \cite[Appendix A]{NairRai} ont démontré un analogue de la conjecture de Borel pour la cohomologie pondérée et les formes automorphes de carré intégrable, que nous rappelons maintenant.
Soit \(\cA^2(G)\) l'espace des formes automorphes pour \(G\) dont la restriction à \(G(\Q) \backslash G(\A)^1\) est de carré intégrable, qui se décrit donc aussi comme
\[ \cA^2(G) = \bigoplus_{\chi \in X^*(G)^{\Gamma_\Q} \otimes \C} \Sym^\bullet \afr_G^* \otimes_{\C} \cA^2(G,\chi) \]
où
\[ \cA^2(G,\chi) := \{ \chi f \,|\, f \in \cA^2(G(\Q) \backslash G(\A) /A_G(\R)^0) \}. \]

Nair et Rai ont démontré le théorème suivant (\cite[Corollary A]{Nair_weighted} et \cite[Proposition A.3]{NairRai} dans le cas d'un groupe \(G\) semi-simple, la même preuve fonctionne dans le cas général).
\begin{theorem} \label{thm:Nair_Rai}
  On a un isomorphisme de \(\C[G(\A_{\ff})]\)-modules
  \begin{equation} \label{eq:Nair_Rai}
    \C \otimes_{\Q} W^\pm H^i(G, \cX; V) \simeq H^i(\gfr,K; \cA^2(G) \otimes_{\C} V_{\C})
  \end{equation}
  où la représentation algébrique \(V_{\C}\) de \(G_{\C}\) est vue comme un \((\gfr,K)\)-module (de dimension finie).
  En particulier le membre de droite est muni d'une structure rationnelle.
  Quand \(V\) varie ces isomorphismes sont fonctoriels en \(V\).
\end{theorem}
Ce résultat de Nair et Rai utilise également le travail de Franke, et on a un diagramme commutatif de représentations admissibles de \(G(\A_{\ff})\)
\[
  \begin{tikzcd}
    \C \otimes_{\Q} W^\pm H^i(G, \cX; V) \ar[d] \ar[r, "{\sim}"] & H^i(\gfr,K; \cA^2(G) \otimes_{\C} V_{\C}) \ar[d] \\
    \C \otimes_{\Q} H^i(G,\cX; V) \ar[r, "{\sim}"] & H^i(\gfr,K; \cA(G) \otimes_{\C} V_{\C})
  \end{tikzcd}
\]

Rappelons maintenant les morphismes de restrictions aux strates de bord et leur interprétation en termes de formes automorphes.
À chaque sous-groupe parabolique \(Q\) de \(G\) (défini sur \(\Q\)) est associée une composante de bord \(e_Q := \cX / A_Q(\R)^0\) (action géodésique), et Borel et Serre définissent une topologie sur \(\tilde{\cX} := \bigsqcup_Q e_Q\), et une action à gauche de \(G(\Q)\) sur \(\tilde{\cX}\) étendant l'action naturelle de \(G(\Q)\) sur \(\cX\).
Cette action satisfait \(\gamma \cdot e_Q = e_{\gamma Q \gamma^{-1}}\).
On a une stratification (union disjointe sur les classes de \(G(\Q)\)-conjugaison de sous-groupes paraboliques de \(G\))
\[ G(\Q) \backslash (\tilde{\cX} \times G(\A_{\ff})/K_{\ff}) = \bigsqcup_{[Q]} Q(\Q) \backslash (e_Q \times G(\A_{\ff})/K_{\ff}). \]
Goresky, Harder et MacPherson \cite[\S 6]{GoreskyHarderMacPherson} considèrent les quotients
\[ \cX_Q := N_Q(\R) \backslash e_Q \]
et le quotient (topologique) \(\ol{\cX}\) de \(\cX^*\) obtenu en contractant les fibres de chaque \(e_Q \to \cX_Q\).
La paire \((M_Q,\cX_Q)\) satisfait les mêmes conditions que celles qu'on a mises sur \((G,\cX)\) (en notant \(K_x\) le stabilisateur de \(x \in \cX\), le stabilisateur de son image dans \(\cX_Q\) dans \(M_Q(\R)/A_Q(\R)^0\) est l'image de \(K_x \cap Q(\R)\)).
La compactification de Borel-Serre réductive \(\ol{X}_{K_{\ff}} = G(\Q) \backslash (\ol{\cX} \times G(\A_{\ff})/K_{\ff})\) a donc une stratification
\[ \ol{X}_{K_{\ff}} = \bigsqcup_{[Q]} Q(\Q) \backslash (\cX_Q \times G(\A_{\ff})/K_{\ff}). \]
Grâce au fait que \(N_Q(\Q)\) est dense dans \(N_Q(\A_{\ff})\) on a, pour tout \(g \in G(\A_{\ff})\) et en notant \(K(gK_{\ff},Q)\) l'image de \(Q(\A_{\ff}) \cap gK_{\ff}g^{-1}\) dans \(M_Q(\A_{\ff})\), une bijection
\begin{align*}
  N_Q(\Q) \backslash Q(\A_{\ff}) g K_{\ff} / K_{\ff} & \longrightarrow M_Q(\A_{\ff})/K(gK_{\ff},Q) \\
  [qg] & \longmapsto [\ol{q}].
\end{align*}
On obtient pour \(S_{Q,K_{\ff}} := Q(\Q) \backslash (\cX_Q \times G(\A_{\ff})/K_{\ff})\)
\begin{align*}
  S_{Q,K_{\ff}}
  &\simeq \bigsqcup_{[g] \in Q(\A_{\ff}) \backslash G(\A_{\ff})/ K_{\ff}} M_Q(\Q) \backslash (\cX_Q \times M_Q(\A_{\ff}) / K(gK_{\ff},Q)) \\
  &\simeq \bigsqcup_{[g] \in Q(\A_{\ff}) \backslash G(\A_{\ff})/ K_{\ff}} X_{Q,K(gK_{\ff},Q)}
\end{align*}
On considère le morphisme de restrictions aux strates correspondant à \(Q\)
\[ H^i(G, \cX; V) = \varinjlim_{K_{\ff}} H^i(\ol{X}_{K_{\ff}}, Rj_* \ul{V}) \xrightarrow{\mathrm{res}} \varinjlim_{K_{\ff}} H^i(S_{Q,K_{\ff}}, i^* Rj_* \ul{V}) \]
où on note \(i: S_{Q,K_{\ff}} \hookrightarrow \ol{X}_{K_{\ff}}\).
Choisissons un facteur de Lévi de \(Q\), autrement dit une section \(s\) de \(Q \to M_Q\).
Soit \(C^\bullet(\Lie N_Q,V) = \Hom(\wedge^\bullet \Lie N_Q, V)\) le complexe de Chevalley-Eilenberg calculant la cohomologie de l'algèbre de Lie rationnelle \(\Lie N_Q\), qui est un complexe de représentations algébriques de dimensions finies de \(M_Q\) grâce à la section \(s\).
(On peut vérifier qu'un autre choix de section donne un complexe homotope.)
\begin{lemma}
  Le complexe \(i^* Rj_* \ul{V}\) est quasi-isomorphe à \(\ul{C^\bullet(\Lie N_Q,V)}\).
\end{lemma}
\begin{proof}
  Notons \(\cX^*_{K_{\ff}} = G(\Q) \backslash (\cX^* \times G(\A_{\ff})/K_{\ff})\) et \(J: \cX_{K_{\ff}} \hookrightarrow \cX^*_{K_{\ff}}\) l'immersion ouverte.
  On a un diagramme cartésien
  \[
    \begin{tikzcd}
      Q(\Q) \backslash (e_Q \times G(\A_{\ff})/K_{\ff}) \ar[r, hook, "{I}"] \ar[d, "{p}"] & \cX^*_{K_{\ff}} \ar[d] \\
      S_{Q,K_{\ff}} \ar[r, hook, "{i}"] & \ol{\cX}_{K_{\ff}}
    \end{tikzcd}
  \]
  Le théorème du changement de base propre identifie \(i^* Rj_* \ul{V}\) à \(R p_* I^* RJ_* \ul{V}\).
  On voit à l'aide d'un voisinage géodésique de \(e_Q\) que \(I^* RJ_* \ul{V}\) s'identifie au système local \(\ul{\Res^G_Q V}\) sur \(Q(\Q) \backslash (e_Q \times G(\A_{\ff})/K_{\ff})\).
  Soit \(\Omega_{\mathrm{dR}, \mathrm{alg}}^\bullet(N_Q)\) le complexe des sections globales du complexe de de Rham algébrique sur \(N_Q\).
  Comme variété \(N_Q\) est isomorphe à \(\A_\Q^d\), donc ce complexe est exact en degré \(>0\).
  Chaque \(\Omega_{\mathrm{dR}, \mathrm{alg}}^k(N_Q)\) est une représentation algébrique de \(M_Q\) grâce à la section \(s\) qui permet de faire agir \(Q\) sur \(N_Q\): pour \(q \in Q\) d'image \(\ol{q}\) dans \(M_Q\) et \(n \in N_Q\) on pose \(q \cdot n = q n s(\ol{q})^{-1}\).
  On dispose ainsi d'une résolution \(\Res^G_Q V \to \Omega_{\mathrm{dR}, \mathrm{alg}}^\bullet(N_Q) \otimes_\Q \Res^G_Q V\).
  L'application \(p\) est propre et ses fibres s'identifient à \(U \backslash N_Q(\R)\) pour des sous-groupes arithmétiques \(U\) de \(N_Q(\Q)\).
  Comme \(N_Q(\R)\) est contractile pour \(x \in S_{Q,K_{\ff}}\) on a une identification
  \[ R^i p_* \ul{\Omega_{\mathrm{dR}, \mathrm{alg}}^k(N_Q) \otimes_\Q \Res^G_Q V} \simeq H^i(U, \Omega_{\mathrm{dR}, \mathrm{alg}}^k(N_Q) \otimes_\Q \Res^G_Q V) \]
  et pour \(i>0\) le terme de droite est nul d'après \cite[Proposition 24.3 (b)]{GoreskyHarderMacPherson}, car en notant \(\operatorname{Spec} N_Q = \cO(N_Q)\) le \(\cO(N_Q)\)-module \(\Omega_{\mathrm{dR}, \mathrm{alg}}^k(N_Q)\) est libre de rang fini.
  On en déduit
  \[ R p_* \ul{\Res^G_Q V} \simeq p_* \ul{\Omega_{\mathrm{dR}, \mathrm{alg}}^\bullet(N_Q) \otimes_\Q \Res^G_Q V} = \ul{\left( \Omega_{\mathrm{dR}, \mathrm{alg}}^\bullet(N_Q) \otimes_\Q \Res^G_Q V \right)^{N_Q}} \]
  et on conclut grâce à l'identification classique
  \[ \left( \Omega_{\mathrm{dR}, \mathrm{alg}}^\bullet(N_Q) \otimes_\Q \Res^G_Q V \right)^{N_Q} \simeq C^\bullet(\Lie N_Q, V). \]
\end{proof}

La cohomologie de \(C^\bullet(\Lie N_Q,V)\) a été calculée (comme représentation de \(M_Q\)) par Kostant \cite{Kostant_Liecoh}.
Comme la catégorie des représentations algébriques d'un groupe réductif (sur un corps de caractéristique zéro) est semi-simple, ce complexe est quasi-isomorphe à une somme directe de complexes concentrés en un degré.
La preuve de Kostant montre en fait (en utilisant l'hypothèse que \(V\) est irréductible) que chaque composante \(M_Q\)-isotypique de \(C^\bullet(\Lie N_Q,V)\) contribuant à la cohomologie est concentrée en un degré, fournissant un quasi-isomorphisme canonique entre \(C^\bullet(\Lie N_Q,V)\) et une somme de complexes concentrés en un degré.
Explicitant l'action des opérateurs de Hecke on a enfin une identification
\[ \varinjlim_{K_{\ff}} H^i(S_{Q,K_{\ff}}, i^* Rj_* V) \simeq \Ind_{Q(\A_{\ff})}^{G(\A_{\ff})} \left( \bigoplus_{a+b=i} H^a(M_Q, \cX_Q; H^b(\Lie N_Q, V)) \right) \]
où \(\Ind_{Q(\A_{\ff})}^{G(\A_{\ff})}\) désigne l'induction parabolique non normalisée.

Ces identifications ont des analogues du côté de la cohomologie de de Rham via les morphismes ``terme constant''
\begin{align*}
  C_Q^G: \cA(G) & \longrightarrow \cA(Q(\Q) N_Q(\A) \backslash G(\A)) \\
  h & \longmapsto \left(C_Q^G h: g \mapsto \int_{N_Q(\Q) \backslash N_Q(\A)} \varphi(ng) dn \right).
\end{align*}
Soit \(K_Q = K \cap Q(\R)\) (on l'identifiera aussi à son image dans \(M_Q(\R)\)).
L'espace d'arrivée s'identifie à \(\Ind_{Q(\A)}^{G(\A)} \cA(M_Q)\) où l'induction est l'induction parabolique lisse aux places finies et l'induction parabolique (toujours non normalisée) \(\Ind_{\qfr,K_Q}^{\gfr,K}\) pour les \((\gfr,K)\)-modules à la place réelle (rappelée en \cite[p.208]{Franke_weighted}).
Cette dernière est l'adjoint à droite du foncteur de restriction (des \((\gfr,K)\)-modules vers les \((\qfr,K_Q)\)-modules), d'où une identification\footnote{Cette identification provient plus concrètement d'un isomorphisme de complexes de Chevalley-Eilenberg.}
\[ H^i \left( \gfr,K; \Ind_{P(\A)}^{G(\A)} \cA(M_Q) \otimes_{\C} V_{\C} \right) \simeq \Ind_{P(\A_{\ff})}^{G(\A_{\ff})} H^i(\qfr,K_Q; \cA(M_Q) \otimes_{\C} V_{\C}). \]
La sous-algèbre de Lie \(\nfr_Q\) de \(\qfr\) agit trivialement sur \(\cA(M_Q)\) et on a un isomorphisme de complexes de Chevalley-Eilenberg
\[ C^\bullet(\qfr,K_Q; \cA(M_Q) \otimes_{\C} V_{\C}) \simeq \mathrm{Tot}^\bullet \left( C^\bullet(\mfr_Q,K_Q; \cA(M_Q) \otimes_{\C} C^\bullet(\nfr_Q; V_{\C})) \right) \]
où à droite \(\mathrm{Tot}^\bullet\) désigne le complexe total associé à un complexe double.
La cohomologie de \(C^\bullet(\nfr_Q; V_{\C})\), comme représentation algébrique de \(M_{Q,\R}\)\footnote{identifié au facteur de Lévi \(Q_{\R} \cap \theta_K(Q_{\R})\) de \(Q_{\R}\), où \(\theta_K\) est l'involution de Cartan de \(G_{\R}\) pour laquelle \(K \subset G(\R)^{\theta_K}\)} a été calculée par Kostant comme rappelé plus haut, et comme dans le cas rationnel on en déduit la simplification
\[ H^i \left( \gfr,K; \Ind_{P(\A)}^{G(\A)} \cA(M_Q) \otimes_{\C} V_{\C} \right) \simeq \bigoplus_{a+b=i} \Ind_{P(\A_{\ff})}^{G(\A_{\ff})} H^a \left( \mfr_Q,K_Q; \cA(M_Q) \otimes_{\C} H^b(\nfr_Q; V_{\C}) \right). \]

Ces deux constructions sont compatibles: le diagramme suivant est commutatif.
\begin{equation} \label{eq:diag_BSres_cstterm}
  \begin{tikzcd}
    \C \otimes_{\Q} H^i(G,\cX; V) \ar[d, dash, "{\sim}"] \ar[r] & \C \otimes_{\Q} \Ind_{P(\A_{\ff})}^{G(\A_{\ff})} \bigoplus\limits_{a+b=i} H^a(M_P, \cX_P; H^b(\Lie N_P, V)) \ar[d, dash, "{\sim}"] \\
     H^i(\gfr,K; \cA(G) \otimes_{\C} V_{\C}) \ar[r] & \Ind_{P(\A_{\ff})}^{G(\A_{\ff})} \bigoplus\limits_{a+b=i} H^a(\mfr_P, K_P; \cA(M_P) \otimes_{\C} H^b(\nfr_P,V_{\C}))
  \end{tikzcd}
\end{equation}

\begin{remark}
  Notons provisoirement \(\Qbar\) la clôture algébrique de \(\Q\) dans \(\C\).
  La représentation admissible \(\Qbar \otimes_\Q H^i(G, \cX; V)\) (resp.\ \(\Qbar \otimes_\Q H^a(M_P, \cX_P; H^b(\Lie N_P, V))\)) admet une décomposition canonique en somme directe d'espaces propres généralisés pour l'action du centre de Bernstein de \(G(\A_{\ff})\) (resp.\ \(M_P(\A_{\ff})\)).
  Ces décompositions sont compatibles entre elles via l'application de restriction au bord.
  Elles sont aussi compatibles via les isomorphismes verticaux avec la décomposition de \(\cA(G)\) par support cuspidal \cite[\S III.2.6]{MoeglinWaldspurger_bookspec}, qui est plus fine en général.
\end{remark}

\section{Fonctions $L$ standard}

\subsection{} Nous étendons d'abord à un corps de nombres arbitraire le Théorème 2.5 de~\cite{ClozelKret}. Considérons une représentation cuspidale, régulière algébrique et auto-duale $\pi$ de $\GL_{2m}(\A_F)$. Nous supposons $\pi$ superrégulière. En une place réelle $v$, la définition est celle de~\cite[Def. 2.1]{ClozelKret}. En une place complexe, la même définition s'applique, le paramètre de Langlands de $\pi_v$ étant
\begin{equation*}
z \mapsto ((z/\li z)^{p_1}, \ldots, (z/\li z)^{p_m}, (z/\li z)^{-p_m}, \ldots, (z/\li z)^{-p_1}).
\end{equation*}

\begin{theorem}\label{thm:Clozel-Standard}
($\pi$ cuspidale, autoduale, algébrique, superrégulière). Pour tout $a \in \Aut(\C/\Q)$,
$$
L(1/2, \pi) = 0 \Leftrightarrow L(1/2, a(\pi)) = 0.
$$
\end{theorem}

Si $F$ est totalement imaginaire, ceci est un résultat de Moeglin. Nous pourrons donc supposer que $F$ a une place réelle. Alors $\pi$ est symplectique~\cite[Prop. 1.1]{ClozelKret}, \textit{i.e.} , $L(s, \pi, \wedge^2)$ a un pôle en $s = 1$. Désormais nous supposons donc $\pi$ symplectique.

Nous considérons le $F$-groupe $G = \Sp_n/F$ ($n = 2m$) et son groupe dual $\hat G = \SO_{2n+1}(\C)$. Il existe un paramètre d'Arthur $\psi = \pi \otimes \sp(2) \oplus \one$ (cf.~\cite[\S 2, 2.1]{ClozelKret}). (On rappelle que $\sp(r)$ désigne la représention de degré $r$ de $\SL(2, \C)$. ) Il définit un espace $\cH_\psi \subset L_\disc^2(G(F) \backslash G(\A))$, $\A = \A_F$.

Le groupe $G$ a un $F$-parabolique $P=MN$ de sous-groupe de Levi $M \cong \GL_n$. Soit $I_s$ la représentation unitairement induite (adélique)
\begin{equation*}
I_s = \ind_{MN}^G(\pi \otimes |\det|^s).
\end{equation*}

Pour $s = 1/2$, la représentation $I_{s, \infty}$ de $G(F_\infty)$ a un caractère infinitésimal entier et régulier~\cite[\S 2.2]{ClozelKret}. Si $\pi_G$ est une représentation unitaire de $G(F_v)$ ($v |\infty$) ayant cette propriété, $\pi_G$ est cohomologique.
(Pour les places complexes, ceci est d\`{u} à Enright~\cite{EnrightDuke79}, pour les places réelles à Salamanca-Riba~\cite{SalamancaRiba} et Vogan-Zuckerman~\cite{VoganZuckerman}).
Pour la theorie des series d'Eisenstein, nous suivons Kim~\cite[Chap.~5]{cogdell2004lectures} ansi que ses notations.

 Pour $\pi_G \subset \cH_\psi$ , $\pi_{G,\infty}$ est donc cohomologique.

Considérons les séries d'Eisenstein, définies pour $\Re(s) \gg 0$~:
\begin{equation*}
E(f,s) \in \cA(G(F)\backslash G(\A)) \quad \tu{pour} \quad f \in I_s.
\end{equation*}
Supposons $\pi$ non-ramifiée hors d'un ensemble fini $S \supset S_\infty$. On considère $f = \bigotimes f_v$ décomposée dans $I_s$, $f_v$ non ramifiée pour $v \notin S$. Le terme constant est
\begin{equation}\label{eq:Clozel4_1}
E(f,s)_P = f + M_S(s) f_S \otimes \frac {L^S(s,\pi)L^S(2s,\pi, \wedge^2)}{L^S(s+1, \pi)L^S(2s+1,\pi,\wedge^2)} f^S
\end{equation}
le second terme étant dans $I_{-s}$. Ici $M_S(s)$ est le produit des op\'erateurs d'entrelacement locaux~\cite{cogdell2004lectures}.

Notons ici que $L(s, \pi, \Lambda^2)$ est la fonction $L$ définie par Shahidi.
On sait \cite{Henniart_LSymAlt} que ses facteurs locaux sont aussi donnés par les paramètres de Langlands des représentations $\pi_v$; il en est de  même pour la fonction $L(s, \pi, S^2)$.
En particulier on a l'identité
\[ L(s, \pi \times \pi) = L(s, \pi, \Lambda^2)  L(s, \pi, S^2). \]

Rappelons aussi que ces fonctions $L$ (complétées par leurs facteurs archimédiens) sont holomorphes dans tout le plan complexe, à l'exception d'un pôle simple en s=1, s=0 quand $\pi$ est symplectique, resp. orthogonale.
Voir par exemple \cite{Grbac_resispec}.

Pour $v \in S$ finie, soit
$$
Q_v(s) = \frac {L_v(s, \pi) L_v(2s, \pi, \wedge^2)}{L_v(s+1, \pi) L_v(2s+1, \pi, \wedge^2)} \eps(s, \pi_v) \eps(2s, \pi_v, \wedge^2)
$$
et $N_v(s) = Q_v(s)\inv M_v(s)$. Soit $S= S_{\infty} \cup S_{\ff}.$

Le second terme de \eqref{eq:Clozel4_1} se réécrit alors
\begin{equation}\label{eq:Clozel4_2}
M_{S_{\infty}} f_{S_{\infty}} \otimes N_{S_{\ff}}(s)f_{S_{\ff}} \otimes f^S \frac {L(s, \pi) L(2s, \pi, \wedge^2)}{L(s+1, \pi) L(2s+1, \pi, \wedge^2)} \prod_{v \in S} \eps(s, \pi_v)\inv \eps(2s, \pi_v, \wedge^2)\inv.
\end{equation}
On peut négliger les facteurs $\eps$ dans l'argument suivant, puisqu'il sont holomorphes et non nuls. D'après \cite[Thm. 11.1]{Cogdell_etal_IHES_2004}, $N_{S_{\ff}(s)}$ est holomorphe non-nul pour $\Re(s) \geq 0$. Supposons maintenant que $L(1/2, \pi) \neq 0$. Puisque $M_v(s)$ est holomorphe et non nul pour $\Re(s) > 0$ et $v | \infty$, on en déduit alors que $E(f, s)$ présente, pour $f$ convenablement choisie, un résidu en $s = 1/2$.

\subsection{} On peut donc suivre pas à pas les arguments de~\cite[\S 2.4]{ClozelKret}. Supposons  pour simplifier que $\pi_G$ soit cohomologique pour le système de coefficients constant.

Soit
\begin{equation*}
R(1/2) \colon I(1/2) \to L^2_\disc(G(F) \backslash G(\A))
\end{equation*}
l'opérateur résidu. Alors l'image de $R(1/2)$ définit une représentation $\pi_G$ de $G(\A)$, unitaire, et qui appartient à $\cH_\psi$ \footnote{Pour $v \in S$, il n'est pas évident que l'image de $M_v(1/2)$ est irréductible. Ceci n'a pas d'importance pour l'argument qui suit.}. On obtient par ailleurs, donc une application
\begin{equation}\label{eq:Clozel4_3}
\bigotimes_{v | \infty} \uH^\bullet (\ig_v, K_v; \cL(\pi_v, 1/2)) \otimes \bigotimes_{v \tu{ finie}} I_{v, 1/2} \to \uH^\bullet(\ig, K_\infty; L^2_\disc(G(F)\backslash G(\A)),
\end{equation}
où $\ig = \prod_{v | \infty} \ig_v$, $K_\infty = \prod_{v | \infty} K_v$. L'espace de droite, d'après Nair et Rai, est muni d'une $\Q$-structure. Si $a \in \Aut(\C)$, $a(\pi_{G, \ff})$ apparaît donc dans le membre de droite de~\eqref{eq:Clozel4_3}. En fait si le degré $q_v$ de la cohomologie ($v | \infty$) est minimal pour tout $v$, la restriction de ces classes de cohomologie au bord de la compactification de Borel-Serre est non nulle~\cite[Thm. 4.3]{ClozelKret} voir aussi le \S~\ref{sec:RohlfsSpeh} et donc $a(\pi_{G, \ff})$ est résiduelle. Enfin, $a(\pi_{G, \ff})$ provient d'une représentation $\pi_G'$ de $G(\A)$ qui appartient à $\cH_{\psi'}$ où $\psi' = a(\pi) \otimes \sp(2) \oplus \one$~\cite[Lemma 2.2]{ClozelKret}. (Rappelons que si $\pi$ est une repésentation cuspidale cohomologique de $GL_n(\A_{\ff})$, il existe une représentation  cuspidale cohomologique $a(\pi)$  de composante finie $a(\pi_{\ff})$: cf. \cite[Thm. 3.13 et §3.3]{Clozel_AA}.
Noter que le type à l'infini de $a(\pi)$ est déterminé. En particulier, $a(\pi)_{\infty}$ est tempérée.)

Pour conclure, nous devons montrer que $\pi_G'$ provient, par la formations de résidus, de $a(\pi)$, \textit{i.e.}, étendre à notre situation (où $\pi$, et $a(\pi)$, ne vérifient pas la conjecture de Ramanujan) le Théorème~2.3 de~\cite{ClozelKret}. On peut remplacer $\pi$ par $a(\pi)$, $\pi'_G$ est alors donc remplacée par $\pi_G \subset \cH_{\psi}$, $\psi = \pi \otimes \sp(2) \oplus \one$; $\pi_G$ n'est pas cuspidale.

Les sous-groupes de Levi $M$ de $G$ sont de la forme $M = \GL_{n_1} \times \cdots \times \GL_{n_r} \times G'$, $G' = \Sp_{n'}$, $n' = n - n_1 - \ldots - n_r$. Les séries d'Eisenstein sur $G$ sont obtenues à partir de représentations cuspidales $(\pi_1, \ldots, \pi_r, \sigma)$. On considère l’induite adélique unitaire
\begin{equation*}
I = \ind_{MN}^G(\pi_1|\ |^{s_1} \otimes \cdots \otimes \pi_r |\ |^{s_r} \otimes \sigma).
\end{equation*}
où $\omega_{\pi_i} = 1$ sur $\R^\times_+ \cong \afr^\circ_{\GL_n(F)} \subset F_\infty^\times$.
Si les résidus et values principales de celles-ci donnent une représentation cohomologique $\pi_G$ de $G(\A)$, on a, supposant que le caractère central de $\pi_i$ est trivial sur $(t, \ldots, t) \in \A_{F, \infty}^\times$ ($t > 0$), $s_i \in \tfrac 12 \Z$ (Théorème~\ref{thm:Clozel2_3}).

Les matrices de Hecke de $I$, aux places non-ramifiées, sont données en une place $v$ par l'expression ($t_i = t_{\pi_i}$)~:
\begin{equation}\label{eq:Clozel4_4}
t_1q^{s_1} \oplus \cdots \oplus t_r q^{s_r} \oplus t_1^{-1} q^{-s_1} \oplus \cdots \oplus t_r^{-1} q^{-s_r} \oplus t_\sigma,
\end{equation}
où $t_\sigma \in \SO_{2n' + 1}(\C)$. La matrice de Hecke de $\pi_G \subset \cH_{\psi}$ est égale à
\begin{equation}\label{eq:Clozel4_5}
t_\pi q^{1/2} \oplus t_\pi q^{-1/2} \oplus \one
\end{equation}
($\pi$, ou $a(\pi)$, est autoduale). Enfin
\begin{equation}\label{eq:Clozel4_6}
t_\sigma = \bigoplus_j t_{\sigma_j} \otimes \sp(m_j),
\end{equation}
cf.~\cite[2.3]{ClozelKret}.

On peut simplifier l'argument de~\cite{ClozelKret}. En effet, l'expression~\eqref{eq:Clozel4_5} contient trois matrices de Hecke cuspidales, or~\eqref{eq:Clozel4_4} en contient $2r + \sum_j m_j$. D'après les résultats de Jacquet-Shalika, on a donc les possibilités suivantes~:
\begin{equation}\label{eq:Clozel4_7}
r = 0, \sum m_j = 3
\end{equation}
\begin{equation}\label{eq:Clozel4_8}
r = 1, \sum m_j = 1.
\end{equation}
Sous l'hypothèse~\eqref{eq:Clozel4_7}, $G = G'$ et $\pi_G$ est cuspidale. Sous l'hypothèse~\eqref{eq:Clozel4_8}, $j = 1$, $\sigma$ est une représentation cuspidale associée à une représentation cuspidale autoduale $\sigma_1$ de $\GL_{2n' + 1}$, et
\begin{equation}
T_v = T = t_1 q^{s_1} \oplus t_1\inv q^{-s_1} \oplus t_{\sigma_1} = t_\pi q^{1/2} \oplus t_\pi q^{-1/2} \oplus 1.
\end{equation}
Formant le produit eulérien associé aux matrices $T_v$, on obtient
\begin{align*}
L^S(s+s_1, \pi_1) &L^S(s - s_1, \tilde \pi_1) L^S(s, \sigma_1) = \cr
&=L^S(s+1/2, \pi) L^S(s-1/2, \pi) \zeta^S_F(s).
\end{align*}
Il en résulte d'après Jacquet-Shalika que
\begin{equation*}
\{\pi_1|\ |^{s_1}, \pi_1|\ |^{-s_1}, \sigma_1 \} = \{ \pi |\ |^{1/2}, \pi |\ |^{-1/2}, \one \}.
\end{equation*}
Si $\pi_1|\ |^{s_1} = \one$, ceci implique $s_1 = 0$ (car $\omega_{\pi_1} = 1$ sur $A_F \cong \R^\times_+ \subset F_\infty^\times \subset \GL_n(F_\infty)$,) soit $\{\one, \one, \sigma_1\} = \{\pi |\ |^{1/2}, \pi |\ |^{-1/2}, \one \}$ ce qui est absurde. Donc (par exemple) $\pi_1 |\ |^{s_1} = \pi |\ |^{1/2}$ ce qui implique par le même argument que $s_1 = 1/2$ et $\pi_1 = \pi$. En conclusion, $M = \GL_n$, $G' = \{1\}$ et $\pi'_G$ provient, par formation de résidus, de $I_{\pi, 1/2}$.

Rappelons que ceci s'applique à $a(\pi)$. Il en résulte donc que $E(f, s)$ présente un résidu en $s = 1/2$, pour une fonction $f \in \ind_{MN}^G(a(\pi)|\ |^{1/2})$. Le terme constant est de nouveau donné par~\eqref{eq:Clozel4_1}, appliqué à $a(\pi)$. On écrit de nouveau son second terme sous la forme~\eqref{eq:Clozel4_2}. L'opérateur $N_{S_{\ff}}(s)$ est holomorphe pour $\Re(s) \geq 1/2$; $M_{\infty} (s)$ l'est pour $\Re(s)>0$; on voit que
$$
\frac {L(s, a(\pi)) L(2s, a(\pi), \wedge^2)} {L(s+1, a(\pi)) L(2s +1, a(\pi), \wedge^2)}
$$
doit présenter un résidu en $s = 1/2$, et on en déduit que $L(1/2, a(\pi)) \neq 0$. D'où le Théorème~\ref{thm:Clozel-Standard}.

Enfin, si $\pi_G$ est cohomologique pour un système de coefficients non-constant, on termine à l'aide des arguments de~\cite[\S~2.5]{ClozelKret}.

\section{Fonction $L$ de Rankin}

\subsection{} Comme dans~\cite[\S~3.2]{ClozelKret}, nous considérons maintenant $r \geq 2$ pair, $t \geq 3$ impair et des représentations régulières algébrique autoduales $\pi, \rho$ de $\GL_r(\A_F)$ et $\GL_t(\A_F)$. Nous supposons $\pi$ superrégulière~\cite[Def.~2.1]{ClozelKret} et $\pi, \rho$ disjointes (Def.~5.4 =~\cite[Def. 3.1]{ClozelKret}). Comme dans le Chapitre 4, ces définitions s'étendent aux places complexes.

\begin{theorem}[$\pi$, superrégulière, $\pi$ et $\rho$ disjointes]\label{thm:Clozel-Rankin} Pour tout $a \in \Aut(\C)$,
\begin{equation*}
L(1/2, \pi \times \rho) = 0 \Leftrightarrow L(1/2, a(\pi) \times a(\rho)) = 0.
\end{equation*}
\end{theorem}

Comme dans le Chapitre 4, nous pouvons supposer que $F$ a une place réelle~; $\pi$ est donc symplectique et $\rho$ orthogonale. On pose $t = 2m+1$, $n = r+ m$, $G = \Sp_n/F$, $G'=\Sp_{n-r}$, de sorte que $\GL_r \times G'$ est un sous-groupe de Levi de $G$. La représentation $\rho$ est un paramètre d'Arthur $\psi' \in \Psi_{\tu{sim}}(G')$, cf. Arthur~\cite[p. 37]{ArthurBook}\footnote{On écrit $\Psi$ plutôt que $\wt \Psi$, la notation $\wt \Psi$ concernant uniquement les groupes orthogonaux. cf. \cite[p.~31]{ArthurBook}.}. On a donc une collection de représentations irréductibles de $G'(\A)$ formant le paquet d'Arthur $L^2_{\disc}(G'(F) \ G'(\A))_{\psi'}$.

\begin{proposition}
Si $\sigma \subset L^2_{\disc}(G'(F) \backslash G'(\A))_{\psi'}$, $\sigma$ est cuspidale.
\end{proposition}

Noter que $\psi'$ est en fait un paramètre dans $\Phi_{\tu{sim}}(G')$, \textit{i. e.} la représentation de $\SU(2)$ est triviale \cite[p.~29]{ArthurBook}. Soit $v$ une place archimédienne~; $\rho$ définit une représentation $\varphi \colon L_v \to \SO_t(\C)$ où $L_v = W_{F_v}$. Celle-ci est unitaire. Si $\sigma$ vérifie la Proposition~6.2, $\sigma_v$ appartient au $L$-paquet local $\wt \Pi_\varphi$ dans  le Théorème~1.5.1 d'Arthur~\cite{ArthurBook}, d'après le Théorème~1.5.2 de~\cite{ArthurBook}. En particulier, $\sigma_v$ est tempérée aux places archimédiennes, donc $\sigma$ est cuspidale~\cite{WallachConstantTerm}. L'argument donné dans~\cite{ClozelKret} après la Prop. 3.1 montre par ailleurs que $L^2_\disc(G'(F) \backslash G'(\A))_{\psi'}$ n'est pas réduit à $0$. Nous utiliserons un résultat beaucoup plus explicite. Soit $B' \subset G'$ un sous-groupe de Borel, $N'$ son radical unipotent et $\xi\colon N'(F) \backslash N'(\A) \to \C$ un caractère non dégénéré. On définit alors les représentations génériques pour $(B', \xi)$ de $G'(\A)$.

\begin{theorem}\label{thm:Clozel6_3}
Il existe une représentation $\sigma \subset L^2_\disc(G'(F) \backslash G'(\A))_{\psi'}$ qui est $(B', \xi)$-générique.
\end{theorem}

Voir ~\cite[p. 483 et 484]{ArthurBook}) ; la démonstration d'Arthur repose sur les travaux de Cogdell, Kim, Piatetski-Shapiro, Shahidi, Ginzburg, Rallis et Soudry.

\subsection{} Considérons maintenant le paramètre d'Arthur (pour $G$)
$$
\psi = \pi \otimes \sp(2) \oplus \rho.
$$
Il définit, en presque toute place finie, une matrice de Hecke (dans $\hat G$) qui coïncide avec celle de la représentation induite adélique suivante. Soit $M = \GL_r \times G' \subset G$, $P = MN$ le sous-groupe parabolique standard associé à $M$ et
\begin{equation}\label{eq:Clozel6_1}
I_s = \ind_{MN}^G (\pi[s] \otimes \sigma), \quad s = 1/2.
\end{equation}

Soit $(p_i))_{i \leq r}$ les paramètres du paramètre de Langlands de $\pi_v$ en une place archimédienne, et $(q_j)_{j \leq t}$ ceux de $\rho_v$~; $p = (p_i)$ et $q = (q_j)$ sont alors les caractères infinitésimaux de $\pi_v$ et $\rho_v$. De l'identité de caractères entre $\sigma_v$ et $\rho_v$, il résulte que $q$ est aussi le caractère infinitésimal de $\sigma_v$. Comme dans~\cite[\S~3.5]{ClozelKret} nous supposons maintenant $\pi$ superrégulière, et de plus

\begin{definition} (=~\cite[Def.~3.1]{ClozelKret}) $\pi$, $\rho$ sont disjointes si, pour tous $i \leq r/2$, $j \leq m$, $p_i \pm 1/2 \neq q_j$.
\end{definition}

On a posé
\begin{align*}
p &= (p_1 > p_2 > \ldots > -p_1) \cr
q &= (q_1 > \ldots > q_m > 0 > \ldots > -q_1).
\end{align*}
Le caractère infinitésimal de la représentation~\eqref{eq:Clozel6_1} est alors régulier~\cite[\S~3.5]{ClozelKret}. En particulier, tout sous-quotient unitaire de~\eqref{eq:Clozel6_1} est cohomologique (pour $s = 1/2$). Soit $I_s$ l'induite~\eqref{eq:Clozel6_1}, réalisée dans un espace $I$ indépendant de $s$. On considère les séries d'Eisenstein
$$
E(-, s) \colon I_s \to \cA(G(F) \backslash G(\A))
$$
absolument convergentes pour $\Re(s)$ assez grand. Soit $S \supset \infty$ un ensemble fini de places tel que, pour $v \notin S$, $\pi \otimes \sigma$ soit non-ramifiée. On considère une fonction décomposée $f = \otimes f_v$ dans $I_s$ telle que, pour $v \notin S$, $f_v$ soit le vecteur sphérique de l'induite. Le terme constant de $E(-, s)$ par rapport à $P$ est
$$
E_P(f_s) = f_s + M(s) f_s
$$
où $M(s) = \bigotimes_v M_v(s) \colon I_s \to I_{-s}$. On peut écrire
\begin{equation}\label{eq:Clozel6_2}
M(s) f_s = \bigotimes_{v \in S} M_v(s) f_v \otimes \prod_{v \notin S} \frac {L_v(s, \pi \times \rho) L_v(2s, \pi, \wedge^2)}{L_v(s+1, \pi \times \rho)L_v(2s+1, \pi, \wedge^2)} \bigotimes_{v \notin S} f_v.
\end{equation}
Le produit est absolument convergent pour $\Re(s) > 1$. Pour tout $v | \infty$, $\pi \otimes \sigma$ est tempérée et $M_v(s)$ est convergent, et donc holomorphe, pour $\Re(s) > 0$. Pour $s = 1/2$, il envoie $I_{1/2}$ sur son quotient de Langlands $\cL(\pi, \sigma, 1/2)$.

Considérons $M_v(s)$ pour $v \in S_\tu{f} := S - S_\infty$.
Shahidi \cite[\S 8.4]{Shahidi_EisL} a défini une fonction $L$ locale $L(s, \pi_v \times \sigma_v)$, holomorphe pour $\Re(s) \gg 0$.

\begin{equation} \label{eq:LShahidi_eq_LRankin}
\textit{On a $L(s, \pi_v \times \sigma_v) = L(s, \pi_v \times \rho_v)$ où la fonction $L$ de droite est la fonction $L$ de Rankin usuelle.}
\end{equation}

Ceci résulte de résultats de Cogdell, Kim, Piatetski-Shapiro et Shahidi \cite{Cogdell_etal_IHES_2004}.
En effet, les auteurs prouvent tout d'abord, si $\sigma_v$ est composante locale d'une représentation cuspidale globale générique, l'existence de $\rho_v$, unique, caractérisée par l'égalité de facteurs gamma  $\gamma (s, \pi_v \times \sigma_v, \psi_v) = \gamma (s, \pi_v \times \rho_v, \psi_v)$ pour toute représentation supercuspidale $\pi_v$ de $GL_r(F_v)$ (Prop.\ 7.2 \loccit).
De plus, ceci est compatible à la fonctorialité globale.
Puis ils démontrent (Prop.\ 7.5 et Lem.\ 7.2 \loccit) l'existence d'une image locale $\rho'_v$ de $\sigma_v$ vérifiant cette identité de facteurs  gamma  et l'identité $L(s, \pi_v \times \sigma_v) = L(s, \pi_v \times \rho'_v)$ pour toute représentation irréductible générique $\pi_v$ de $GL_r(F_v)$.
En particulier $\rho_v$ doit être égale à $\rho'_v$, d'où l'égalité des fonctions L.

\begin{remark}
  Il résulte de \eqref{eq:LShahidi_eq_LRankin} et de la méthode de Luo, Rudnick et Sarnak \cite{LuoRudnickSarnak} que $L(s, \pi_v \times \sigma_v)$ est holomorphe pour $\Re(s)> 1- 1/(r^2+1)-1/(t^2+1)$.
\end{remark}

Soit $N_v(s)$ l'opérateur d'entrelacement normalisé \cite{Shahidi90}
\begin{align*}
N_v(s) &= r(s, \pi_v \times \sigma_v)\inv M_v(s), \cr
r(s, \pi_v \times \sigma_v) &= \frac {L(s, \pi_v \times \sigma_v) L(2s, \pi_v, \wedge^2)}{L(s+1, \pi_v \times \sigma_v)L(2s+1, \pi_v, \wedge^2)} \lbr \eps(s, \pi_v \times \sigma_v) \eps(s, \pi_v, \wedge^2)\rbr.
\end{align*}
On peut ignorer le facteur epsilon, qui est holomorphe et partout non nul. On a maintenant~\cite[Thm. 11.1]{Cogdell_etal_IHES_2004}
\begin{equation}
\textit{$N_v(s)$ est holomorphe et non nul pour $\Re(s) \geq 0$}.
\end{equation}
On peut maintenant réécrire~\eqref{eq:Clozel6_2} comme
$$
M(s) f_s = \bigotimes_{v |\infty} M_v(s) f_v \otimes \bigotimes_{v \in S_{\tu{f}}} N_v(s) f_v( \prod_{v \in S_{\tu{f}}} \eps_v(s)\inv) \frac {L(s, \pi \times \rho)L(2s, \pi, \wedge^2)}{L(s+1, \pi \times \rho) L(2s+1, \pi, \wedge^2)} \bigotimes_{v \notin S} f_v.
$$
Pour $\Re(s) > 0$, $M(s) f_s$ est meromorphe, et son seul pôle provient de celui de $L(2s, \pi, \wedge^2)$ en $s= 1/2$, s'il n'est pas annulé par un zéro de $L(s, \pi \times \rho)$. En particulier

\begin{proposition}
Supposons que $L(1/2, \pi \times \rho) \neq 0$. Alors $E(f, s)$, pour $f$ convenable dans $I$, présente un pôle en $s = 1/2$. Pour $f \in I$, $\Res_{s=1/2} E(f, s)$ appartient à $L^2_\disc(G(F) \backslash G(\A))$.
\end{proposition}

Soit $R_{1/2} \colon I_{1/2} \to L^2_\disc(G(F) \backslash G(\A))$ l'opérateur d'entrelacement ainsi obtenu. Soit $\pi_G = \bigotimes_v \pi_{G, v}$ un constituant de l'image de $R_{1/2}$ (celle-ci est somme directe car contenue dans $L^2_\disc$). Si $v \notin S$, on peut supposer $\pi_{G, v}$ non-ramifiée. Si $v | \infty$, $\pi_{G, v}$ est l'image par $M_v(1/2)$ de $\ind(\pi_v[1/2] \otimes \sigma_v)$, c'est-à-dire le quotient de Langlands $\cL(\pi_v, \sigma_v, 1/2) := \cL(\ind(\pi_v[1/2]\otimes \sigma_v)$. Celui-ci est donc unitaire, et par conséquent cohomologique. (Si $v \in S_\tu{f}$, nous ne savons pas si $\ind(\pi_v[1/2] \otimes \sigma_v)$ a un unique quotient irréductible~; cela résulte peut-être des propriétés de ``quasi-temperedness'' de $\pi_v$ et $\sigma_v$, \textit{cf.}~\cite{Cogdell_etal_IHES_2004}.).

Supposons pour l'instant que, pour $v |\infty$, $\pi_{G, v}$ a pour caractère infinitésimal celui de la représentation triviale.

Nous imitons maintenant les arguments de~\cite[\S~3.8]{ClozelKret}. Soit $q = \sum_{v |\infty} q_v$ le degré minimal en lequel $\pi_{G, \infty}$ a un cohomologie non-nulle~; soit $\ig = \Lie_{\R} G(F_\infty)\otimes \C$, et $K_\infty \subset G(F_\infty)$ un sous-groupe compact maximal. Alors
$$
\pi_{G, \ff} \subset \uH^q(\ig, K_\infty; \uL^2_\disc(G(F)\G(\A))).
$$
D'après Nair et Rai, l'espace de droite est muni d'une $\Q$-structure. Comme dans le \S4 on en déduit que, pour $a \in \Aut(\C/\Q)$, $a(\pi_{G, \ff})$ se plonge dans $\uH^1(\ig, K_\infty; \uL^2_\disc(G(F) \backslash G(\A))$~; l'argument de restriction au bord de la compactification de Borel-Serre montre alors que c'est la composante finie d'une représentation résiduelle.

Pour $v \notin S$, $\pi_{G, v}$ est l'unique constituant irréductible non-ramifié de $\ind(\pi[1/2] \times \sigma)$.

\begin{lemma}
$a(\pi_{G, v})$ est l'unique constituant irréductible non-ramifié de $\ind(a(\pi_v)[1/2] \otimes a(\sigma_v))$.
\end{lemma}

Ceci résulte du Lemme~3.2 de~\cite{ClozelKret}.

La représentation $\pi_{G, v}$ appartient au paquet d'Arthur associé à $\psi = \pi \otimes \sp(2) \oplus \rho$, et $a(\pi_{G, v})$ appartient donc au paquet d'Arthur associé à $a(\psi) = a(\pi) \otimes \sp(2) \oplus a(\rho)$. (Noter que les paquets d'Arthur sont définis par les matrices de Hecke en presque toutes les places.) Il nous suffit donc d'étendre au nouveau contexte le Théorème~3.1 de~\cite{ClozelKret}~:

\begin{theorem}\label{thm:Clozel6_7}
Soit $\psi = \pi \otimes \sp(2) \oplus \rho$, $\pi_G \subset \cH_\psi$ une représentation cohomologique, non cuspidale, et supposons que $\pi_G$ est obtenue à l'aide de multi-résidus et de valeurs principales de séries d'Eisenstein à partir de $(M, \pi_M)$ où $M \subset G$ est un sous-groupe de Levi et $\pi_M$ cuspidale. Alors $M = \GL_r \times \Sp_m$, $\pi_M = \pi \times \sigma$ (ou $\sigma$ est associée à $\rho$) et $s = 1/2$.
\end{theorem}

Supposons $\pi_G$ obtenue à partir de $M = \prod_{i=1}^I \GL_{r_i} \times \Sp_{m'}$, et d'une représentation cuspidale $\pi_1 \times \ldots \times \pi_I \times \sigma'$. Alors, pour presque toute place finie $v$ de $F$, on a d'une part une matrice de Hecke
\begin{equation}
T_v = T = t_\pi q^{1/2} \oplus t_\pi q^{-1/2} \oplus t_\rho
\end{equation}
et par ailleurs, d'après Arthur~:
\begin{align}\label{eq:Clozel6_8}
T_{\sigma'} &= \bigoplus_{j=1}^J t_{\sigma_j} \otimes \sp(m_j) \cr
T &= \bigoplus_{i=1}^I t(\pi_i) q^{s_i} \oplus t(\pi_i)\inv q^{-s_i} \oplus T_{\sigma'}
\end{align}

D'après le théorème d'unicité de Jacquet-Shalika, il y a trois termes cuspidaux dans la somme~\eqref{eq:Clozel6_8}. Si $|I| = 0$, la représentation $\pi_G$ est cuspidale, contrairement à l'hypothèse. Sinon, on doit avoir $|I| = 1$ ; alors $\sigma'$ est associée à une représentation cuspidale $\rho'$ de $\GL_{2m' + 1}$.

Puisque $r$ est pair, on doit avoir $\pi_1 \vert \vert^{s_1}= \pi \vert \vert^{1/2}$ (ou $ \pi \vert \vert^{-1/2}$, ce qui produit une donnée associée) et enfin $\rho = \rho'$. De plus $\pi_1 \vert \vert^{s_1} = \pi \vert \vert^{1/2}$ implique $\pi_1 = \pi, s_1 = 1/2$ : considérer les représentations sur $A_M(\R)^\circ \subset GL_r(\A_F)$.

\subsection{} D'après le paragraphe precédent, il existe une représentation résiduelle associée à $a(\psi)$ et quotient de l'induite $\ind_{MN}^G(a(\pi)|\ |^{1/2} \times \sigma)$ où $\sigma$ est une représentation cuspidale de $\Sp_m(\A_F)$ associée à $a(\rho)$. Noter que cet argument ne semble pas permettre d'obtenir $\sigma$ générique. Pour simplifier les notations, écrivons maintenant $\pi, \rho$ plutôt que $a(\pi)$, $a(\rho)$, et soit
$$
I_s = \ind_{MN}^G(\pi |\ |^s \times \sigma)
$$
Les résidus de $E(f, s)$, pour $f \in I_s$, sont de nouveau déterminés par les résidus de $M(s)f_s$; supposant de nouveau $\pi, \sigma$ non ramifiées pour $v$ (fixé) $\notin S$, et $f_v$ alors égale au vecteur sphérique, on a alors (pour $f$ décomposée):
\begin{equation}\label{eq:Clozel_6_9}
M(s) f_s = \bigotimes_{v | \infty} M_v(s) f_v \otimes \bigotimes_{v \in S_\ff} M_v(s) f_v \otimes Q^S(s) f^S,
\end{equation}
où $Q^S(s)$ est le quotient de fonctions $L$ partielles
\begin{equation}\label{eq:clozel_6_10}
Q^S(s) = \frac {L^S(s, \pi \times \rho) L^S(2s, \pi, \wedge^2)} {L^S(s+1, \pi \times \rho) L^S(2s + 1, \pi, \wedge^2)}
\end{equation}
Ainsi $Q^S(s) \bigotimes_{v \in S_\ff} M_v(S)$ présente un pôle en $s = 1/2$ (pour une fonction $f$ convenable.)

Pour $v \in S_\ff$, introduisons l'opérateur d'entrelacement normalisé
\begin{equation}
N_v(s) = Q_v(s)\inv M_v(s)
\end{equation}
où $Q_v(s)$ est le facteur en $v$ du quotient total analogue à~\eqref{eq:clozel_6_10}. Noter que la définition correcte (Langlands, Shahidi) de $N_v(s)$ contient un facteur $\eps$, cf. la formule pour $r(s)$ après~\eqref{eq:Clozel4_1}, mais celui-ci est holomorphe non nul. On remarque aussi (cf.~
Théorème~\ref{thm:Clozel6_8}) que $Q_v(s)$ est défini par $(\pi, \rho)$ et ne dépend pas du choix de $\sigma$. On peut alors écrire
\begin{equation}\label{eq:Clozel_6_12}
Q^S(s) \prod_{v \in S_\ff} M_v(s) = Q(s) \bigotimes_{v \in S_\ff} N_v(s).
\end{equation}
Comme fonction méromorphe de $s$, $Q(s)$ étant le quotient total (sur toutes les places finies) analogue à \eqref{eq:clozel_6_10}.

\begin{theorem}[Waldspurger]\label{thm:Clozel6_8}
Pour toute place finie $v$, $N_v(s)$ est holomorphe pour $\Re(s) \geq 1/2$.
\end{theorem}

Voir la Proposition~\ref{pro:entrelac_hol_nonnul} dans l'Appendice. 

Puisque l'opérateur \eqref{eq:Clozel_6_12} a un pôle en $s = 1/2$, celui-ci doit donc provenir de $Q(s)$. Le dénominateur $L(s+1, \pi \times \rho) L(2s + 1, \pi, \wedge^2)$ étant non nul, le pôle provient de $L(s, \pi \times \rho) L(2s, \pi, \wedge^2)$ et on en déduit que $L(1/2, \pi \times \rho) \neq 0$. En se rappelant notre changement de notation, on voit que $L(1/2, a(\pi) \times a(\rho)) \neq 0$, d'où le Théorème~\ref{thm:Clozel-Rankin}.

\section{Algébricité des paramètres de Satake}
\label{sec:alg_Satake}

Suivant \cite[\S 3.2]{Clozel_AA} et \cite[\S 5]{BuzzardGee_conj} (avec une formulation légèrement différente) on rappelle comment tordre l'isomorphisme de Satake afin de le rendre défini sur \(\Q\).

Dans cette section \(F\) est un corps local non archimédien et \(G\) est un groupe réductif connexe sur \(F\).
On suppose que \(G\) est non ramifié.
On fixe une clôture séparable \(\ol{F}\) de \(F\) et on note \(\Gamma_F\) le groupe de Galois absolu de \(F\) correspondant.
On note \(q\) le cardinal du corps résiduel de \(F\), \(\varpi\) une uniformisante de \(F\) et \(|\cdot|\) la norme sur \(F\) normalisée par \(|\varpi|=q^{-1}\).
On note \(\sigma \in \Gamma_F\) un élément de Frobenius (arithmétique ou géométrique, peu importe: voir la Remarque \ref{rem:Satake_ind_gen} ci-dessous).
Soit \(S\) un tore déployé maximal dans \(G\), et soit \(T\) son centralisateur dans \(G\), qui est donc un tore maximal déployé par une extension non ramifiée de \(F\).
On choisit un sous-groupe de Borel \(B\) de \(G\) contenant \(S\).
Soit \(K\) un sous-groupe compact hyperspécial de \(G(F)\), i.e.\ \(K = \ul{G}(\cO_F)\) pour un modèle réductif \(\ul{G}\) de \(G\) sur l'anneau des entiers \(\cO_F\) de \(F\).
Le groupe dual \(\hat{G}\) vient (par définition) avec un épinglage \((\cB,\cT,(X_\alpha)_\alpha)\), en particulier on le considère comme un groupe déployé épinglé sur \(\Q\), et \({}^L G := \hat{G} \rtimes \Gamma_F\) est un groupe réductif sur \(\Q\) (au moins si on remplace \(\Gamma_F\) par un quotient fini).
On note \(W\) le groupe de Weyl de \(T_{\ol{F}}\) dans \(G_{\ol{F}}\), qui s'identifie au groupe de Weyl de \(\cT\) dans \(\hat{G}\).
On note \(W_F = W^{\Gamma_F}\), qui est aussi l'image dans \(W\) du normalisateur de \(S\) (ou \(T\)) dans \(G(F)\).
On note \(\cN_F\) la préimage de \(W_F\) dans le normalisateur de \(\cT\) dans \(\hat{G}\).
On munit \(N(F)\) de la mesure de Haar satisfaisant \(\operatorname{vol}(N(F) \cap K) = 1\).
On note \(\delta_B: T(F)/T(\cO_F) \to \R_{>0}\) le caractère modulaire (\(\operatorname{vol}(t A t^{-1}) = \delta_B(t) \operatorname{vol}(A)\) pour toute partie mesurable \(A\) de \(N(F)\) et tout \(t \in T(F)\)), aussi donné par la formule \(\delta_B(t) = |t^{2\rho_B}|\) où \(2 \rho_B \in X^*(T)^{\Gamma_F}\) est la somme des racines de \(T_{\ol{F}}\) dans \(B_{\ol{F}}\).

Pour un anneau commutatif \(R\) on note \(\cH_R(G(F),K)\) l'algèbre de Hecke en niveau \(K\), i.e.\ l'algèbre des fonctions bi-\(K\)-invariantes \(G(F) \to R\) à support compact, pour le produit de convolution.
On a en particulier une identification \(\cH_{\C}(G(F),K) \simeq \C \otimes_{\Q} \cH_{\Q}(G(F),K)\).
L'isomorphisme de Satake
\begin{align*}
  \zeta_{G,K}: \cH_{\C}(G(F),K) & \longrightarrow \cH_{\C}(T(F), T(\cO_F))^{W_F} \\
  \varphi & \longmapsto \left( t \mapsto \delta_B^{1/2}(t) \int_{N(F)} \varphi(tn) \, dn \right)
\end{align*}
n'est malheureusement pas défini sur \(\Q\) en général, mais seulement sur \(\Q(q^{1/2}) \subset \C\).
Il est toutefois défini sur \(\Q\) si \(q^{1/2} \in \Z\), ou si \(\rho_B \in X^*(T)\) (voir \ref{exa:fonct_Sat_tordu} pour quelques exemples).
Rappelons l'interprétation de Langlands de l'isomorphisme de Satake, suivant \cite[\S 6 et \S 7]{BorelCorvallis}.
Notant \(\cO(X)\) l'anneau des sections globales du faisceau structural sur un schéma \(X\) (en l'occurence affine), on a une suite d'isomorphismes de \(\Q\)-algèbres
\begin{equation} \label{eq:chain_Satake}
  \cH_{\Q}(T(F),T(\cO_F))^{W_F} \simeq \Q[X_*(S)]^{W_F} \simeq \Q[X^*(\hat{S})]^{W_F} = \cO(\hat{S})^{W_F} \xrightarrow{\sim} \cO(\hat{T} \rtimes \sigma)^{\cN_F} \xleftarrow{\sim} \cO(\hat{G} \rtimes \sigma)^{\hat{G}}.
\end{equation}
Le premier isomorphisme est donné par \(X_*(S) \simeq T(F)/T(\cO_F)\), \(\mu \mapsto \mu(\varpi)\).
L'isomorphisme \(\cO(\hat{S})^{W_F} \xrightarrow{\sim} \cO(\hat{T} \rtimes \sigma)\) est induit par le morphisme \(\hat{T} \rtimes \sigma \xrightarrow{\mathrm{pr}_1} \hat{T} \to \hat{S}\) où \(\hat{T} \to \hat{S}\) est dual de l'inclusion \(S \hookrightarrow T\).
On note \(\theta_{G,K}\) la composée de \eqref{eq:chain_Satake} et \(\theta_{G,K,\C}: \cH_{\C}(T(F),T(\cO_F))^{W_F} \simeq \cO(\hat{G}_{\C} \rtimes \sigma)^{\hat{G}_{\C}}\) son extension des scalaires.
L'isomorphisme \(\theta_{G,K,\C} \circ \zeta_{G,K}\) identifie les caractères de \(\cH_{\C}(G(F),K)\) avec les classes de \(\hat{G}(\C)\)-conjugaison semi-simples dans \(\hat{G}(\C) \rtimes \sigma\).
On notera \(c(\pi)\) la classe de conjugaison associée à une représentation irréductible non ramifiée (complexe) de \(G(F)\) (i.e.\ au caractère par lequel \(\cH_{\C}(G(F),K)\) agit sur la droite \(\pi^K\)).
Comme la notation l'indique celle-ci ne dépend pas des choix: sous-groupe compact hyperspécial \(K\) satisfaisant \(\pi^K \neq 0\), sous-groupe de Borel \(B \supset S\), tore déployé maximal \(S\).
En effet \(c(\pi)\) s'interprète aussi comme le support cuspidal de \(\pi\).

\begin{remark}
  La preuve de \cite[Proposition 6.7]{BorelCorvallis} montre en fait que la \(\Q\)-algèbre \(\cO(\hat{G} \rtimes \sigma)^{\hat{G}}\) admet une base indexée par les poids \(\cB\)-dominants \(\lambda \in X^*(\cT)^{\Gamma_F}\).
  Plus précisément à un tel poids dominant \(\lambda\) on peut associer une représentation (sur \(\Q\)) absolument irréductible \(V_\lambda\) de \(\hat{G}\), qui se prolonge de façon unique à \({}^L G\) si l'on impose que \(\Gamma_F\) agisse trivialement sur la droite des vecteurs de plus haut poids pour \((\cB,\cT)\).
  L'élément de \(\cO(\hat{G} \rtimes \sigma)^{\hat{G}}\) associé à \(\lambda\) est la fonction trace dans \(V_\lambda\).
  Notant \(V_{\lambda,\mu}\) l'espace propre pour \(\mu \in X^*(\cT)\) dans \(V_\lambda\), on a
  \[ \Tr( t \rtimes \sigma \,|\, V_{\lambda} ) = \sum_{\mu \in X^*(\cT)^{\Gamma_F}} \mu(t) \Tr(\sigma \,|\, V_{\lambda,\mu}) \]
  (les espaces correspondant à des poids \(\mu\) non invariants par \(\Gamma_F\) sont permutés par \(\sigma\) et disparaissent dans la trace).
  Vus la définition \cite[Lemma 6.4]{BorelCorvallis} de l'isomorphisme \(\cO(\hat{S})^{W_F} \xrightarrow{\sim} \cO(\hat{T} \rtimes \sigma)^{\cN_F}\) et l'identification
  \[ X_*(S) \simeq X^*(\hat{S}) \simeq X^*(\hat{T})^{\Gamma_F} \simeq X^*(\cT)^{\Gamma_F}, \]
  on voit que \(\Tr(- \,|\, V_\lambda) \in \cO(\hat{G} \rtimes \sigma)^{\hat{G}}\) correspond via les quatre derniers isomorphismes de \eqref{eq:chain_Satake} à
  \begin{equation} \label{eq:expl_Satake}
    \sum_{\mu \in X_*(S)} \Tr(\sigma \,|\, V_{\lambda,\mu}) [\mu] \in \Q[X_*(S)]^{W_F}.
  \end{equation}
\end{remark}

\begin{remark} \label{rem:Satake_ind_gen}
  Il existe une extension non ramifiée \(E\) telle que l'action de \(\Gamma_F\) sur \(\hat{G}\) se factorise par \(\Gal(E/F)\).
  Soit \(n\) le degré de \(E\) sur \(F\).
  Si \(\sigma' \in \Gamma_F\) est un autre élément dont l'image dans \(\Gal(E/F)\) est génératrice, la formule \eqref{eq:expl_Satake} montre également que pour \(f \in \cO(\hat{G}_\C)^{\hat{G}_\C}\) et \(t \in \hat{T}(\C)\) on a \(f(t \rtimes \sigma) = f(t \rtimes \sigma')\).
  En effet il est clair que l'endomorphisme \(\sigma|_{V_{\lambda,\mu}}\) du \(\Q\)-espace vectoriel de dimension finie \(V_{\lambda,\mu}\) est d'ordre fini divisant \(n\), et donc \(\sigma'|_{V_{\lambda,\mu}} = \sigma|_{V_{\lambda,\mu}}^k\) (pour un certain \(k \in (\Z/n\Z)^\times\)) lui est conjugué par \(\GL(V_{\lambda,\mu})\), donc a même trace.
  En ce sens la composée \eqref{eq:chain_Satake} ne dépend pas du choix de \(\sigma\).
\end{remark}

\begin{proposition} \label{pro:Sat_tordu}
  Soit \(G\) un groupe réductif connexe sur \(F\), supposé non ramifié.
  Soit \(\epsilon \in \frac{1}{2} X^*(G)^{\Gamma_F}\) satisfaisant \(\epsilon+\rho_{B'} \in X^*(T')\) (pour n'importe quelle paire de Borel \((B',T')\) de \(G_{\ol{F}}\), cette condition n'en dépendant pas).
  Notons \(z = (q^{1/2})^{2 \epsilon} \in Z(\hat{G})(\C)^{\Gamma_F}\) (rappelons qu'ici on a choisi \(q^{1/2}>0\) dans \(\C\)).
  On note \(m_z: \hat{G}_{\C} \rtimes \sigma \simeq \hat{G}_{\C} \rtimes \sigma\) la multiplication par \(z\), et \(m_z^*\) l'automorphisme de \(\cO(\hat{G}_{\C} \rtimes \sigma)^{\hat{G}_{\C}}\) induit.

  Alors la composée \(m_z^* \circ \theta_{G,K,\C} \circ \zeta_{G,K}\) est définie sur \(\Q\), i.e.\ c'est l'extension des scalaires d'un unique isomorphisme d'algèbre \(\cH_{\Q}(G(F),K) \simeq \cO(\hat{G} \rtimes \sigma)^{\hat{G}}\).

  En particulier pour une représentation lisse irréductible (complexe) \(\pi\) de \(G(F)\), supposée non ramifiée pour un sous-groupe compact hyperspécial \(K\) de \(G\) et pour \(a \in \Aut(\C)\) on a \(c(a(\pi)) = z^{-1} a(z c(\pi))\).
\end{proposition}
\begin{proof}
  Un petit calcul montre que la composée \(m_z^* \circ \theta_{G,K,\C} \circ \zeta_{G,K}\) est égale à la composée de
  \begin{align}
    \cH_{\C}(G(F),K) & \longrightarrow \C[X_*(S)]^{W_F} \label{eq:norm_Satake} \\
    \varphi & \longmapsto \sum_{\mu \in X_*(S)} (q^{1/2})^{-2 \langle \rho_B + \epsilon, \mu \rangle} \int_{N(F)} \varphi(\mu(\varpi) n) \, dn \, [\mu] \nonumber
  \end{align}
  avec \(\theta_{G,K,\C}\).
  Or l'application \eqref{eq:norm_Satake} est définie sur \(\Q\) car \(\rho_B + \epsilon \in X^*(T)\).
\end{proof}

\begin{example} \label{exa:fonct_Sat_tordu}
  \begin{enumerate}
  \item Pour \(G=\Sp_{2n,F}\), \(\GL_{2n+1,F}\), un groupe unitaire quasi-déployé (pour l'extension quadratique non ramifiée de \(F\)) en dimension impaire, ou un groupe spécial orthogonal (supposé non ramifié) en dimension paire, on a \(\rho_{B'} \in X^*(T')\) et on peut simplement prendre \(\epsilon=0\).
    On a donc tout simplement \(c(a(\pi)) = a(c(\pi))\).
  \item Pour \(G=\GL_{2n}\) on n'a pas \(\rho_{B'} \in X^*(T')\) mais on peut prendre pour \(2 \epsilon\) le morphisme déterminant.
    En identifiant comme d'habitude \(\hat{G}\) à \(\GL_{2n,\Q}\) la fonctorialité s'écrit \(c(a(\pi)) = q^{1/2} a(q^{-1/2} c(\pi))\).
    On retrouve \cite[\S 3.2]{Clozel_AA}.
  \end{enumerate}
\end{example}

En général il n'existe pas de caractère \(2 \epsilon\) comme dans la Proposition \ref{pro:Sat_tordu}, mais d'après \cite[\S 5.3]{BuzzardGee_conj} il existe une extension centrale (canonique)
\[ 1 \to \GL_{1,F} \xrightarrow{\nu} \widetilde{G} \to G \to 1 \]
et \(2 \epsilon \in X^*(\widetilde{G})^{\Gamma_F}\) (lui aussi canonique) satisfaisant la condition de la proposition (pour \(\tilde{G}\)).
On a en fait une isogénie centrale \(\widetilde{G} \to G \times \GL_{1,F}\) dont la composée avec \(\nu\) est \(x \mapsto (1,x^2)\) et qu'on peut écrire en termes de données radicielles de la façon suivante.
Soit \(T'\) un tore maximal de \(G_{\ol{F}}\) et \(\widetilde{T'}\) sa préimage dans \(\widetilde{G}\).
Alors \(X^*(\widetilde{T'}) / (X^*(T') \oplus \Z)\) est isomorphe à \(\Z/2\Z\), d'élément non trivial la classe de \((\rho_{B'},1/2) \in \frac{1}{2} (X^*(T') \oplus \Z)\) pour n'importe quel sous-groupe de Borel \(B'\) de \(G_{\ol{F}}\) contenant \(T'\).
On définit \(2 \epsilon = (0,1) \in X^*(T') \oplus \Z\).

On considère l'application \(\hat{\nu}: \widehat{\widetilde{G}} \to \GL_{1,\Q}\) duale de \(\nu\) et on note abusivement \(\hat{\nu}^{-1}(\{q\})\) le produit fibré de \(\hat{\nu}\) avec le morphisme \(q: \operatorname{Spec} \Q \to \GL_{1,\Q}\).
Appliquant la Proposition \ref{pro:Sat_tordu} on obtient un isomorphisme (canonique)
\[ \cH_{\Q}(G(F),K) \simeq \cO(\hat{\nu}^{-1}(\{q\}) \rtimes \sigma)^{\hat{G}}. \]
En particulier à une représentation irréductible non ramifiée (complexe) \(\pi\) de \(G(F)\) on associe un paramètre de Satake \(\tilde{c}(\pi)\) qui est une classe de conjugaison semi-simple dans \(\widehat{\widetilde{G}}(\C) \rtimes \sigma\) satisfaisant \(\hat{\nu}(\tilde{c}(\pi)) = q\), et l'application \(\tilde{c}\) est \(\Aut(\C)\)-équivariante.

\begin{example} \label{exa:fonct_Sat_tordu_SO_impair}
  Pour \(G=\SO_{2n+1,F}\) (déployé), dont le groupe dual \(\hat{G}\) s'identifie à \(\Sp_{2n,\Q}\), le groupe \(\widetilde{G}\) est isomorphe à \(\GSpin_{2n+1,F}\) et son groupe dual s'identifie à \(\GSp_{2n,\Q}\).
  On peut donc écrire la relation de fonctorialité pour le paramètre de Satake \(c(\pi)\) d'une représentation non ramifiée \(\pi\) de \(\SO_{2n+1}(F)\) dans \(\GSp_{2n}(\C)\) comme suit: \(c(a(\pi)) = q^{1/2} a(q^{-1/2} c(\pi))\).
\end{example}

\section{Paramètres d'Arthur-Langlands et paramètres de Satake}
\label{sec:para_Arthur_Satake}

Soit \(F\) un corps de nombres.
On considère \(G\) l'un des groupes quasi-déployés suivant:
\begin{enumerate}
\item \(\SO_{2n+1,F}\) déployé, où \(n \geq 0\),
\item \(\Sp_{2n,F}\) déployé, où \(n \geq 1\),
\item un groupe quasi-déployé \(\SO_{2n}^\alpha\) où \(\alpha \in F^\times/F^{\times,2}\) et \(n \geq 1\), qui est déployé sur \(F\) si \(\alpha=1\) et déployé sur \(F(\sqrt{\alpha})\) sinon.
\end{enumerate}
Pour fixer les idées dans le cas orthogonal pair non déployé on fixe le plongement suivant du quotient \(\hat{G} \rtimes \Gal(F(\sqrt{\alpha})/F)\) de \({}^L G\) dans \(\GL_{2n,\Qbar}\).
Soit \(S \in \mathrm{M}_{2n}(\Q)\) la matrice symétrique antidiagonale définie par \(S_{i,j} = \delta_{i,2n+1-j}\).
Soit \(\hat{G}\) le sous-groupe fermé de \(\GL_{2n,\Qbar}\) défini par (pour toute \(\Qbar\)-algèbre \(R\)):
\[  \hat{G}(R) = \SO(S,R) = \{ g \in \mathrm{M}_{2n}(R) \,|\, {}^t g S g = S \text{ et } \det g = 1 \}. \]
Enfin l'élément non trivial de \(\Gal(F(\sqrt{\alpha})/F)\) est identifié à
\[ \diag(I_{n-1}, \begin{pmatrix} 0 & 1 \\ 1 & 0\end{pmatrix}, I_{n-1}) \in \mathrm{O}(S,R) \smallsetminus SO(S,R). \]
On notera dans ce cas \(\Std_{{}^L G}: {}^L G \to \GL_{2n,\Qbar}\) cette représentation.
Dans les autres cas (déployés) on a également une représentation \(\Std_{{}^L G}\) de \({}^L G = \hat{G} \times \Gal(\ol{F}/F)\) triviale sur le second facteur et égale à la représentation standard de \(\hat{G} \simeq \Sp_{2n,\Qbar}\) (resp.\ \(\SO_{2n+1,\Qbar}\), resp.\ \(\SO_{2n,\Qbar}\)) si \(G=\SO_{2n+1}\) (resp.\ \(\Sp_{2n}\), resp.\ \(\SO_{2n}^1\)).

Rappelons qu'un paramètre d'Arthur-Langlands pour \(G\) est une somme formelle (non ordonnée) \(\bigoplus_i \pi_i \otimes \mathrm{sp}(d_i)\) où les \(\pi_i\) sont des représentations automorphes cuspidales pour \(\GL_{n_i,F}\) (supposées unitaires, c'est-à-dire à caractère central unitaire), les \(d_i \geq 1\) sont des entiers et le produit tensoriel est aussi formel\footnote{Plus rigoureusement cette notation désigne le multi-ensemble des paires \((\pi_i,d_i)\).}.
Cette somme formelle vérifie des conditions supplémentaires (autodualité et type orthogonal/symplectique des \(\pi_i\), multiplicité un des paires \((\pi_i,d_i)\), ainsi qu'une condition sur les caractères centraux) que nous ne détaillons pas ici.
À une telle somme formelle on peut associer une somme formelle de représentations automorphes cuspidales
\begin{equation} \label{eq:somme_formelle_ass_Arthur}
  \bigoplus_i \bigoplus_{j=0}^{d_i-1} \pi_i[(d_i-1)/2-j].
\end{equation}
Cette dernière détermine le paramètre d'Arthur-Langlands: étant donnée une représentation automorphe cuspidale unitaire \(\pi\) pour \(\GL_{m,F}\) telle qu'il existe un entier \(d\) pour lequel \(\pi[(d-1)/2]\) intervient dans \eqref{eq:somme_formelle_ass_Arthur}, on peut supposer \(d\) maximal pour cette propriété et alors \(d \geq 1\) et \((\pi,d)\) est l'un des \((\pi_i,d_i)\); une simple récurrence permet de conclure.
En outre il résulte de \cite{JacquetShalika_Euler2} qu'une somme formelle \(\sum_j \pi_j'\) de représentations automorphes cuspidales pour des groupes linéaires sur \(F\) (pas forcément unitaires) est déterminée par les paramètres de Satake associés
\[ \left( \bigoplus_j c(\pi'_{j,v}) \right)_{v \not\in S} \]
quel que soit l'ensemble fini \(S\) de places de \(F\) contenant toutes les places archimédiennes ainsi que les places finies où l'une des \(\pi'_j\) est ramifiée.

À une représentation automorphe discrète (c'est-à-dire se plongeant dans \(\cA^2(G)\)) \(\sigma\) pour \(G\) est associé un paramètre d'Arthur-Langlands \(\bigoplus_i \pi_i \otimes \mathrm{sp}(d_i)\), caractérisé par le fait qu'à toute place non-archimédienne \(v\) de \(F\) où \(\sigma_v\) est non ramifiée on a
\begin{equation} \label{eq:caract_para_Arthur_Satake}
  \Std_{{}^L G}(c(\sigma_v)) = \bigoplus_i \bigoplus_{j=0}^{d_i-1} q_v^{(d_i-1)/2-j} c(\pi_{i,v})
\end{equation}
où \(q_v\) est le cardinal du corps résiduel de \(F_v\).

\section{Action de \(\Aut(\C)\) sur les représentations cuspidales}

Soient \(F\) un corps de nombres et \(N \geq 1\) un entier.
D'après \cite[Théorème 3.13]{Clozel_AA} on a une action à gauche de \(\Aut(\C)\) (se factorisant par le groupe de Galois absolu de \(\Q\)) sur l'ensemble des classes d'isomorphismes de représentations automorphes cuspidales pour \(\GL_{N,F}\) qui sont algébriques régulières (\cite[Définitions 1.8 et 3.12]{Clozel_AA}), notée \((a,\pi) \mapsto a(\pi)\).
Cette action est caractérisée par la relation \(a(\pi)_{\ff} = a(\pi_{\ff})\), où \(\pi = \pi_\infty \otimes \pi_{\ff}\) et si \(\pi_{\ff}\) a pour espace sous-jacent \(W\), \(a(\pi_{\ff})\) est l'espace \(\C \otimes_{a^{-1}, \C} W\) sur lequel \(g \in \GL_N(\A_{F,\ff})\) agit par \(\mathrm{id}_{\C} \otimes \pi_{\ff}(g)\).

\begin{definition} \label{def:tilde_action}
  \begin{itemize}
  \item Soit \(\pi\) une représentation automorphe cuspidale \(\pi\) pour \(\GL_{N,F}\).
    Le \emph{poids} de \(\pi\) est l'unique réel \(w\) tel que le caractère central de \(\pi[w/2]\) soit unitaire (\(\pi[w/2]\) est alors unitaire).
  \item Une représentation automorphe cuspidale \(\pi\) pour \(\GL_{N,F}\) est \emph{demi-algébrique régulière} si \(\pi[1/2]\) est algébrique régulière.
  \item Pour \(a \in \Aut(\C)\) et \(\pi\) une représentation automorphe cuspidale pour \(\GL_{N,F}\) qui est algébrique régulière (resp.\ demi-algébrique régulière) on note \(\tilde{a}(\pi) = a(\pi)\) (resp.\ \(\tilde{a}(\pi) = a(\pi[1/2])[-1/2]\)).
  \end{itemize}
\end{definition}

Rappelons que d'après le lemme de pureté \cite[Lemme 4.9]{Clozel_AA} si \(\pi\) est une représentation automorphe cuspidale pour \(\GL_{N,F}\) telle qu'il existe \(s \in \C\) pour lequel \(\pi[-s]\) est algébrique, et si \(w \in \R\) désigne le poids de \(\pi\), alors \(\pi[w/2]\) est tempérée aux places archimédiennes.

\begin{proposition} \label{pro:AutC_poids_et_type_autodual}
  \begin{enumerate}
  \item L'action de \(\Aut(\C)\) sur l'ensemble des classes d'isomorphisme de représentations automorphes cuspidales algébriques (resp.\ demi-algébriques) régulières préserve le poids.
  \item Cette action commute avec le passage à la contragrédiente.
  \item Si \(\pi\) est une représentation automorphe cuspidale algébrique ou demi-algébrique régulière qui est de plus autoduale, alors pour tout \(a \in \Aut(\C)\) la représentation \(\tilde{a}(\pi)\), également autoduale par le point précédent, est du même type (symplectique ou orthogonal, voir \cite[Theorem 1.5.3]{ArthurBook}) que \(\pi\).
  \end{enumerate}
\end{proposition}
\begin{proof}
  Soit \(\pi\) une représentation automorphe cuspidale algébrique ou demi-algébrique pour \(\GL_{N,F}\), de poids \(w\).
  Soit \(v\) une place archimédienne de \(F\).
  \begin{itemize}
  \item Si \(v\) est réelle elle correspond à plongement \(\iota: F \hookrightarrow \C\) et grâce au lemme de pureté la restriction du paramètre de Langlands de \(\pi_v\) à \(\C^\times\) est conjuguée par \(\GL_N(\C)\) à
    \[ z \mapsto \diag(z^{p_{\iota,1}} \ol{z}^{-w-p_{\iota,1}}, \dots, z^{p_{\iota,N}} \ol{z}^{-w-p_{\iota,N}}) \]
    où \(p_{\iota,i} \in \Z\) pour tout \(i\) ou \(p_{\iota,i} \in \frac{1}{2} + \Z\) pour tout \(i\), avec la convention habituelle \(z^a \ol{z}^b = z^{a-b} (z \ol{z})^b\) si \(a,b \in \C\) satisfont \(a-b \in \Z\).
    On peut supposer \(p_{\iota,1} \geq \dots \geq p_{\iota,N}\).
    Le fait que ce morphisme se prolonge au groupe de Weil de \(F_v\) implique la relation \(p_{\iota,i}+p_{\iota,N+1-i}=-w\), et donc \(\sum_{i=1}^N p_{\iota,i} = -Nw/2\).
  \item Si \(v\) est complexe elle correspond à deux plongements \(\iota,\ol{\iota}: F \hookrightarrow \C\) et le paramètre de Langlands de \(\pi_v\) est conjugué à
    \[ z \in F_v^\times \mapsto \diag(\iota(z)^{p_{\iota,1}} \ol{\iota}(z)^{p_{\ol{\iota},1}}, \dots, \iota(z)^{p_{\iota,N}} \ol{\iota}(z)^{p_{\ol{\iota},N}}). \]
    Le lemme de pureté implique \(p_{\iota,i}+p_{\ol{\iota},i}=-w\) pour tout \(i\), en particulier on a
    \[ \sum_{i=1}^N p_{\iota,i} + p_{\ol{\iota},i} = -Nw. \]
  \end{itemize}
  Pour \(\iota: F \hookrightarrow \C\) on note \(p_\iota\) le multi-ensemble (i.e.\ \(\mathfrak{S}_N\)-orbite dans \(\C^N\)) associé à la famille \((p_{\iota,i})_{1 \leq i \leq N}\).
  À un décalage par \((N-1)/2\) près la famille \((p_\iota)_\iota\) est le type à l'infini \cite[\S 3.3]{Clozel_AA} de \(\pi\).
  Le type à l'infini détermine le poids: on a
  \begin{equation} \label{eq:type_infty_poids}
    \sum_{\iota: F \hookrightarrow \C} \sum_{i=1}^N p_{\iota,i} + p_{\ol{\iota},i} = -[F:\Q] N w.
  \end{equation}
  Supposons maintenant \(\pi\) régulière.
  On sait que le type à l'infini de \(\tilde{a}(\pi)\) est \((p_{a^{-1} \circ \iota})_\iota\), il résulte de \eqref{eq:type_infty_poids} que le poids de \(\tilde{a}(\pi)\) est \(w\).

  Montrons maintenant que \(\tilde{a}(\pi)^\vee \simeq \tilde{a}(\pi^\vee)\).
  Si \(\pi\) est algébrique régulière on a pour toute place non-archimédienne \(v\) la relation \(a(\pi)_v \simeq a(\pi_v)\), d'où il découle que la représentation contragrédiente de \(a(\pi)_v\) est \(a(\pi_v^\vee)\), et donc \(a(\pi)^\vee \simeq a(\pi^\vee)\) par multiplicité un forte.
  Si \(\pi\) est demi-algébrique régulière on a pour tout place non-archimédienne \(v\) la relation \(\tilde{a}(\pi)_v \simeq a(\pi[1/2])[-1/2]_v\) d'où il découle
  \[ \tilde{a}(\pi)_v^\vee \simeq a(\pi^\vee[-1/2])[1/2]_v \simeq a(\pi^\vee[1/2])[-1/2]_v \simeq \tilde{a}(\pi^\vee) \]
  puisque la torsion par une puissance entière de \(|\cdot|\) commute à l'action de \(\Aut(\C)\).
  On en déduit encore \(\tilde{a}(\pi)^\vee \simeq \tilde{a}(\pi^\vee)\).

  Supposons enfin que \(\pi\) est en outre autoduale.
  Si \(N\) est impair on sait que \(\pi\) est de type orthogonal, même sans hypothèse d'algébricité ou de régularité.
  On suppose donc \(N\) pair.
  Si \(F\) n'est pas totalement complexe, i.e.\ s'il existe une place réelle \(v\) de \(F\), alors l'alternative symplectique/orthogonal se lit sur le caractère infinitésimal de \(\pi_v\) (par le même argument que la preuve de \cite[Proposition 1.1]{ClozelKret}).
  En particulier \(\pi\) est de type symplectique si elle est algébrique régulière, de type orthogonal si elle est demi-algébrique régulière.
  Il nous reste donc à traiter le cas où \(N\) est pair et \(F\) est totalement complexe.

  Supposons que \(\pi\) est algébrique régulière et de type symplectique.
  Soit \(G\) le groupe déployé \(\SO_{N+1,F}\).
  Soit \(\sigma\) une représentation automorphe cuspidale pour \(G\) de paramètre d'Arthur-Langlands \(\pi\) (par exemple celle qui est générique, voir \cite[Proposition 8.3.2 (b)]{ArthurBook}).
  Pour une place archimédienne (complexe) de \(F\) on note \(\gfr_v\) la complexifiée de l'algèbre de Lie du groupe de Lie (réel) \(G(F_v)\) et \(K_v\) un sous-groupe compact maximal de \(G(F_v)\).
  On note \(\gfr = \prod_{v|\infty} \gfr_v\), \(K_\infty = \prod_{v | \infty} K_v\) et \(\cX = G(\R \otimes_\Q F) / K_\infty\).
  Pour une telle place \(v\) le caractère infinitésimal de \(\sigma_v\) est algébrique régulier et \(\sigma_v\) est unitaire, donc elle est cohomologique (d'après \cite{SalamancaRiba} et \cite{VoganZuckerman}, en notant que \(G(F_v)\) est semi-simple et connexe): il existe une (unique) représentation algébrique irréductible \(V_v\) de \(G(\C \otimes_{\R} F_v)\) telle que \(H^\bullet(\gfr_v,K_v; \sigma_v \otimes V) \neq 0\).
  Soit \((V, E \xrightarrow{\iota_0} \C))\) la paire associée par la Proposition \ref{pro:irr_alg_rep_qs} à la représentation \(\otimes_{v|\infty} V_v\) de \((\Res_{F/\Q} G)_\C = \prod_{v|\infty} G(\C \otimes_\R F_v)\), en particulier \(\otimes_{v|\infty} V_v \simeq \C \otimes_{\iota_0,E} V\).
  Il existe donc un entier \(q \geq 0\) tel que la représentation \(\sigma_{\ff}\) intervient avec multiplicité non nulle dans
  \[ H^q(\gfr,K_\infty; \cA^2(G) \otimes_{\C} V_{\C}) \simeq \bigoplus_{\iota: E \to \C} H^q(\gfr,K_\infty; \cA^2(G) \otimes_{\C} (\C \myotimes{\iota,E} V)), \]
  plus précisément elle intervient dans le facteur à droite correspondant à \(\iota_0\).
  D'après le Théorème \ref{thm:Nair_Rai} la représentation \(\sigma_{\ff}\) intervient donc dans
  \[ \C \otimes_{\Q} W^{\pm} H^q(G, \cX; V) \simeq \bigoplus_{\iota: E \hookrightarrow \C} \C \myotimes{\iota,E} W^{\pm} H^q(G, \cX; V), \]
  de nouveau dans le terme correspondant à \(\iota_0\) car l'isomorphisme \eqref{eq:Nair_Rai}, étant fonctoriel en \(V\), est \(E\)-équivariant.
  Remarquons que le Théorème \ref{thm:Nair_Rai} implique aussi que cette représentation de \(G(\A_{F,\ff})\) est semi-simple.
  On en déduit que pour \(a \in \Aut(\C)\) la composante \(a(\sigma_{\ff})\)-isotypique de
  \[ \C \myotimes{a^{-1} \circ \iota_0,E} W^{\pm} H^q(G, \cX; V) \]
  est non nulle.
  Utilisant le Théorème \ref{thm:Nair_Rai} dans l'autre sens on en déduit l'existence d'un \((\gfr,K)\)-module admissible irréductible \(\sigma'_\infty\) tel que \(\sigma' := \sigma'_\infty \otimes a(\sigma_{\ff})\) se plonge dans \(\cA^2(G)\) et
  \[ H^q(\gfr,K_\infty; \sigma'_\infty \otimes_\C (\C \otimes_{a^{-1} \circ \iota_0, E} V)) \neq 0. \]
  D'après l'Exemple \ref{exa:fonct_Sat_tordu_SO_impair} on a en toute place \(v\) où \(\sigma\) est non ramifiée \(c(\sigma'_v) = q_v^{1/2} a(q_v^{-1/2} c(\sigma_v))\) (égalité de classes de conjugaisons semi-simples dans \(\GSp_N(\C)\), où \(q_v\) est le cardinal du corps résiduel de \(F_v\)).
  En utilisant l'Exemple \ref{exa:fonct_Sat_tordu} (pour \(\GL_N\)) on déduit
  \[ \Std_{{}^L G}(c(\sigma'_v)) = q_v^{1/2} \Std_{{}^L G}(q_v^{-1/2} a(c(\sigma_v))) = q_v^{1/2} a(q_v^{-1/2} c(\pi_v)) = c(a(\pi_v)). \]
  On en déduit (voir \eqref{eq:caract_para_Arthur_Satake}) que le paramètre d'Arthur-Langlands de \(\sigma'\) est \(a(\pi)\), ce qui montre que \(a(\pi)\) est également de type symplectique.
  
  On peut raisonner de façon similaire si \(\pi\) est demi-algébrique régulière et de type orthogonal, en utilisant cette fois le groupe \(G = \SO_N^\alpha\) où \(\Gamma_{F(\sqrt{\alpha})}\) est le noyau du caractère \(\Gamma_F \to \{\pm 1\}\) correspondant au caractère central de \(\pi\).
  En utilisant cette fois l'Exemple \ref{exa:fonct_Sat_tordu} pour \(G\) on conclut encore que \(\tilde{a}(\pi) = a(\pi[1/2])[-1/2]\) est de type orthogonal.

  Ces deux cas particuliers suffisent pour montrer que \(\tilde{a}(\pi)\) est toujours du même type que \(\pi\).
  En effet on peut raisonner par l'absurde et supposer que \(\tilde{a}(\pi)\) et \(\pi\) ne sont pas du même type.
  Comme on sait que \(\tilde{a}(\pi)\) est algébrique (resp.\ demi-algébrique) si \(\pi\) est algébrique (resp.\ demi-algébrique), quitte à remplacer \(\pi\) par \(\tilde{a}(\pi)\) et \(a\) par \(a^{-1}\) on se ramène à l'un des deux cas précédents.
\end{proof}

\begin{remark}
  On a rappelé dans la preuve précédente que pour une représentation automorphe cuspidale algébrique régulière \(\pi\) pour \(\GL_{N,F}\) de poids \(w\), si \((p_\iota)_{\iota: F \hookrightarrow \C}\) désigne son type à l'infini (i.e.\ le caractère infinitésimal de \(\otimes_{v|\infty} \pi_v\)) alors
  \begin{itemize}
  \item pour tout plongement \(\iota\) on a l'égalité (de multi-ensembles) \(p_{\ol{\iota}} = -w-p_\iota\),
  \item pour \(a \in \Aut(\C)\) le type à l'infini de \(a(\pi)\) est \((p_{a^{-1} \iota})_\iota\).
  \end{itemize}
  Lorsque \(F\) n'est ni totalement réel ni CM la deuxième condition renforce la première, puisque la première s'applique aussi bien à \(a(\pi)\).
  Ainsi notant \(c \in \Aut(\C)\) la conjugaison complexe on a \(p_{aca^{-1} \iota} = -w-p_{\iota}\) pour tous \(\iota\) et \(a\).
  Soit \(F_\mathrm{CM}\) la plus grande sous-extension CM\footnote{corps de nombres totalement réel ou extension quadratique totalement imaginaire d'un corps de nombre totalement réel} de \(F\).
  On déduit facilement des relations ci-dessus que l'application \(\iota \mapsto p_\iota\) se factorise par \(\iota \mapsto \iota|_{F_\mathrm{CM}}\).
  Notant pour \(\iota: F_\mathrm{CM} \hookrightarrow \C\): \(p_\iota = p_{\tilde{\iota}}\) où \(\tilde{\iota}: F \hookrightarrow \C\) est n'importe quel prolongement de \(\iota\), on a toujours \(p_{\ol{\iota}} = -w-p_\iota\).
  Si \(F_\mathrm{CM}\) est totalement réelle (ce qui peut arriver même si \(F\) n'est pas totalement réelle) cela donne \(p_\iota = -w - p_\iota\) pour tout \(\iota\).
\end{remark}

\begin{definition} \label{def:SO_reg}
  Soit \(n \geq 1\) un entier et \(\pi\) une représentation automorphe cuspidale pour \(\GL_{2n,F}\), qu'on suppose demi-algébrique.
  On dit que \(\pi\) est \emph{\(\SO\)-régulière} si en toute place réelle (resp.\ complexe) \(v\) de \(F\) le caractère infinitésimal \(p_\iota\) (resp.\ \((p_\iota, p_{\ol{\iota}})\)) de \(\pi_v\) est de la forme
  \[ p_{\iota,1} > \dots > p_{\iota,n} \geq -p_{\iota,n} > \dots > -p_{\iota,1} \]
  resp.
  \[ (p_{\iota,1} > \dots > p_{\iota,n} \geq -p_{\iota,n} > \dots > -p_{\iota,1}),\ (p_{\ol{\iota},1} > \dots > p_{\ol{\iota},n} \geq -p_{\ol{\iota},n} > \dots > -p_{\ol{\iota},1}). \]
\end{definition}

\begin{proposition} \label{pro:AutC_SO_reg}
  Soit \(n \geq 1\) un entier et \(\pi\) une représentation automorphe cuspidale pour \(\GL_{2n,F}\) qui est autoduale, demi-algébrique et \(\SO\)-régulière.
  \begin{enumerate}
  \item Si \(F\) a une place réelle et \(n>1\) alors \(\pi\) est de type orthogonal.
  \item Si \(\pi\) est de type orthogonal alors pour tout \(a \in \Aut(\C)\) il existe une unique représentation automorphe cuspidale \(\tilde{a}(\pi)\) pour \(\GL_{2n,F}\) telle que pour toute place non-archimédienne \(v\) de \(F\) où \(\pi_v\) est non ramifiée on a
    \[ c(\tilde{a}(\pi_v)) = a(c(\pi_v)). \]
    De plus le caractère infinitésimal de \(\tilde{a}(\pi)_\infty\) est \(a((p_\iota)_\iota)\) où \((p_\iota)_\iota\) est le caractère infinitésimal de \(\pi_\infty\) (en particulier \(\tilde{a}(\pi)\) est également demi-algébrique et \(\SO\)-régulière) et \(\tilde{a}(\pi)\) est autoduale de type orthogonal.
  \end{enumerate}
\end{proposition}
\begin{proof}
  Le premier point se prouve exactement comme \cite[Proposition 1.1]{ClozelKret}: si \(v\) est une place réelle de \(F\) nos hypothèses impliquent que le paramètre de Langlands de \(\pi_v\) admet un facteur irréductible de dimension \(2\) avec multiplicité un qui est autodual de type orthogonal (mais pas symplectique), et \cite[Theorem 1.4.2]{ArthurBook} implique que \(\pi\) est de type orthogonal.

  Démontrons le deuxième point par récurrence sur \(n \geq 1\).
  Dans le cas \(n=1\) on sait que le caractère central de \(\pi\) est d'ordre deux, correspondant à une extension quadratique \(E\) de \(F\) et \(\pi\) est l'induite automorphe d'un caractère \(\chi\) de \(\A_E^\times/E^\times\) dont la restriction à \(\A_F^\times/F^\times\) est triviale.
  On sait aussi que toute induite automorphe de cette forme est autoduale de type orthogonal, et qu'elle est demi-algébrique si et seulement si \(\chi\) est algébrique.
  On se ramène ainsi à l'action de \(\Aut(\C)\) sur les caractères de Hecke algébriques.
  Supposons maintenant le deuxième point vérifé pour \(n'<n\) et déduisons-le pour une représentation \(\pi\) pour \(\GL_{2n,F}\).
  On procède comme dans la démonstration de la Proposition \ref{pro:AutC_poids_et_type_autodual}.
  On choisit une représentation automorphe discrète \(\sigma\) pour \(G=\SO_{2n,F}^\alpha\) de paramètre d'Arthur-Langlands \(\pi\).
  En toute place non-archimédienne \(v\) où \(\pi_v\) est non ramifiée il existe un sous-groupe compact hyperspécial \(K_v\) de \(G(F_v)\) telle que \(\sigma_v\) soit sphérique pour \(K_v\), et on peut les choisir de sorte que \(\prod_v K_v\) soit un sous-groupe ouvert du facteur direct \(\prod'_v G(F_v)\) de \(G(\A_F)\).
  Observons qu'en une place archimédienne \(v\) de le caractère infinitésimal de \(\sigma_v\) est algébrique régulier (la Définition \ref{def:SO_reg} est faite pour cela), ainsi on obtient encore que \(\sigma_{\ff}\) intervient dans
  \[ \C \otimes_{\iota_0,E} W^{\pm} H^q(G,\cX; V) \]
  pour \((V, \iota_0: E \hookrightarrow \C)\) et \(q\) comme dans la preuve de la Proposition \ref{pro:AutC_poids_et_type_autodual}.
  Ici \(\cX = G(\R \otimes_\Q F) / K_\infty^0\) où \(K_\infty = \prod_{v|\infty} K_v\) et \(K_v\) un sous-groupe compact maximal de \(G(F_v)\) (comme \(F\) n'est pas forcément totalement complexe \(K_v\) n'est pas forcément connexe).
  On obtient encore une représentation automorphe discrète \(\sigma'\) pour \(G\) telle que \(\sigma'_{\ff} \simeq a(\sigma_{\ff})\) et
  \begin{equation} \label{eq:SOreg_sigma'_coh}
    H^q(\gfr, K_\infty^0; \sigma'_\infty \otimes_\C (\C \otimes_{a^{-1} \circ \iota_0, E} V)) \neq 0.
  \end{equation}
  D'après l'Exemple \ref{exa:fonct_Sat_tordu} (pour \(G\)) en toute place non ramifiée
  \[ \Std_{{}^L G}(c(\sigma'_v)) = \Std_{{}^L G}(a(c(\sigma_v))) = a(c(\pi_v)). \]
  Le paramètre d'Arthur-Langlands de \(\sigma'\) est donc une somme formelle \(\bigoplus_i \pi_i \otimes \mathrm{sp}(d_i)\) satisfaisant en toute place non-archimédienne \(v\) où \(\pi_v\) est non ramifiée
  \[ a(c(\pi_v)) = \bigoplus_i \bigoplus_{j=0}^{d_i-1} q_v^{(d_i-1)/2-j} c(\pi_{i,v}). \]
  On sait que tout valeur propre \(x\) de \(c(\pi_v)\) vérifie \(q_v^{-1/2} < |x| < q_v^{1/2}\), cela implique que pour toute paire \((x_1,x_2)\) de valeurs propres on a \(x_1/x_2 \neq q\).
  Cette condition est préservée par l'action de \(\Aut(\C)\), on en déduit qu'on a \(d_i=1\) pour tout \(i\), i.e.\ le paramètre d'Arthur-Langlands est somme formelle de représentations cuspidales autoduales de type orthogonal.
  On déduit de \eqref{eq:SOreg_sigma'_coh} que les caractères infinitésimaux \((p_{i,\iota})_\iota\) des facteurs \(\pi_i\) satisfont \(\bigsqcup_i p_{i,a \circ \iota} = p_\iota\), en particulier:
  \begin{itemize}
  \item si un facteur \(\pi_i\) est de dimension impaire il est algébrique régulier,
  \item sinon il est demi-algébrique \(\SO\)-régulier.
  \end{itemize}
  Si on suppose qu'il y a plus d'un facteur l'hypothèse de récurrence (et le cas déjà connu de la dimension impaire) nous fournit l'existence de représentations automorphes cuspidales \(\widetilde{a^{-1}}(\pi_i)\) telles que pour presque toute place non-archimédienne \(v\) de \(F\)
  \[ c(\pi_v) = \bigoplus_i c \left( \widetilde{a^{-1}}(\pi_i)_v \right) \]
  ce qui contredit le théorème de Jacquet-Shalika.
  Ainsi il n'y a qu'un facteur, qui est la représentation \(\tilde{a}(\pi)\) souhaitée.
\end{proof}

\begin{remark}
  Comme la notation le suggère dans le deuxième point de la Proposition \ref{pro:AutC_SO_reg} si \(\pi\) est régulière la représentation \(\tilde{a}(\pi)\) coïncide avec celle introduite dans la Définition \ref{def:tilde_action} (construite grâce à la cohomologie de \(\GL_{2n,F}\)).
\end{remark}

\section{Le théorème de Rohlfs et Speh}\label{sec:RohlfsSpeh}

On généralise \cite[Proposition II.1, Theorem II.2]{RohlfsSpeh_PsEis3} aux systèmes de coefficients quelconques dans le Théorème \ref{thm:RohlfsSpeh}.
Nous n'avons pas entièrement compris la preuve de Rohlfs et Speh (en particulier la phrase ``Then \(\omega_M^* \wedge \psi(\omega_N)\) is the lowest weight vector of a representation of \(K\)'' du paragraphe 2.3 loc.\ cit.) donc nous donnons une preuve différente, inspirée du calcul de la \((\gfr,K)\)-cohomologie par Zuckerman (voir le paragraphe suivant \cite[Theorem 3.3]{VoganZuckerman}) et du premier terme de la résolution de Johnson \cite{Johnson}.

\begin{lemma} \label{lem:pi0_Q_thetast}
  Soit \(G\) un groupe réductif connexe sur \(\R\) et \(Q\) un sous-groupe parabolique de \(G_\C\) tel que \(\sigma(Q)\) est opposé à \(Q\), où \(\sigma\) désigne la conjugaison complexe.
  Soit \(L\) le sous-groupe (``Lévi tordu'') de \(G\) tel que \(L_\C = Q \cap \sigma(Q)\).
  Alors l'application naturelle \(\pi_0(L(\R)) \to \pi_0(G(\R))\) est injective.
\end{lemma}
\begin{proof}
  Commençons par supposer que le sous-groupe dérivé \(G_\der\) de \(G\) est simplement connexe.
  On sait dans ce cas que \(G_\der(\R)\) est connexe \cite[Corollaire 4.7]{BorelTits_compl}, en particulier \(\pi_0(G(\R))\) se plonge dans \(\pi_0((G/G_\der)(\R))\) (noter que \(G/G_\der\) est un tore).
  Le sous-groupe dérivé \(L_\der\) de \(L\) est également simplement connexe, d'où un diagramme commutatif
  \[
    \begin{tikzcd}
      \pi_0(L(\R)) \ar[d, hook] \ar[r] & \pi_0(G(\R)) \ar[d, hook] \\
      \pi_0((L/L_\der)(\R)) \ar[r] & \pi_0((G/G_\der)(\R))
    \end{tikzcd}
  \]
  ce qui nous ramène à montrer que la flèche horizontale du bas est injective.
  Le morphisme de tores \(L/L_\der \to G/G_\der\) est clairement surjectif, et son noyau \(C\) est un tore anisotrope: on se ramène à montrer que si \(G\) est semi-simple alors le plus grand tore central de \(L\) est anisotrope, ce qui découle facilement du fait que \(\sigma(Q)\) est opposé à \(Q\).
  En particulier \(C(\R)\) est connexe et l'application \(\pi_0((L/L_\der)(\R)) \to \pi_0((G/G_\der)(\R))\) est injective, ce qui conclut la preuve dans le cas où \(G_\der\) est supposé simplement connexe.

  Pour ramener le cas général à ce cas particulier on utilise une z-extension \(\tilde{G} \to G\) où \(\tilde{G}\) est un groupe réductif connexe sur \(\R\) de sous-groupe dérivé simplement connexe, et \(\tilde{G} \to G\) est surjective de noyau un tore central de \(\tilde{G}\) isomorphe à \(\Res_{\C/\R} \GL_{1,\C}^r\) pour un entier \(r \geq 0\).
  L'application induite \(\tilde{G}(\R) \to G(\R)\) est surjective de noyau connexe (isomorphe à \((\C^\times)^r\)), donc \(\pi_0(\tilde{G}(\R)) \simeq \pi_0(G(\R))\).
  En notant \(\tilde{Q}\) la préimage de \(Q\) dans \(\tilde{G}_\C\) et \(\tilde{L}\) celle de \(L\) dans \(\tilde{G}\), on a \(\tilde{L}_\C = \tilde{Q} \cap \sigma(\tilde{Q})\) et \(\pi_0(\tilde{L}(\R)) \simeq \pi_0(L(\R))\), donc le lemme pour \((\tilde{G}, \tilde{Q})\) implique le lemme pour \((G,Q)\).
\end{proof}

\begin{theorem}[Rohlfs-Speh] \label{thm:RohlfsSpeh}
  Soit \(G\) un groupe réductif connexe sur \(\R\), \(K\) un sous-groupe compact maximal de \(G(\R)\), \(\pi\) un \((\gfr,K)\)-module irréductible unitarisable après torsion par un caractère, \(V\) une représentation algébrique irréductible de \(G_\C\) telle que \(H^\bullet(\gfr,K^0; \pi \otimes V) \neq 0\).
  Soit \(i\) le plus petit degré pour lequel \(H^i(\gfr,K^0; \pi \otimes V) \neq 0\).
  Soit \(\pi \hookrightarrow \ind_P^G \sigma\) le plongement dans une module anti-standard (dual de la réalisation de la contragrédiente \(\pi^\vee\) comme quotient de Langlands).
  Alors l'application induite
  \[ H^i(\gfr,K^0; \pi \otimes V) \to H^i(\gfr,K^0; \ind_P^G \sigma \otimes V) \]
  est injective, et \(H^\bullet(\gfr,K^0; \ind_P^G \sigma \otimes V)\) est nulle en degré \(<i\).
\end{theorem}
\begin{proof}
  Soit \(\theta\) l'involution de Cartan de \(G\) correspondant à \(K\).
  Soit \(\pi\) un \((\gfr,K)\)-module irréductible unitarisable après torsion par un caractère tel qu'il existe un \((\gfr,K^0)\)-module irréductible de dimension finie \(V\) tel que \(H^\bullet(\gfr,K^0; \pi \otimes V) \neq 0\).
  (On ne suppose pas tout de suite que \(V\) provient d'une représentation algébrique de \(G_\C\).)
  On sait que \(\pi\) est isomorphe à \(A_\qfr(\lambda)\) où
  \begin{itemize}
  \item \(\qfr\) est une sous-algèbre parabolique \(\theta\)-stable de \(\gfr\), elle s'écrit donc \(\qfr = \lfr \oplus \ufr\) où \(\ufr\) est l'algèbre de Lie du radical unipotent \(U\) du sous-groupe parabolique \(Q\) de \(G_\C\) correspondant à \(\qfr\), \(L\) est le stabilisateur de \(\qfr\) dans \(G\) et \(\lfr = \C \otimes_\R \Lie L\);
  \item \(\lambda\) est un caractère de \(L(\R)\), également vu comme un \((\lfr,K_L)\)-module de dimension un \(\C(\lambda)\) où \(K_L = K \cap L(\R)\), vérifiant une condition de positivité: soit \(T\) un tore maximal de \(L\), pour toute racine \(\alpha\) de \(T\) dans \(\ufr\) on demande \(\Re \langle \alpha^\vee, d \lambda|_{T(\R)} \rangle \geq 0\) où \(d \lambda|_{T(\R)} \in \Hom_\R(\Lie T(\R), \C) \simeq \C X^*(T)\).
      En fait le nombre complexe \(\langle \alpha^\vee, d \lambda|_{T(\R)} \rangle\) est entier car \(\alpha^\vee \in X_*(T)\) est dans \(X_*(T_\mathrm{der})\) et l'image de \(T_\mathrm{der}\) dans \(L/L_\mathrm{der}\) est un tore anisotrope, et pour un tore réel anisotrope \(T'\) les caractères continus de \(T(\R)\) sont exactement les éléments de \(X^*(T')\);
  \item \(A_\qfr(\lambda) := \cR^S_{\qfr,K}(\C(\lambda))\) où \(S = \dim_\C \ufr \cap \kfr\) et \(\cR^\bullet_{\qfr,K}\) est le foncteur dérivé à droite du foncteur \(\cR_{\qfr,K}^0 := \cF(\res^{\lfr,K_L}_{\qfr,K_L}(-) \otimes_\C D)\) où \(\cF\) est l'adjoint à droite du foncteur de restriction des \((\gfr,K)\)-modules vers les \((\qfr,K_L)\)-modules et \(D := \wedge^{\dim \ufr} \ufr\).
  \end{itemize}
  On sait également que \(\cR^k_{\qfr,K}(\C(\lambda))\) est nul pour \(k \neq S\), plus généralement \(\cR^k_{\qfr,K}(X)=0\) pour \(k \neq S\) si \(X\) est un \((\lfr,K_L)\)-module de longueur finie dont tous les sous-quotients irréductibles ont même caractère infinitésimal que \(\C(\lambda)\) (voir \cite[Propositions IV.1 et IV.2]{Johnson}).
  En particulier la restriction du foncteur \(\cR^S_{\qfr,K}\) à cette sous-catégorie est exacte.

  Soit \(P_{L,0}\) un sous-groupe parabolique minimal de \(L\), de quotient réductif \(M_{L,0}\), de sorte que l'application évidente \(\C(\lambda) \hookrightarrow \Ind_{P_{L,0}}^L \C(\lambda|_{M_{L,0}})\) est le plongement de \(\C(\lambda)\) dans un module anti-standard.
  Le plongement \(\pi \hookrightarrow \ind_P^G \sigma\) est isomorphe à \(\cR^S_{\qfr,K}\) appliqué à \(\C(\lambda) \hookrightarrow \Ind_{P_{L,0}}^L \C(\lambda|_{M_{L,0}})\): le fait qu'appliquer \(\cR^S_{\qfr,K}\) préserve l'injectivité est un cas particulier de l'exactitude rappelée ci-dessus, et le fait que
  \[ \cR^S_{\qfr,K}(\Ind_{P_{L,0}}^L \C(\lambda|_{M_{L,0}})) \]
  est un module anti-standard résulte de \cite[Proposition IV.1]{Johnson} (qui renvoie à \cite[Theorem 8.2.4]{Vogan_book}, voir aussi \cite[Theorem 8.2]{AdamsJohnson} pour la traduction ``à la Langlands'').

  On peut calculer \(H^\bullet(\gfr,K; - \otimes_\C V)\) comme
  \[ \Ext^\bullet_{\gfr,K}(V^\vee, -) = H^\bullet(\RHom_{\gfr,K}(V^\vee, -)). \]
  Par adjonction, et grâce notamment au fait que le foncteur \(\res^{\gfr,K}_{\qfr,K_L}\) est exact, on a pour tout \(X^\bullet \in D^-(\gfr,K)\) et tout \(Y^\bullet \in D^+(\qfr,K_L)\)
  \[ \RHom_{\gfr,K}(X^\bullet, R \cR^0_{\qfr,K}(Y^\bullet)) \simeq \RHom_{\qfr,K_L}(\res^{\gfr,K}_{\qfr,K_L}(X^\bullet), Y^\bullet \otimes_\C D). \]
  On souhaite appliquer cette adjonction au cas où \(Y^\bullet = \res^{\lfr,K_L}_{\qfr,K_L}(B^\bullet)\) où \(B^\bullet \in D^+(\lfr,K_L)\).
  Le foncteur exact \(\res^{\lfr,K_L}_{\qfr,K_L}\) admet un adjoint à gauche, le foncteur \(\cG := - \otimes_{U(\ufr)} \C\) des \(\ufr\)-coinvariants.
  On en déduit pour \(A^\bullet \in D^-(\qfr,K_L)\) et \(B^\bullet \in D^+(\lfr,K_L)\)
  \[ \RHom_{\qfr,K_L}(A^\bullet, \res^{\lfr,K_L}_{\qfr,K_L}(B^\bullet)) \simeq \RHom_{\lfr,K_L}(L\cG(A^\bullet), B^\bullet) \]
  et on souhaite donc calculer
  \[ L\cG(\res^{\gfr,K}_{\qfr,K_L}(X^\bullet) \otimes_\C D^\vee) \]
  dans le cas où \(X = V^\vee\) (concentré en degré \(0\)).
  Un argument classique (utilisant notamment le fait que le foncteur de restriction envoie \((\qfr,K_L)\)-modules projectifs sur \(\ufr\)-modules projectifs) nous permet de calculer \(L\cG(V^\vee \otimes_\C D^\vee)\) (on omet dorénavant le foncteur \(\res^{\gfr,K}_{\qfr,K_L}\) pour simplifier la notation) comme étant quasi-isomorphe au complexe \(\left( \bigwedge^{-\bullet} \ufr \right) \otimes_\C V^\vee \otimes_\C D^\vee\) avec différentielles
  \begin{multline*}
    d((X_1 \wedge \dots \wedge X_k) \otimes f) = \sum_{a=1}^k (-1)^a (X_1 \wedge \dots \widehat{X_a} \dots \wedge X_k) \otimes (X_a \cdot f) \\
    + \sum_{1 \leq a < b \leq k} (-1)^{a+b} ([X_a,X_b] \wedge X_1 \wedge \dots \widehat{X_a} \dots \widehat{X_b} \dots X_k) \otimes f,
  \end{multline*}
  où \(X_a \in \ufr\) et \(f \in V^\vee \otimes_\C D^\vee\).
  Comme les termes du complexe \((\bigwedge^{-\bullet} \ufr) \otimes_\C V^\vee \otimes_\C D^\vee\) sont de dimension finie on a
  \begin{align}
    \RHom_{\lfr,K_L} \left( \left( \bigwedge^{-\bullet} \ufr \right) \otimes_\C V^\vee \otimes_\C D^\vee, B^\bullet \right)
    &\simeq \RHom_{\lfr,K_L} \left( \C, \Tot^\bullet \left( \left( \bigwedge^\bullet \ufr \right)^\vee \otimes_\C V \otimes_\C D \otimes_\C B^\bullet \right) \right) \nonumber \\
    &\simeq \RHom_{\lfr,K_L} \left( \C, \Tot^\bullet \left( C^\bullet(\ufr, V) \otimes_\C B^\bullet \right) \otimes_\C D \right) \label{eq:RHom_ufr_V_D_B}
  \end{align}
  où \(C^\bullet(\ufr, V)\) est le complexe de Chevalley-Eilenberg \(\Hom_\C(\bigwedge^\bullet \ufr, V)\).
  Supposons maintenant que \(V\) est la restriction d'une représentation algébrique irréductible \(V^G_\mu\) de \(G_\C\) de plus haut poids \(\mu\).
  Rappelons que d'après Kostant \(C^\bullet(\ufr, V)\) est quasi-isomorphe à une somme directe de \((\lfr,K_L)\)-modules irréductibles \(V^L_{w(\mu+\rho)-\rho}[\lng(w)]\) (restrictions de représentations algébriques irréductibles de \(L_\C\), décalées).
  Pour \(B^\bullet = \C(\lambda)\) (concentré en degré \(0\)) un seul de ces termes peut avoir une contribution non nulle à \eqref{eq:RHom_ufr_V_D_B}: celui pour lequel \(d \lambda|_{T(\R)}\) est algébrique (i.e.\ provient d'un élément de \(X^*(T)\)) et \(w(\mu+\rho)-\rho+2\rho = -d \lambda|_{T(\R)}\) (\(2\rho\) étant le caractère par lequel \(L\) agit sur \(D\)).
  La condition de positivité sur \(\lambda\) implique que ce \(w\) est en fait le plus long des représentants de Kostant pour \(Q=L_\C U \subset G_\C\), qui a longueur \(\dim_\C \ufr\).
  On retrouve le fait \cite[Proposition 5.4 (c)]{VoganZuckerman} que \(H^\bullet(\gfr,K; A_\qfr(\lambda) \otimes_\C V_\mu)\) est non nulle seulement si \(d \lambda|_{T(\R)}\) est la restriction d'un \(\lambda_0 \in X^*(T)\) et \(\mu = -w_0(\lambda+\rho) - \rho\), et dans ce cas
  \[ H^\bullet(\gfr,K; A_\qfr(\lambda) \otimes_\C V_\mu) \simeq H^{\bullet-R}(\lfr, K_L; \C(\lambda \lambda_0^{-1})) \]
  où \(R = \dim_\C \ufr^{\theta=-1} = \dim_\C \ufr - S\).
  Remarquons que \(\lambda \lambda_0^{-1}\) est un caractère \(L(\R) \to \{\pm 1\}\).

  On obtient aussi un diagramme commutatif
  \[
    \begin{tikzcd}
      H^\bullet(\gfr,K; A_\qfr(\lambda) \otimes_\C V_\mu) \ar[r, dash, "{\sim}"] \ar[d] & H^{\bullet-R}(\lfr, K_L; \C(\lambda \lambda_0^{-1})) \ar[d] \\
      H^\bullet(\gfr,K; (\ind_P^G \sigma) \otimes_\C V_\mu) \ar[r, dash, "{\sim}"] & H^{\bullet-R}(\lfr, K_L; \Ind_{P_{L,0}}^L \C(\lambda \lambda_0^{-1} |_{M_{L,0}}))
    \end{tikzcd}
  \]
  et lorsque \(\lambda = \lambda_0\) il est clair que la flèche verticale de droite est non nulle en degré \(R\).
  Si \(G(\R)\) est connexe (i.e.\ si \(K^0=K\)) alors \(L(\R)\) est également connexe (Lemme \ref{lem:pi0_Q_thetast}), en particulier \(\lambda=\lambda_0\) et cela conclut la preuve dans ce cas.

  Expliquons maintenant comment ramener le théorème dans le cas général au cas où \(G(\R)\) est connexe.
  Soit \(\pi\) un \((\gfr,K)\)-module irréductible unitarisable (après torsion par un caractère) tel qu'il existe une représentation irréductible algébrique \(V\) de \(G_\C\) vérifiant \(H^\bullet(\gfr,K^0; \pi \otimes V) \neq 0\).
  Notons \(G_\sico\) le revêtement simplement connexe du sous-groupe dérivé \(G_\der\) de \(G\), et \(K_\sico\) la préimage de \(K\) dans \(G_\sico(\R)\).
  On sait que \(G_\sico(\R)\) est connexe \cite[Corollaire 4.7]{BorelTits_compl}, donc \(K_\sico\) est connexe.
  La restriction \(\pi'\) de \(\pi\) à \((\gfr_\der,K_\sico)\) est semi-simple: il existe une famille finie \((\pi_j)_{j \in J}\) de \((\gfr_\der,K_\sico)\)-modules irréductibles et une décomposition canonique
  \[ \pi' \simeq \bigoplus_{j \in J} \pi'[\pi_j] \text{ où } \pi'[\pi_j] := \Hom_{{\gfr_\der,K_\sico}}(\pi_j, \pi) \otimes_\C \pi_j. \]
  On peut supposer les \(\pi_j\) non isomorphes et les composantes isotypiques \(\pi'[\pi_j]\) non nulles pour tout \(j \in J\).
  Cette hypothèse et l'irréductibilité de \(\pi\) impliquent que \(K/K^0\) agit transitivement sur l'ensemble des classes d'isomorphisme des \((\pi_j)_{j \in J}\) (autrement dit, sur \(J\)).
  Notant comme d'habitude \(A_G\) le plus grand tore déployé central de \(G\) et \(\afr_G = \C \otimes_\R \Lie A_G(\R)\), l'action de \((\afr_G, K^0 \cap Z(G(\R)))\) sur \(\pi \otimes_\C V\) est triviale et on a un isomorphisme d'espaces vectoriels gradués
  \[ H^\bullet(\gfr,K^0, \pi \otimes_\C V) \simeq H^\bullet(\gfr_\der,K_\sico; \pi' \otimes_\C V) \otimes_\C H^\bullet(\afr_G; \C), \]
  en particulier \(i\) est également le plus petit degré en lequel \(H^\bullet(\gfr_\der,K_\sico; \pi' \otimes_\C V)\) est non nulle et
  \[ H^i(\gfr,K^0, \pi \otimes_\C V) \simeq H^i(\gfr_\der,K_\sico; \pi' \otimes_\C V). \]
  On a enfin
  \[ H^i(\gfr_\der,K_\sico; \pi' \otimes_\C V) \simeq \bigoplus_{j \in J} H^i(\gfr_\der,K_\sico; \pi'[\pi_j] \otimes_\C V) \]
  et tout \(k \in K\) induit un isomorphisme entre les facteurs \(H^i(\gfr_\der,K_\sico; \pi'[\pi_j] \otimes_\C V)\) et \(H^i(\gfr_\der,K_\sico; \pi'[k^{-1}(\pi_j)] \otimes_\C V)\).
  Comme \(K/K^0\) agit transitivement sur \(J\) on en déduit que chaque \(H^i(\gfr_\der,K_\sico; \pi'[\pi_j] \otimes_\C V)\) est non nul.
  Enfin \(\pi \hookrightarrow \ind_P^G \sigma\) se restreint en \(\bigoplus_j \pi'[\pi_j] \hookrightarrow \Hom_{{\gfr_\der,K_\sico}}(\pi_j, \pi) \otimes _\C \ind_{P_\sico}^{G_\sico} \sigma_j\), où chaque induite \(\ind_{P_\sico}^{G_\sico} \sigma_j\) est anti-standard.
  Appliquant le théorème à tous les \(\pi_j\), on le déduit donc pour \(\pi\).
\end{proof}

\section{Fonction \(L\) de Rankin pour des représentations autoduales}\label{sec:FunctionRankin}

Soit \(F\) un corps de nombres, \(r,t \geq 1\) des entiers.
Soit \(\pi\) (resp.\ \(\rho\)) une représentation automorphe cuspidale autoduale pour \(\GL_{r,F}\) (resp.\ \(\GL_{t,F}\)).
On suppose que parmi \(\pi\) et \(\rho\), l'une est de type orthogonal et l'autre est de type symplectique, que celle qui est symplectique est algébrique régulière et que celle qui est orthogonale est algébrique régulière (resp.\ demi-algébrique \(\SO\)-régulière) si elle est en dimension impaire (resp.\ paire).
On suppose qu'en toute place archimédienne \(v\) de \(F\) la restriction du paramètre de Langlands de \(\pi_v\) (resp.\ \(\rho_v\)) à \(\C^\times\) est
\[ z \mapsto \diag((z/\ol{z})^{p_1}, \dots, (z/\ol{z})^{p_r}) \ \ \text{resp. } \diag((z/\ol{z})^{q_1}, \dots, (z/\ol{z})^{q_t}) \]
où \(p_1 \geq \dots \geq p_r\) (resp.\ \(q_1 \geq \dots \geq q_t\)) sont des éléments de \(\Z\) si \(\pi\) (resp.\ \(\rho\)) est orthogonale, \(\frac{1}{2} + \Z\) si \(\pi\) (resp.\ \(\rho\)) est symplectique, et les multi-ensembles
\[ \{ p_i \pm 1/2 \,|\, 1 \leq i \leq r\}, \ \ \{q_j \,|\, 1 \leq j \leq t\} \]
sont disjoints et sans multiplicité, sauf peut-être pour l'élément \(0\) qu'on autorise à avoir multiplicité \(2\).

\begin{theorem} \label{thm:equiv_Sp_SO}
  Pour tout \(a \in \Aut(\C)\) on a
  \[ L(1/2, \pi \times \rho) = 0 \Leftrightarrow L(1/2, \tilde{a}(\pi) \times \tilde{a}(\rho)) = 0. \]
\end{theorem}

Le reste de cette section est dédié à la preuve de ce théorème.
Grâce aux Propositions \ref{pro:AutC_poids_et_type_autodual} et \ref{pro:AutC_SO_reg} le rôle de \((\pi,\rho)\) et \((\tilde{a}(\pi),\tilde{a}(\rho))\) est symétrique, il suffit donc de montrer que si \(L(1/2, \pi \times \rho)\) est non nulle alors il en va de même de \(L(1/2, \tilde{a}(\pi) \times \tilde{a}(\rho))\).

Soit \(n = r + \lfloor t/2 \rfloor\) et \(G\) le groupe quasi-déployé suivant:
\begin{enumerate}
\item si \(\rho\) est symplectique, \(G = \SO_{2n+1} = \SO_{2r+t+1}\) (déployé),
\item si \(\rho\) est orthogonale de dimension impaire, \(G = \Sp_{2n} = \Sp_{2r+t-1}\) (déployé),
\item sinon (\(\rho\) orthogonale de dimension paire) \(G\) est le groupe quasi-déployé \(\SO_{2n}^\alpha\) où \(\alpha \in F^\times/F^{\times,2}\) correspond au caractère central de \(\rho\).
\end{enumerate}

Dans le cas où \(\rho\) est orthogonale de dimension impaire, son caractère central \(\omega_\rho\) est à valeurs dans \(\{\pm 1\}\) et quitte à remplacer \((\pi, \rho)\) par \(((\omega_\rho \circ \det) \otimes \pi, (\omega_\rho \circ \det) \otimes \rho)\) (ce qui ne change pas les fonctions L considérées) on peut supposer \(\omega_\rho = 1\).
Soit \(P=MN\) le sous-groupe parabolique standard (triangulaire supérieur) de \(G\) de facteur de Lévi \(M = \GL_{r,F} \times G'\) où \(G'\) est du même type que \(G\), de rang absolu \(\lfloor t/2 \rfloor\).
Soit \(\sigma\) une représentation automorphe cuspidale pour \(G'\) de paramètre d'Arthur-Langlands \(\rho\) (par exemple celle qui est générique, voir \cite[Proposition 8.3.2 (b)]{ArthurBook}).
On fixe un ``bon'' sous-groupe compact maximal \(K = \prod_v K_v\) de \(G(\A_F)\) (cf.~\cite[\S I.1.4]{MoeglinWaldspurger_bookspec}) et on note \(K_{v,M} = K_v \cap M(F_v)\).
On note aussi \(K_{\infty} = \prod_{v | \infty} K_v\) et \(K_{\infty,M} = \prod_{v|\infty} K_{v,M}\).
Pour une place archimédienne \(v\) de \(F\) on note \(\mfr_v = \C \otimes_\R \Lie M(F_v)\), et on note \(\mfr = \prod_{v|\infty} \mfr_v\).
On fixe un produit tensoriel restreint \(\pi \otimes \sigma = \bigotimes'_v \pi_v \otimes \sigma_v\) et un plongement \(j: \pi \otimes \sigma \hookrightarrow \cA_\cusp(M(F) \backslash M(\A_F))\).
Ici \(\pi_v \otimes \sigma_v\) est un \((\mfr,K_{\infty,M})\)-module irréductible (resp.\ une représentation lisse irréductible de \(M(F_v)\)) si \(v\) est archimédienne (resp.\ non-archimédienne).
En dehors d'un ensemble fini \(S\) de places (contenant toutes les places archimédiennes et celles où \(G\) est ramifié) la représentation \(\pi_v \otimes \sigma_v\) est non ramifiée, i.e.\ la dimension du sous-espace des vecteurs \(K_{v,M}\)-sphériques est \(1\), et on a fixé une base \(e_v\) de cette droite pour former le produit tensoriel restreint.
On considère pour \(s \in \C\) l'induite parabolique (\(K\)-finie et normalisée) \(\ind_P^G \pi[s] \otimes \sigma\).
L'espace de cet induite ne dépend pas de \(s \in \C\): le morphisme de restriction à \(K\) induit un isomorphisme avec l'induite \(K\)-finie \(\Ind_{K_M}^K (\pi \otimes \sigma)|_{K_M}\).
Si \(v\) est une place non-archimédienne telle que \(G_{F_v}\) est non ramifié et \(\pi_v \otimes \sigma_v\) est non ramifiée (pour \(K_{v,M} = K_v \cap M(F_v)\)) on note \(f^\mathrm{nr}_{P,v,s}\) l'élément \(K_v\)-invariant de \(\ind_{P(F_v)}^{G(F_v)} \pi[s] \otimes \sigma\) valant \(e_v\) en \(1 \in G(F_v)\).
L'induite adélique \(\ind_P^G \pi[s] \otimes \sigma\) s'identifie au produit tensoriel restreint des induites locales \(\ind_{P(F_v)}^{G(F_v)} \pi_v[s] \otimes \sigma_v\) (\(\ind_{\pfr_v, K_{v,M}}^{\gfr_v,K_v}\) aux places archimédiennes \(v\)).
On a un plongement de \((\gfr,K_\infty,G(\A_{F,\ff}))\)-modules
\begin{align*}
  j_{P,s}: \ind_P^G \pi[s] \otimes \sigma & \longrightarrow \cA_\cusp(N(\A_F) M(F) \backslash G(\A)) \\
  f & \longmapsto (mnk \mapsto \delta_P^{1/2}(m) |\det m_\GL|^s j(f(k))(m))
\end{align*}
où \(m = (m_\GL, m_{G'}) \in M(\A_F)\), \(n \in N(\A_F)\) et \(k \in K\).
Il existe \(s_0 \in \R\) tel que pour tout \(s \in \C\) satisfaisant \(\Re s > s_0\) et tout \(f \in \ind_P^G \pi[s] \otimes \sigma\) et tout compact \(C\) de \(G(\A_F)\) la série d'Eisenstein
\[ E_P^G(f, s)(g) := \sum_{\gamma \in P(F) \backslash G(F)} j_{P,s}(f)(\gamma g) \]
converge normalement pour \(g \in C\), et définit une forme automorphe \(E_P^G(f, s) \in \cA(G)\).
Le sous-groupe parabolique \(\ol{P} = M \ol{N}\) de \(G\) opposé à \(P\) est conjugué à \(P\): il existe \(w \in G(F) \smallsetminus M(F)\) normalisant \(M\), et on a \(\Ad(w) \ol{N} = N\).
On a un isomorphisme
\begin{align*}
  L_{w^{-1}}: \cA_\cusp(\ol{N}(\A_F) M(F) \backslash G(\A_F)) & \longrightarrow \cA_\cusp(N(\A_F) M(F) \backslash G(\A_F)) \\
  h & \longmapsto (g \mapsto h(w^{-1} g))
\end{align*}
qui d'ailleurs ne dépend pas du choix de \(w\).
Pour une forme automorphe \(h \in \cA(G)\) on note \(C^G_P h \in \cA(N(\A_F) M(F) \backslash G(\A_F))\) son terme constant (pour \(P\)):
\[ (C^G_P h)(g) := \int_{N(F) \backslash N(\A_F)} h(ng) dn \]
où la mesure de Haar choisie sur \(N(\A_F)\) est celle donnant \(\vol( N(F) \backslash N(\A_F)) = 1\).
D'après \cite[\S II.1.7]{MoeglinWaldspurger_bookspec} on a pour \(\Re s > s_0\)
\begin{equation} \label{eq:terme_cst_Eis_Rankin_autodual}
  C^G_P (E_P^G(f, s)) = j_{P,s}(f) + L_{w^{-1}} M(s) j_{P,s}(f)
\end{equation}
où
\[ (M(s) f)(g) := \int_{\ol{N}(\A_F)} f(\ol{n} g) d\ol{n}. \]
(Cette intégrale est convergente pour \(\Re s > s_0\) et \(f\) dans l'image de \(j_{P,s}\), et définit un élément de \(\cA_\cusp(\ol{N}(\A_F) M(F) \backslash G(\A))\).)
Les séries d'Eisenstein \(E_P^G(f,s)\) et les opérateurs d'entrelacement \(M(s)\) admettent un prolongement méromorphe à \(\C\) \cite[\S IV]{MoeglinWaldspurger_bookspec}, \cite{BernsteinLapid}.

\begin{lemma} \label{lem:alt_Lhalf_res}
  \begin{itemize}
  \item Si \(L(1/2, \pi \times \rho)\) est nulle alors \(M(s)\) est holomorphe en \(s=1/2\),
  \item sinon \((s-1/2)E_P^G(-,s)\) est holomorphe en \(s=1/2\), sa valeur en \(s=1/2\) est un morphisme non identiquement nul
    \begin{equation} \label{eq:residu_Eis}
      \beta: \ind_P^G \pi[1/2] \otimes \sigma \longrightarrow \cA(G).
    \end{equation}
    d'image incluse dans le sous-espace \(\cA^2(G)\) des formes automorphes de carré intégrable, et il existe un morphisme non nul
    \[ N: \ind_P^G \pi[1/2] \otimes \sigma \longrightarrow \ind_{\ol{P}}^G \pi[1/2] \otimes \sigma\]
    tel que \(C_{\ol{P}}^G \circ \beta = j_{\ol{P}, 1/2} \circ N\).
  \end{itemize}
\end{lemma}

\begin{proof}
  D'après \cite[Théorème IV.1.1]{Waldspurger_Plancherel} en toute place non-archimédienne \(v\) de \(F\) est défini pour \(\Re(s) \gg 0\) un opérateur d'entrelacement
  \begin{align*}
    J_{\ol{P}|P,v,s}: \ind_{P(F_v)}^{G(F_v)} \pi_v[s] \otimes \sigma_v & \longrightarrow \ind_{P(F_v)}^{G(F_v)} \pi_v[s] \otimes \sigma_v \\
    f_v & \longmapsto (g \mapsto \int_{\ol{N}(F_v)} f_v(\ol{n} g) d \ol{n}).
  \end{align*}
  Si \(G_{F_v}\) et \(\pi_v \otimes \sigma_v\) sont non ramifiés alors \(J_{\ol{P}|P,v,s}\) envoie \(f^\mathrm{nr}_{P,v,s}\) sur
  \begin{equation} \label{eq:intertw_nr}
    \frac{L(s, \pi_v \times \rho_v) L(2s, \pi_v, R)}{L(s+1, \pi_v \times \rho_v) L(2s+1, \pi_v, R)} f^\mathrm{nr}_{\ol{P},v,s}
  \end{equation}
  où \(R = \bigwedge^2\) (resp.\ \(\Sym^2\)) si \(G\) est un groupe symplectique ou orthogonal pair (resp.\ orthogonal impair).
  Cette formule découle de \cite[Proposition 4.3.1]{Shahidi_EisL}\footnote{Des représentations \emph{duales} \(\tilde{r}_i\) apparaissent dans l'énoncé de Shahidi, ce n'est pas le cas pour nous car n'utilisons pas la même convention pour l'isomorphisme de Satake: pour un tore déployé \(S\) sur \(F_v\) Shahidi identifie \(X_*(S)\) à \(S(F_v)/S(\cO_{F_v})\) en envoyant \(\mu\) sur la classe de \(\mu(\varpi_v)^{-1}\), où \(\varpi_v\) est une uniformisante, tandis que nous utilisons l'identification opposée.
    De toutes façons les représentations \(\pi_v\) et \(\rho_v\) sont autoduales \dots
    Voir aussi la Remarque \ref{rem:Satake_ind_gen}.}.
  On dispose également \cite[\S 10]{Wallach_book2} d'opérateurs d'entrelacement aux places archimédiennes entre induites \(K_\infty\)-finies (la définition impose de travailler avec des représentations de \(\prod_{v|\infty} M(F_v)\) plutôt qu'avec des \((\mfr,K_{\infty,M})\)-modules, on se contente de restreindre les opérateurs d'entrelacement aux sous-espaces de vecteurs \(K_{\infty}\)-finis).
  Pour \(\Re s \gg 0\) le produit sur presque toutes les places \(v\) des facteurs eulériens apparaissant dans \eqref{eq:intertw_nr} sont absolument convergents et le diagramme suivant est commutatif.
  \[
    \begin{tikzcd}
      \ind_P^G \pi[s] \otimes \sigma \ar[r, "{j_{P,s}}"] \ar[d, "{\prod_v J_{\ol{P}|P,v,s}}"] & \cA_\cusp(N(\A_F) M(F) \backslash G(\A)) \ar[d, "{M(s)}"] \\
      \ind_{\ol{P}}^G \pi[s] \otimes \sigma \ar[r, "{j_{\ol{P},s}}"] & \cA_\cusp(\ol{N}(\A_F) M(F) \backslash G(\A))
    \end{tikzcd}
  \]
  Les opérateurs d'entrelacement locaux \(J_{\ol{P}|P,v,s}\), initialement définis pour \(\Re s \gg 0\), admettent un prolongement méromorphe à \(\C\) (cf.\ \cite[Théorème IV.1.1]{Waldspurger_Plancherel}, \cite[\S 10]{Wallach_book2}).
  Les fonctions \(L(s, \pi \times \rho)\) et \(L(s, \pi, R)\) sont définies par des produits absolument convergents pour \(\Re(s) > 1\), avec prolongement méromorphe à \(\C\).
  Pour une place non-archimédienne \(v \in S\) on écrit \(J_{\ol{P}|P,v,s} = r(s, \pi_v, \rho_v) N_v(s)\) où
  \[ r(s, \pi_v, \rho_v) := \frac {L(s, \pi_v \times \rho_v) L(2s, \pi_v, R)}{L(s+1, \pi_v \times \rho_v)L(2s+1, \pi_v, R)} \epsilon(s, \pi_v \times \rho_v) \epsilon(s, \pi_v, R). \]
  (Les facteurs epsilon ne jouent aucun rôle dans la suite, en fait en \(s=1/2\) seul le facteur \(L(s, \pi_v \times \rho_v)\) peut présenter un pôle.)
  D'après la Proposition \ref{pro:entrelac_hol_nonnul} (voir aussi \S \ref{sec:app_autres_gpes}) l'opérateur d'entrelacement normalisé \(N_v(s)\) est holomorphe et non identiquement nul en \(s=1/2\).
  À une place archimédienne \(v\) de \(F\) la représentation \(\pi_v \otimes \sigma_v\) de \(M(F_v)\) est tempérée et pour \(\Re(s)>0\) l'opérateur d'entrelacement \(J_{\ol{P}|P,v,s}\) est défini par des intégrales intégrales absolument convergentes (uniformément en \(s\) dans un compact de \(\Re^{-1}(]0,+\infty[)\)) \cite[Theorem 7.22]{Knapp_book} et il est donc holomorphe sur ce domaine, et \(J_{\ol{P}|P,v,s}\) n'est pas identiquement nul.
  On a donc une égalité d'opérateurs (méromorphes en la variable \(s\))
  \begin{equation} \label{eq:intertw_norm_prod}
    \prod_v J_{\ol{P}|P,v,s} = \frac {L(s, \pi \times \rho) L(2s, \pi, R)}{L(s+1, \pi \times \rho)L(2s+1, \pi, R)} \epsilon(s, \pi \times \rho) \epsilon(s, \pi, R) N(s)
  \end{equation}
  où \(s \mapsto N(s)\) est holomorphe sur \(\{ s \in \C \,|\, \Re(s) \geq 1/2 \}\) et \(N(s)\) n'est pas identiquement nul pour \(\Re s = 1/2\).
  Remarquons que par hypothèse \(L(s,\pi,R)\) admet un pôle simple en \(s=1\), et les autres facteurs \(L\) et \(\epsilon\) apparaissant dans \eqref{eq:intertw_norm_prod} sont holomorphes en \(s=1/2\).
  On voit déjà que si \(L(1/2, \pi \times \rho) = 0\) alors \(M(s)\) est holomorphe en \(s=1/2\).
  Supposons maintenant \(L(1/2, \pi \times \rho)\) non nulle.
  Alors pour toute \(f \in \ind_P^G \pi[s] \otimes \sigma\) la fonction \(s \mapsto (s-1/2) M(s) j_{P,s}(f)\) est holomorphe en \(s=1/2\).
  Le fait que \(N(1/2)\) n'est pas identiquement nul implique l'existence de \(f \in \ind_P^G \pi[s] \otimes \sigma\) telle que \(s \mapsto (s-1/2) M(s) j_{P,s}(f)\) soit non nulle en \(s=1/2\).
  On en déduit grâce à \cite[Remark IV.4.4]{MoeglinWaldspurger_bookspec} que \(s \mapsto (s-1/2) E_P^G(f,s)\) est holomorphe en \(s=1/2\) et induit un morphisme non  nul \(\ind_P^G \pi[1/2] \otimes \sigma \longrightarrow \cA(G)\).
  D'après la construction de Langlands du spectre résiduel par résidus \cite[\S V]{MoeglinWaldspurger_bookspec} ce morphisme arrive dans \(\cA^2(G)\).
  Dans notre situation particulièrement simple on peut plus directement vérifier le critère sur les exposants \cite[\S I.4.11]{MoeglinWaldspurger_bookspec}.

  L'opérateur \(N\) de la deuxième partie du lemme est, à un élément de \(\C^\times\) près, l'opérateur \(N(1/2)\): l'égalité \(C_{\ol{P}}^G \circ \beta = j_{\ol{P},1/2} \circ N\) découle de la formule \eqref{eq:terme_cst_Eis_Rankin_autodual}, du diagramme commutatif ci-dessus et de la formule \(C_{\ol{P}}^G = L_w \circ C_P^G\).
\end{proof}

On suppose dans la suite que \(L(1/2, \pi \times \rho)\) est non nulle, on doit donc montrer qu'il en va de même de \(L(1/2, \tilde{a}(\pi) \times \tilde{a}(\rho))\).
Il existe un facteur irréductible \(\pi_G\) de l'image de \(\beta\) \eqref{eq:residu_Eis} qui est non ramifié en dehors de \(S\).
Pour une place archimédienne \(v\) le \((\gfr_v,K_v)\)-module \(\pi_{G,v}\) est le quotient de Langlands de \(\ind_{\pfr_v,K_{v,M}}^{\gfr_v,K_v} \pi_v[1/2] \otimes \sigma_v\) (car c'est son seul quotient irréductible).
Le caractère infinitésimal de \(\pi_{G,v}\) est algébrique régulier et \(\pi_{G,v}\) est unitaire, donc elle est cohomologique (d'après \cite{SalamancaRiba} et \cite{VoganZuckerman}): il existe une (unique) représentation algébrique irréductible \(V_v\) de \(G(\C \otimes_{\R} F_v)\) telle que \(H^\bullet(\gfr_v,K_v^0; \pi_{G,v} \otimes V) \neq 0\).
Soit \(q_v\) le degré minimal en lequel cette cohomologie est non nulle.
Soient \(q = \sum_{v | \infty} q_v\) et \(\cX = G(\R \otimes_{\Q} F)/K_\infty^0\).
On voit \(\otimes_{v|\infty} V_v\) comme une représentation irréductible de \((\Res_{F/\Q} G)_{\C}\).
Soit \((V, E \xrightarrow{\iota_0} \C))\) la paire associée par la Proposition \ref{pro:irr_alg_rep_qs} à la représentation irréductible \(\otimes_{v|\infty} V_v\) de \((\Res_{F/\Q} G)_{\C}\), en particulier \(\otimes_{v|\infty} V_v \simeq \C \otimes_{\iota_0,E} V\).
La représentation \(\pi_{G,\ff}\) intervient donc avec multiplicité non nulle dans
\begin{equation} \label{eq:decomp_g_K_coh_V_iota}
 H^q(\gfr,K_\infty^0; \cA^2(G) \otimes_{\C} V_{\C}) \simeq \bigoplus_{\iota: E \to \C} H^q(\gfr,K_\infty^0; \cA^2(G) \otimes_{\C} (\C \myotimes{\iota,E} V)), 
 \end{equation}
plus précisément elle intervient dans le facteur à droite correspondant à \(\iota_0\).
D'après le Théorème \ref{thm:Nair_Rai} la représentation \(\pi_{G,\ff}\) intervient donc dans
\[ \C \otimes_{\Q} W^{\pm} H^q(G, \cX; V) \simeq \bigoplus_{\iota: E \hookrightarrow \C} \C \myotimes{\iota,E} W^{\pm} H^q(G, \cX; V), \]
de nouveau dans le terme correspondant à \(\iota_0\) car l'isomorphisme \eqref{eq:Nair_Rai}, étant fonctoriel en \(V\), est \(E\)-équivariant.
Le Théorème \ref{thm:Nair_Rai} implique aussi que cette représentation de \(G(\A_{F,\ff})\) est semi-simple.

\begin{definition} \label{def:geneigen_colim}
  Soient \(S_0\) un ensemble fini de places de \(F\) contenant toutes les places archimédiennes et \(\tau \simeq \otimes'_v \tau_v\) une représentation admissible irréductible de \(M(\A_F^{S_0})\).

  Pour une représentation complexe admissible \(U\) de \(M(\A_{F,\ff})\) soit \(U_{\chi(\tau)} \subset U\) le sous-espace des \(u \in U\) tels qu'il existe un ensemble fini \(S \supset S_0\) de places de \(F\) et un sous-groupe compact ouvert \(K^S = \prod_{v \not\in S} K_{M,v}\) de \(M(\A_F^S)\) satisfaisant les conditions suivantes:
  \begin{itemize}
  \item pour toute place \(v \not\in S\) le sous-groupe compact ouvert \(K_{M,v}\) de \(M(F_v)\) est hyperspécial,
  \item en notant \(\tau^S = \otimes'_{v \not\in S} \tau_v\) on a \((\tau^S)^{K^S} \neq 0\), et \(K^S\) laisse \(u\) invariant,
  \item en notant \(\chi^S(\tau^S)\) le caractère de l'algèbre de Hecke (commutative) \(\cH_\C(M(\A_F^S), K^S)\) par lequel elle agit sur la droite \((\tau^S)^{K^S}\), il existe un entier \(N \geq 0\) tel que pour tout \(\varphi \in \cH_\C(M(\A_F^S), K^S)\) on a
  \[ (\varphi - \chi^S(\tau^S)(\varphi) \id)^N u = 0. \]
  \end{itemize}

  On définit aussi
  \[ \cA(\ol{N}(\A_F) M(F) \backslash G(\A_F))_{\chi(\tau)} \subset \cA(\ol{N}(\A_F) M(F) \backslash G(\A_F)) \]
  le sous-espace des formes automorphes \(h: \ol{N}(\A_F) M(F) \backslash G(\A_F) \to \C\) telles que pour tout \(k \in K\) la forme automorphe
  \begin{align*}
    M(F) \backslash M(\A_F) & \longrightarrow \C \\
    m & \longmapsto h(mk)
  \end{align*}
  soit dans \(\cA(M)_{\chi(\tau)}\).
\end{definition}

Noter que dans la Définition \ref{def:geneigen_colim} le sous-espace \(U_{\chi(\tau)}\) est une sous-représentation de \(U\), et l'application naturelle
\[ \bigoplus_{[\tau]} U_\tau \longrightarrow U \]
est injective, où la somme directe porte sur les classes d'équivalences de représentations admissibles irréductibles \(\tau\), deux représentations \(\tau\) et \(\tau'\) étant équivalentes si et seulement si \(\tau_v \simeq \tau'_v\) pour presque toute place \(v\).
(On peut penser à \(\chi(\tau)\) comme la classe d'équivalence de \(\tau\).)
En outre si \(U\) est limite directe de représentations admissibles de \(M(\A_{F,\ff})\) (par exemple si \(U\) est admissible ou si \(U = \cA(M)\)), alors cette application est surjective.
En particulier on dispose alors de projections canoniques \(U \to U_{\chi(\tau)}\).

On en déduit une décomposition canonique du \((\gfr,K) \times G(\A_{F,\ff})\)-module
\begin{equation} \label{eq:decomp_cA_Nbar_chi_tau}
  \cA(\ol{N}(\A_F) M(F) \backslash G(\A_F)) = \bigoplus_{[\tau]} \cA(\ol{N}(\A_F) M(F) \backslash G(\A_F))_{\chi(\tau)}.
\end{equation}

\begin{lemma} \label{lem:trad_piGf_res_bord}
  L'image par restriction aux strates de bord correspondant à \(\ol{P}\) de la composante \(\pi_{G,\ff}\)-isotypique
  \begin{equation} \label{eq:comp_isot}
    \left( \C \myotimes{\iota_0,E} W^{\pm} H^q(G, \cX; V) \right)[\pi_{G,\ff}]
  \end{equation}
  est non nulle, plus précisément il existe des entiers positifs \(a,b \geq 0\) tels que l'image de \eqref{eq:comp_isot} dans
  \[ \Ind_{\ol{P}(\A_{F,\ff})}^{G(\A_{F,\ff})} \left( \left( \C \myotimes{\iota_0,E} H^a(M, \cX_{\ol{P}}; H^b(\Lie \ol{N}, V)) \right)_{\chi(\delta_{\ol{P},\ff}^{1/2} \otimes (\pi_{\ff}[1/2] \otimes \sigma_{\ff}))} \right) \]
  est non nulle.
\end{lemma}
\begin{proof}
  Cela résulte du diagramme commutatif \eqref{eq:diag_BSres_cstterm} (pour \(\ol{P}\)), de la factorisation \(C_{\ol{P}}^G \circ \beta = j_{\ol{P},1/2} \circ N\) du Lemme \ref{lem:alt_Lhalf_res}, et (crucialement) du Théorème \ref{thm:RohlfsSpeh}.
\end{proof}

\begin{lemma} \label{lem:trad_a_piGf_res_bord}
  L'image de la composée
  \begin{align*}
    &\ \cA^2(G)[a(\pi_{G,\ff})] \\
    \xrightarrow{C_{\ol{P}}^G} &\ \cA(\ol{N}(\A_F) M(F) \backslash G(\A_F)) \\
    \twoheadrightarrow &\ \cA(\ol{N}(\A_F) M(F) \backslash G(\A_F))_{\chi(a(\delta_{\ol{P},\ff}^{1/2} \otimes (\pi_{\ff}[1/2] \otimes \sigma_{\ff})))}
  \end{align*}
  est non nulle, où la deuxième flèche est la projection donnée par \eqref{eq:decomp_cA_Nbar_chi_tau}.
\end{lemma}
\begin{proof}
  La restriction aux strates de bord et l'induction parabolique non normalisée étant définies sur \(\Q\), on déduit du Lemme \ref{lem:trad_piGf_res_bord} que l'image de la composante isotypique
  \[ \left( \C \myotimes{a^{-1} \iota_0,E} W^{\pm} H^q(G, \cX; V) \right)[a(\pi_{G,\ff})] \]
  dans
  \[ \Ind_{\ol{P}(\A_{F,\ff})}^{G(\A_{F,\ff})} \left( \left( \C \myotimes{a^{-1} \iota_0,E} H^a(M, \cX_{\ol{P}}; H^b(\Lie \ol{N}, V)) \right)_{\chi(a(\delta_{\ol{P},\ff}^{1/2} \otimes (\pi_{\ff}[1/2] \otimes \sigma_{\ff})))} \right) \]
  est également non nulle.
  On conclut en utilisant à nouveau le Théorème \ref{thm:Nair_Rai} et le diagramme \eqref{eq:diag_BSres_cstterm}.
\end{proof}

\begin{lemma} \label{lem:a_delta_pi_sigma}
  Pour une place non-archimédienne \(v\) de \(F\) on a
  \[ a(\delta_{\ol{P},v}^{1/2} \otimes (\pi_v[1/2] \otimes \sigma_v)) = \delta_{\ol{P},v}^{1/2} \otimes (\tilde{a}(\pi)_v[1/2] \otimes a(\sigma_v)). \]
\end{lemma}
\begin{proof}
  Un petit calcul au cas par cas nous donne
  \[ \delta_{\ol{P},v}^{1/2} \otimes (\pi_v[1/2] \otimes \sigma_v) = \pi_v[\delta/2-n+r/2] \otimes \sigma_v \]
  où
  \[ \delta =
    \begin{cases}
      1 & \text{ si } G = \SO_{2n+1}, \\
      0 & \text{ si } G = \Sp_{2n}, \\
      2 & \text{ si } G = \SO_{2n}^\alpha.
    \end{cases} \]
  Si \(\hat{G}\) est orthogonal, c'est-à-dire si \(G = \Sp_{2n}\) ou \(\SO_{2n}^\alpha\), alors \(r\) et \(\delta\) sont pairs et \(\pi\) est algébrique, donc
  \[ a(\pi_v[\delta/2-n+r/2]) = a(\pi_v)[\delta/2-n+r/2] = \tilde{a}(\pi)_v [\delta/2-n+r/2]. \]
  Si \(G = \SO_{2n+1}\) alors \(\pi\) est algébrique (resp.\ demi-algébrique) si \(r\) est impair (resp.\ pair), et quelle que soit la parité de \(r\) on conclut
  \[ a(\pi_v[\delta/2-n+r/2]) = \tilde{a}(\pi)_v [\delta/2-n+r/2]. \]
\end{proof}

\begin{lemma} \label{lem:a_sigma_cusp}
  Le sous-espace \(\cA(G')_{\chi(a(\sigma_{\ff}))}\) de \(\cA(G')\) est inclus dans \(\cA_\cusp(G')\), et toute sous-représentation irréductible de \(\cA(G')_{\chi(a(\sigma_{\ff}))}\) a pour paramètre d'Arthur \(\tilde{a}(\rho)\).
\end{lemma}
\begin{proof}
  On a une décomposition canonique (similaire à \ref{eq:decomp_cA_Nbar_chi_tau})
  \[ \cA(G')/\cA_\cusp(G') = \bigoplus_{[\tau]} (\cA(G')/\cA_\cusp(G'))_{\chi(\tau)}. \]
  Considérons une classe d'équivalence \([\tau]\) telle que le facteur associé soit non nul.
  Il résulte de \cite[Corollary 1 p.236]{Franke_weighted} qu'il existe un sous-groupe de Lévi propre \(L\) de \(G'\), un ensemble fini \(S' \supset S\) de places de \(F\) et une représentation automorphe cuspidale \(\tau_L\) pour \(L\) non ramifiée hors de \(S'\) tels que pour tout \(v \not\in S'\) la représentation \(\tau_v\) est non ramifiée et son paramètre de Satake \(c(\tau_v)\) est l'image de \(c(\tau_{L,v})\) par l'application induite par \({}^L L \hookrightarrow {}^L G'\) sur les classes de conjugaison sous \(\hat{L}\) et \(\hat{G'}\).
  Comme \(L\) s'identifie à un produits de groupes linéaires et d'un groupe classique quasi-déployé du même type que \(G'\), la classification endoscopique d'Arthur \cite{ArthurBook} implique l'existence d'une famille \((\eta_i)_{i \in I}\) de représentations automorphes cuspidales (pas forcément unitaires) de groupes linéaires non ramifiées hors de \(S'\) telles que pour toute place \(v \not\in S'\) on ait l'égalité de classes de conjugaisons dans \(\GL_N(\C)\)
  \begin{equation} \label{eq:std_LGprime_chi}
    \Std_{{}^L G'}(c(\tau_{L,v})) = \bigoplus_{i \in I} c(\eta_{i,v}).
  \end{equation}
  Le fait que \(L\) est propre implique \(|I| \geq 3\).

  Par ailleurs d'après les Exemples \ref{exa:fonct_Sat_tordu} et \ref{exa:fonct_Sat_tordu_SO_impair} pour \(v \not\in S\) on a \(c(a(\sigma_v)) = q_v^{-\epsilon/2} a(q_v^{\epsilon/2} c(\sigma_v))\) où \(\epsilon=0\) (resp.\ \(\epsilon=1\)) si \(G=\Sp_{2n}\) ou \(\SO_{2n}^\alpha\) (resp.\ \(G = \SO_{2n+1}\)).
  On en déduit
  \begin{equation} \label{eq:std_LGprime_asigma}
    \Std_{{}^L G'}(c(a(\sigma_v))) = q_v^{-\epsilon/2} a(q_v^{\epsilon/2} \Std_{{}^L G'}(c(\sigma_v))) = q_v^{-\epsilon/2} a(q_v^{\epsilon/2} c(\rho_v)) = c(\tilde{a}(\rho)_v)
  \end{equation}
  où la dernière égalité utilise l'Exemple \ref{exa:fonct_Sat_tordu} pour un groupe linéaire (et la définition de \(\tilde{a}(\rho)\) au cas par cas).

  Comparant \eqref{eq:std_LGprime_chi} et \eqref{eq:std_LGprime_asigma} on conclut grâce au théorème de Jacquet-Shalika que le caractère \(\chi(a(\sigma_{\ff}))\) n'intervient pas dans \(\cA(G')/\cA_\cusp(G')\).
  On a donc \(\cA_\cusp(G')_{\chi(a(\sigma_{\ff}))} = \cA(G')_{\chi(a(\sigma_{\ff}))}\).
  Le fait que toute sous-représentation irréductible de \(\cA^2(G')_{\chi(a(\sigma_{\ff}))}\) a pour paramètre d'Arthur \(\tilde{a}(\rho)\) résulte de l'égalité \eqref{eq:std_LGprime_asigma}.
\end{proof}

\begin{lemma} \label{lem:cA_geneigen_subset_cusp}
  Le sous-espace (introduit dans la Définition \ref{def:geneigen_colim})
  \[ \cA(\ol{N}(\A_F) M(F) \backslash G(\A_F))_{\chi(a(\delta_{\ol{P},\ff}^{1/2} \otimes (\pi_{\ff}[1/2] \otimes \sigma_{\ff})))} \]
  de \(\cA(\ol{N}(\A_F) M(F) \backslash G(\A_F))\) est inclus dans \(\cA_\cusp(\ol{N}(\A_F) M(F) \backslash G(\A_F))\).
\end{lemma}
\begin{proof}
  Vus les Lemmes \ref{lem:a_delta_pi_sigma} et \ref{lem:a_sigma_cusp} il reste à montrer l'inclusion
  \[ \cA(\GL_{r,F})_{\chi(\tilde{a}(\pi)_{\ff})} \subset \cA_\cusp(\GL_{r,F}) \]
  et la preuve est similaire à celle du Lemme \ref{lem:a_sigma_cusp} (et même un peu plus simple puisqu'on ne fait pas appel à la classification d'Arthur).
\end{proof}

Rappelons qu'on a \cite[\S III.2.6]{MoeglinWaldspurger_bookspec} une décomposition par support cuspidal de l'espace des formes automorphes pour \(G\)
\[ \cA(G) = \bigoplus_D \cA(G)_D \]
où la somme porte sur les données cuspidales \(D = [L,\pi_L]\), i.e.\ les classes d'équivalence de paires \((L,\pi_L)\) où \(L\) est un sous-groupe de Lévi de \(G\) et \(\pi_L\) est un \((\lfr,K_L) \times L(\A_{F,\ff})\)-module admissible irréductible apparaissant dans \(\cA_\mathrm{cusp}(L)\), pour la relation d'équivalence donnée par la conjugaison par \(G(F)\) et l'équivalence de \((\lfr,K_L) \times L(\A_{F,\ff})\)-modules.
Comme expliqué en \cite[\S III.3.3]{MoeglinWaldspurger_bookspec} cette décomposition passe au sous-espace \(\cA^2(G) \subset \cA(G)\) des formes automorphes de carré intégrable:
\[ \cA^2(G) = \bigoplus_D \cA^2(G)_D \ \ \ \text{où } \cA^2(G)_D = \cA^2(G) \cap \cA(G)_D. \]

Combinant les Lemmes \ref{lem:trad_a_piGf_res_bord}, \ref{lem:a_delta_pi_sigma} et \ref{lem:cA_geneigen_subset_cusp} on obtient l'existence d'un \((\mfr,K_{\infty,M}, M(\A_{F,\ff}))\)-module admissible irréductible \(\tilde{a}(\pi)[1/2] \otimes \sigma'\) intervenant dans \(\cA_\mathrm{cusp}(M)\), où \(\sigma'\) a paramètre d'Arthur \(a(\rho)\), et d'une sous-représentation irréductible \(\pi'_G\) de \(\cA^2(G)_{[M, \tilde{a}(\pi)[1/2] \otimes \sigma']}\) telle que \(\pi'_{G,\ff} \simeq a(\pi_{G,\ff})\).
D'après la construction de Langlands du spectre résiduel par résidus (voir \cite[\S V]{MoeglinWaldspurger_bookspec}) on conclut que \(\pi'_G\) intervient avec multiplicité non nulle dans l'image du résidu en \(s=1/2\) des séries d'Eisenstein (méromorphes en \(s\))
\[ E(-,s): \ind_P^G (\tilde{a}(\pi)[s] \otimes \sigma') \longrightarrow \cA(G). \]
Appliquant le Lemme \ref{lem:alt_Lhalf_res} à \((\tilde{a}(\pi),\tilde{a}(\rho),\sigma')\), on déduit
\[ L(1/2, \tilde{a}(\pi) \times \tilde{a}(\rho)) \neq 0 \]
et cela conclut la preuve du Théorème \ref{thm:equiv_Sp_SO}.

\begin{remark}
  Pour pouvoir appliquer le Lemme \ref{lem:alt_Lhalf_res} à \((\tilde{a}(\pi),\tilde{a}(\rho),\sigma')\) on peut utiliser la Proposition \ref{pro:AutC_poids_et_type_autodual} pour voir que \(\tilde{a}(\pi)\) est autoduale du même type que \(\pi\).
  Ce n'est en fait pas nécessaire: la preuve du Lemme \ref{lem:alt_Lhalf_res} montre que si la série d'Eisenstein a un pôle en \(s=1/2\) alors \(L(1/2,\tilde{a}(\pi) \times \tilde{a}(\rho)) \neq 0\) et \(L(s, \tilde{a}(\pi), R)\) a un pôle en \(s=1\).
\end{remark}

\begin{remark} \label{rem:Arthur_pour_Gprime}
  La preuve du Théorème \ref{thm:equiv_Sp_SO} n'utilise la classification endoscopique d'Arthur que pour le facteur \(G'\) du sous-groupe de Lévi \(M\) de \(G\), pas pour \(G\) lui-même.
  On peut vérifier que lorsque \(G'\) est trivial ou isomorphe à \(\SO_2^\alpha\) ou à \(\SO_3 \simeq \PGL_2\), ce qui revient à \(t \leq 2\), la preuve n'utilise pas du tout les résultats d'Arthur pour les groupes classiques.
  (Il faut également vérifier que la preuve de la Proposition \ref{pro:entrelac_hol_nonnul} ne les utilise pas; c'est bien le cas).
\end{remark}

\newcommand{\BC}{\textup{BC}}
\newcommand{\ut}{{}^t}
\newcommand{\antidiagTwo}[2]{{ \lhk \begin{matrix} &#1 \cr #2&  \cr \end{matrix} \rhk }}
\newcommand{\antidiagThree}[3]{ \lhk \begin{matrix} &&#1 \cr &#2&  \cr #3&&  \cr \end{matrix} \rhk }
\newcommand{\antidiagFour}[4]{ \lhk \begin{matrix} &&& #1 \cr &&#2&  \cr &#3&&  \cr #4&&& \end{matrix} \rhk }
\newcommand{\diagTwo}[2]{ \lhk \begin{matrix} #1& \cr &#2  \cr \end{matrix} \rhk }
\newcommand{\diagThree}[3]{ \lhk \begin{matrix} #1&& \cr &#2&  \cr &&#3 \end{matrix} \rhk }
\newcommand{\diagFour}[4]{ \lhk \begin{matrix} #1&&& \cr &#2&&  \cr &&#3&  \cr &&&#4 \end{matrix} \rhk }
\newcommand{\vierkantThree}[9]{\lhk \begin{matrix} #1 & #2 & #3 \cr #4 & #5 & #6 \cr #7 & #8 & #9 \end{matrix}\rhk }
\newcommand{\nilpThree}[3]{\lhk \begin{matrix} & #1 & #2 \cr &  & #3 \cr & & \end{matrix}\rhk }
\newcommand{\unipThree}[3]{\lhk \begin{matrix} 1 & #1 & #2 \cr & 1 & #3 \cr & & 1 \end{matrix}\rhk }
\newcommand{\BorelThree}[6]{\lhk \begin{matrix} #1 & #2 & #3 \cr & #4 & #5 \cr & & #6 \end{matrix}\rhk }

\newcommand{\lowerunipThree}[3]{\lhk \begin{matrix}
1 &  & \cr
#1 & 1 &  \cr
#2& #3& 1 \end{matrix}\rhk }

\newcommand{\smallantidiagTwo}[2]{{ \lhk \begin{smallmatrix} &#1 \cr #2&  \cr \end{smallmatrix} \rhk }}
\newcommand{\smallantidiagThree}[3]{ \lhk \begin{smallmatrix} &&#1 \cr &#2&  \cr #3&&  \cr \end{smallmatrix} \rhk }
\newcommand{\smallantidiagFour}[4]{ \lhk \begin{smallmatrix} &&& #1 \cr &&#2&  \cr &#3&&  \cr #4&&& \end{smallmatrix} \rhk }
\newcommand{\smalldiagTwo}[2]{ \lhk \begin{smallmatrix} #1& \cr &#2  \cr \end{smallmatrix} \rhk }
\newcommand{\smalldiagThree}[3]{ \lhk \begin{smallmatrix} #1&& \cr &#2&  \cr &&#3 \end{smallmatrix} \rhk }
\newcommand{\smalldiagFour}[4]{ \lhk \begin{smallmatrix} #1&&& \cr &#2&&  \cr &&#3&  \cr &&&#4 \end{smallmatrix} \rhk }
\newcommand{\smallvierkantThree}[9]{\lhk \begin{smallmatrix} #1 & #2 & #3 \cr #4 & #5 & #6 \cr #7 & #8 & #9 \end{smallmatrix}\rhk }
\newcommand{\smallnilpThree}[3]{\lhk \begin{smallmatrix} & #1 & #2 \cr &  & #3 \cr & & \end{smallmatrix}\rhk }
\newcommand{\smallunipThree}[3]{\lhk \begin{smallmatrix} 1 & #1 & #2 \cr & 1 & #3 \cr & & 1 \end{smallmatrix}\rhk }
\newcommand{\smallBorelThree}[6]{\lhk \begin{smallmatrix} #1 & #2 & #3 \cr & #4 & #5 \cr & & #6 \end{smallmatrix}\rhk }

\newcommand{\Asai}{\tu{Asai}}
\newcommand{\iA}{{\mathfrak A}}
\newcommand{\rel}{{\tu{rel}}}
\renewcommand{\BC}{{\tu{CB}}}
\newcommand{\CB}{{\tu{CB}}}
\newcommand{\ia}{{\mathfrak a}}
\newcommand{\isomfrom}{\overset \sim \leftarrow}

\section{Fonction $L$ de Rankin, version unitaire}

Dans ce chapitre, nous étudions le cas des groupes unitaires. Soit $E/F$ une extension quadratique de corps de nombres. Nous prouverons en particulier le résultat d’invariance en $s = 1/2$ pour la fonction $L$
\[
L(s, \pi \times \rho, \std \otimes \std^\vee)
\]
où $\pi$, $\rho$ sont des représentations automorphes cuspidales conjuguées autoduales de $\GL_n(\A_E)$, $\GL_r(\A_E)$, et sous certaines hypothèses (voir ci-dessus Théorème~\ref{thm:TheoremeUnitaire}) de nature similaire au cas de Rankin (Théorème~\ref{thm:Clozel-Rankin}).

On démontre également l’invariance du pôle en $s = 1/2$ de la fonction $L$ d’Asai
\[
L(s, \pi, \Asai^\pm),
\]
voir Proposition~\ref{prop:ParityInvariance}.

Nous commençons le chapitre par quelques préparations, car nous voulions déterminer quelles fonctions $L$ apparaissent dans la formule pour le terme constant des séries d’Eisenstein (voir Proposition~\ref{prop:Lquot}). Sans aucun doute, ce résultat peut également être déduit de \cite{Shahidi90} (et, quand $r \leq 1$, de \cite{GoldbergCrelle}).

\subsection{Notation}
Soit $E/F$ une extension quadratique de corps locaux ou globaux de caractéristique $0$, et $W_E, W_F$ leurs groupes de Weil. Pour $N \geq 0$ un entier, on écrit $G_N = \Res^E_F \GL_N$ et $U_N$ le groupe unitaire défini par la matrice
\begin{equation*}
\Phi_N = {\tiny \antidiagFour {1}{-1}{\iddots}{(-1)^{(N-1)}}}.
\end{equation*}
On note les faits suivants :
\begin{equation}\label{eq:Obvious}
\Phi_N^2 = (-1)^{N-1} \cdot \tu{id}_N \quad \tu{et} \quad \ut \Phi_N = \Phi_N\inv.
\end{equation}
Pour tout $F$-alg\`ebre $A$ on a
\[
U_N(A) = \{ g \in \GL_N(A \otimes_F E)\ |\ \ut \li g \Phi_N g = \Phi_N\}.
\]
Le groupe $U_N$ est quasi-déployé, un sous-groupe de Borel $B$ étant les matrices triangulaires supérieures. Soit $T \subset B$ le tore diagonale, donc pour $N$ pair (resp. impair)
\begin{align}\label{eq:ToreMaximal}
G_1^{N/2 }  \isomto &T, \quad (t_i)  \mapsto \diag(t_1, \ldots, t_{ N/2}, \li t_{ N/2 }\inv, \ldots, \li t_1\inv) \quad (\tu{resp.}) \cr
G_1^{\lfloor N/2 \rfloor} \times U_1 \isomto &T, \quad ((t_i), s)  \mapsto \diag(t_1, \ldots, t_{\lfloor N/2 \rfloor}, s, \li t_{\lfloor N/2 \rfloor}\inv, \ldots, \li t_1\inv)
\end{align}
où $\li t$ designe l'action de $\Gal(E/F)$ sur $t \in G_1$ et $U_1$ est le noyeau du norme $G_1 \to \Gm$. Soit $A_0 \subset T$ le tore déployé maximal, donc $A_0 \simeq \Gm^{\lfloor N/2 \rfloor}$ (via l'isomorphisme~\eqref{eq:ToreMaximal}).

On pose $\ia^* = X^*(A_0)_\C$, $\ia = X_*(A_0)_\C$, $\iA^* = X^*(T)_{\C}$, $\iA = X_*(T)_{\C}$. Soient $\Phi, \Phi^+, \Delta$ (resp. $\Phi_{\rel}, \Phi^+_{\rel}, \Delta_{\rel}$) les ensembles des racines, des racines positives et des racines simples de $G$ associés à $B, T$ (resp. $B, A_0$). On a $\Delta_\rel = \{\alpha_1, \ldots, \alpha_{\lfloor N/2 \rfloor}\}$ avec
\begin{equation}\label{eq:racines}
\alpha_i = t_i t_{i+1}\inv \tu{ pour } i < \lfloor N/2 \rfloor, \quad\tu{et} \quad
\alpha_m = t_m^{2-\eps} \tu{ si $N = 2m + \eps$, $\eps \in \{0,1\}$}.
\end{equation}
Comme d'habitude nous écrivons $\check \alpha$ pour la coracine correspondante à une racine $\alpha$.

Les $L$-groupes sont donn\'es par
\begin{alignat*}{3}
\LL G_N  =  \GL_N(\C)^2 \rtimes W_F,    & \quad (g, h) &&\mapsto (h, g) \cr
\LL U_N = \GL_N(\C) \rtimes W_F,        &  \quad \quad\textup{\ \ } a &&\mapsto \Phi_N \ut a\inv \Phi_N\inv  \cr
\LL (\Res_{E/F} (U_N)_E) = \GL_N(\C) \times \GL_N(\C) \rtimes W_F & \quad (g, h) &&\mapsto (\Phi_N {}^t h\inv \Phi_N\inv, \Phi_N {}^t g\inv \Phi_N\inv)
\end{alignat*}
où a chaque fois $W_F$ agit \`a travers $\Gal(E/F)$ et on a noté l'action des éléments dans $W_F \backslash W_E$. Soit $\BC$
le morphisme de changement de base stable, défini par la composition
\begin{equation}\label{eq:CB}
\xymatrix@R=0pt{
\LL U_N \ar[r] &  \LL (\Res_{E/F} (U_N)_E) \ar[r]^{\sim} & \LL G_N \cr
g \rtimes w \ar@{|->}[r] & (g,g) \rtimes w\textup{ ; } (g,h) \rtimes w \ar@{|->}[r] & (g, \Phi_N \ut h\inv \Phi_N\inv) \rtimes w
}
\end{equation}
(la deuxième provient du fait que $(U_N)_E \simeq \GL_{N, E}$).

Nous utiliserons également les deux morphismes de changement de base
\begin{equation}\label{eq:xi_kappa}
\xi_{\kappa} \colon \LL U_N \to \LL G_N \quad\quad (\kappa \in \{\pm 1\})
\end{equation}
introduits par Mok \cite[(2.1.9), p. 7]{Mok}. Mok écrit en réalité $\xi_{\chi_\kappa}$ pour notre $\xi_\kappa$ car les deux morphismes reposent sur le choix d’un caractère $\chi_\kappa \colon \A_E^\times / E^\times \to \C^\times$ avec la propriété que $\chi_\kappa$ restreint à $\A_F^\times / F^\times$ est le caractère trivial lorsque $\kappa = +1$, et est sinon le caractère $\omega_{E/F} \colon \A_F^\times / F^\times \to \{\pm 1\}$ associé à $E/F$ par la théorie du corps de classes. Le morphisme $\BC$ coincide avec $\xi_+$ si $\chi_+ = 1$ en conjuguant par l'élément $(1, \Phi_N\inv) \in \wh G_N$. Dans notre travail, le choix des caractères $\chi_\pm$ n’a pas d’importance, donc nous fixerons simplement $\chi_+ = 1$, et $\chi_-$ sera un choix arbitraire. 

\subsection{Sous-groupes de Levi maximale} Soit $1 \leq n \leq \lfloor N/2 \rfloor$ et $r = N - 2n$. Alors on note $P_r = M_r N_r \subset U_N$  le sous-groupes paraboliques triangulaires supérieurs maximale, avec
\begin{align*}
M_r  = \left\lbrace \left. \smalldiagThree{g}{a}{h}
\in U_N \ \right|  \ g \in G_n,  a \in U_r, h = \Phi_n \,{}^t g^{-1} \Phi_n^{-1} \right\rbrace \simeq G_n \times U_r,
\end{align*}
\vskip-0.5cm
\begin{align*}
N_r  = \left\lbrace \lhk \begin{smallmatrix}
1_n & * & * \\
  & 1_r & *  \\
  &   & 1_n
\end{smallmatrix} \rhk  \in U_N \right\rbrace.
\end{align*}
Si $r$ est clair, on note plus simplement $P, M$ et $N$ pour $P_r, M_r$ et $N_r$.

\begin{lemma}\label{lem:modulaire}
Soit $\delta_{P}(m) = |\det(m; \Lie(N))|$ le caractère modulaire. Alors $\delta_{P}(m) = |\det(g)|_E^{n+r}$ pour $m = \diag(g,a,h) \in M(F)$.
\end{lemma}

On inclut la vérification directe, car il est crucial de savoir sous quelles hypothèses sur $n,r$, le caractère $\delta_{P}^{1/2}$ prend des valeurs dans $\Q$ (donc si $n+r$ est pair). On a $\delta_{P}|_{U_r} = 1$, car $U_r$ n'a pas de caractère non-trivial défini sur $F$, et $\det(\cdot ; \Lie(N))|_{G_n}$ étant algébrique doit être de la forme $(N_{E/F} \det(\cdot))^x$ pour $x \in \Z$. Le déterminant de $m = \diag(t, 1, 1, \ldots, 1, 1, \li t\inv)$ ($t \in E^\times$) agissant sur
\[
\Lie(N) = \left\{
\left( \begin{smallmatrix}
0  & {b}  & {c} \\
0  & 0  & { (-1)^{n+r} \Phi_r \cdot \ut \li b \cdot \Phi_n} \\
0  & 0  & 0
\end{smallmatrix} \right) \ |\ b \in \uM_{n \times r}(E), c \in \Lie U_n
\right\},
\]
est égal à $\det(m; M_{n \times r}(E)) \det(m; \Lie U_n) = N_{E/F}(t)^r N_{E/F}(t)^n$, et donc $x = r+n$. Pour vérifier la deuxième, on constate que $\Lie U_n$ consiste en des $X \in \uM_{n \times n}(E)$ qui sont fixés par l'opérateur
\[
X \mapsto - \Phi_n \ut \li X \Phi_n\inv = [(-1)^{i+j+1} \li X_{n+1-j, n+1-i}]_{i,j=1}^n,
\]
et donc $\Lie U_n = \bigoplus V_{i,j}$ ($i \leq n+1-j$) où
\begin{align*}
E \cong V_{i,j} &= \{x e_{i,j} + (-1)^{i+j-1}\li {x} e_{n+1-j,n+1-i}\ |\ x \in E\}  \quad\quad (i < n+1 - j)  \cr
F \cong V_{i,n+1-i} &= \{x e_{i,j}\ |\  x \in E, \text{ tel que } \li x = (-1)^{n} x\} \quad\quad (i = n+1-j).
\end{align*}
Sur les $V_{i,j}$ du premier type, $m$ agit avec déterminant $N_{E/F}(t) = t \li t$ ($n-1$ fois, correspondant à la première ligne) ou déterminant $1$ (pour les autres lignes). Sur les $V_{i,n+1-i}$, qui sont de dimension $1$, $m$ agit par $N_{E/F}(m) = t \li t$ si $i = 1$, et par $1$ si $i \neq 1$. Donc on trouve $(t \li t)^n$, ce qui montre le lemme.

\label{lem:modulaire2}
\begin{lemma}
Nous avons $\delta_B(t) = |t_1|^{N-1}_E |t_2|^{N-3}_E \cdots |t_m|^{\eps+1}_E$ où $t = (t_1, t_2, \ldots) \in T(F)$ comme dans \eqref{eq:ToreMaximal} et $N= 2m + \eps$ avec $m \in \Z$ et $\eps \in \{0,1\}$.
\end{lemma}

Pour cela, on utilise $\delta_B(t) = \delta_{P_\eps}(t) \delta^{M_\eps}_{B \cap M_\eps}(t)$, où $\delta_{P_\eps}(t) = |\det(t)|_E^{m+\eps}$ par Lemme~\ref{lem:modulaire} et $\delta_{B \cap M_\eps}^{M_\eps}$ est le caractère module du Borel $B \cap M_\eps$ de $M_\eps$. Comme $M_\eps = G_m \times U_\eps$ nous avons $\delta_{B \cap M_\eps}^{M_\eps}(t) = |t_1|^{m-1}_E |t_2|^{m-3}_E \cdots |t_m|_E^{1-m},$ et le lemme s’ensuit.

Soit $A_{M} \subset M$ le centre depoloyé. On pose $\ia_P^* = X^*(A_{M})_\C$ et $\ia_P = X_*(A_{M})_\C$. On identifie $\ia_P^* = X^*_F(P)_\C$ (caractères $F$-rationels). On note $\rho_{P}$ le demi somme des $A_0$-racines qui apparessent dans $N$. Le caractère ``$\det$'' est un base de $\ia_P^*$, et on a $\rho_P = (n+r) \det \in \ia_P^*$ (Lemme~\ref{lem:modulaire}). Soit $e_i^*$ (resp. $e_i$) ($i \leq \lfloor N/2 \rfloor$) le base standard de $\ia^*$ (resp.  $\ia$) qui provient de l'isomorphisme $A_0 \simeq \Gm^{\lfloor N/2 \rfloor}$. Soit $\alpha = e_n^* - e_{n+1}^* \in \Delta_\rel$ si $r > 0$ et $\alpha = 2 e_n^*$ si $r = 0$. Alors $\alpha$ est l'unique racine simple tel que $\alpha|_{A_{M}} \neq 1$. On a $\check{\alpha} = e_n - e_{n+1}$ si $r > 0$ et $\check \alpha = e_n$ si $r = 0$. On pose
\[
\til \rho_{P} := \langle \rho_{P}, \check \alpha \rangle\inv \rho_{P} = e_1^* + e_2^* + \ldots + e_n^* \in \ia^*.
\]
Le plongement $\C^{\lfloor N/2 \rfloor} = \ia^* = (\iA^*)_{\Gal(E/F)} \isomfrom (\iA^*)^{\Gal(E/F)} \subset \iA^* = \C^N$ est donn\'e par
\[
(x_i) \mapsto \begin{cases} (x_1, \ldots, x_{N/2}, -x_{N/2}, \ldots, - x_1) \quad & \tu{$N$ pair} \cr
(x_1, \ldots, x_{\lfloor N/2 \rfloor }, 0, -x_{\lfloor  N/2 \rfloor }, \ldots, - x_1) \quad & \tu{$N$ impair} \end{cases}
\]
Donc
\begin{equation}\label{eq:DemiSommeRacines}
\til \rho_{P} = (\underset {\tu{$n$ fois}} {\underbrace { 1, \ldots, 1}}, \underset {\tu{$r$ fois}} {\underbrace { 0, \ldots, 0}} , \underset {\tu{$n$ fois}} {\underbrace  { -1, \ldots, -1}}) \in \iA^*.
\end{equation}

\subsection{Les répresentations $\nn_i$}\label{sec:LanglandsRecette}
Dans le terme constant de la série d'Eisenstein intervient un certain quotient de fonctions $L$. Nous déterminerons précisément quelles fonctions $L$ apparaissent. Nous rappelons d'abord la recette. À partir du $P = P_r \subset U_N$ comme avant, on peut prendre les L-groupes $\LL P, \LL M, \LL N$ plong\'es dans $\LL U_N$. Explicitement
\begin{equation}\label{eq:M_r-hat}
\wh M = \lbr \left. m = \smalldiagThree gah \in \GL_N(\C)\ \right| \ g,h \in \GL_n(\C), a \in \GL_r(\C) \rbr,
\end{equation}
avec action de $W_F$ \`a travers $\Gal(E/F)$ donnée par
\begin{align}\label{eq:Laction_on_M_r}
\Phi_N \ut m\inv \Phi_N\inv & =
 \smalldiagThree { \Phi_n \ut h\inv  \Phi_n\inv}{\Phi_r \ut a\inv \Phi_r\inv}{ \Phi_n \ut g\inv \Phi_n\inv}.
\end{align}
On a alors l'isomorphisme (cf. \eqref{eq:Obvious})
\begin{align}\label{eq:IdentifyLgroupM}
\varphi \colon \LL [G_n \times U_r]  & \isomto \LL M, \quad (g,h;a) \rtimes w  \mapsto \smalldiagThree {g}a{\Phi_n \ut h\inv \Phi_n\inv} \rtimes w.
\end{align}
Soit
\begin{equation}\label{eq:LieDuale}
\nn = \Lie(\LL N) = \lbr \left. \smallnilpThree y x {\ut z} \right| x \in \uM_{n \times n}(\C), y,z \in \uM_{n \times r}(\C) \rbr.
\end{equation}
Le groupe dual $\widehat{U} (N)$ est $\GL_N(\C)$, avec les données radicielles habituelles correspondant au tore diagonal et au sous-groupe de Borel triangulaire supérieur. Pour chaque racine $\check{\beta}$ de $\GL_N(\C)$ nous avons l'espace radiciel correspondant $X_{\check{\beta}}$ provenant de l'épinglage standard de $\GL_N(\C)$ ($X_{\check{\beta}}$ correspond essentiellement à une entrée de matrice). Selon la recette générale de Langlands, les nombres $\langle \tilde{\rho}_P, \check{\beta} \rangle$ donnent un intervalle d'entiers de la forme ${1, \ldots, m}$ pour un certain entier $m$, si $\check{\beta}$ parcourt les racines pour lesquelles $X_{\check{\beta}} \subset \nn$. Pour $i \in {1, \ldots, m}$ on pose
\begin{equation}\label{eq:Espace_n_i}
\nn_i = \bigoplus_{\check{\beta}:\ X_{\check{\beta}} \subset \nn,\ \langle \tilde \rho_P, \check \beta \rangle = i } X_{\check \beta} \subset \nn.
\end{equation}
Par Langlands l'action adjoint de $\LL M$ laisse l'espace $\nn_i$ stable et $\nn_i$ est irréductible.

Soit $E_i = (0, \ldots, 0, 1_i, 0, \ldots 0)$ ($i \leq N$) le base de $\iA$ qui provient de \eqref{eq:ToreMaximal}. Explicitement on a $\check \beta = E_i - E_j$ pour $i \neq j$, $i,j\leq N$ et $X_{\check{\beta}} \subset \nn$ dans un des 3 cas suivant:
\begin{itemize}[leftmargin=2.5cm,labelsep=0.4cm]
\item[(Cas X)] $1 \leq i \leq n$ et $n + r + 1 \leq j \leq N$ (le bloc ``$x$'' dans~\eqref{eq:LieDuale});
\item[(Cas Y)] $1 \leq i \leq n$ et $n + 1 \leq j \leq n+r$ (le bloc ``$y$'' dans~\eqref{eq:LieDuale});
\item[(Cas Z)] $n+1 \leq i \leq n + r$ et $n + r + 1 \leq j \leq N$ (le bloc ``$z$'' dans~\eqref{eq:LieDuale}).
\end{itemize}
On a donc
\[
\langle \til \rho_P, \check \beta \rangle =
\langle ( 1, \ldots, 1_{(n)},  0, \ldots, 0_{(n+r)} ,  -1, \ldots, -1_{(2n+r)}), E_i - E_j \rangle =
\begin{cases}
2 & \textup{(cas X)}\cr
1 & \textup{(cas Y et Z)}.
\end{cases}
\]
L'espace $\nn$ se decompose donc en deux representations irreducibles $\nn_1$ et $\nn_2$.

Soit $m = \diag(g,a,h) \in \widehat M$, alors
\begin{equation}\label{eq:Conjugation1}
\tu{Ad}_m \smallunipThree yx {\ut z}  = \smallunipThree {gy a^{-1}}{gxh^{-1}}{ a\, \ut z  h^{-1}},
\end{equation}
et $1 \rtimes w$ ($w \in W_F \backslash W_E$) agit par
\begin{align}\label{eq:Conjugation2}
\Phi_N \prescript{t}{} {\smallunipThree y x z^{-1}} \Phi_N^{-1}
= \smallunipThree {\Phi_r^{-1} \ut z \Phi_n^{-1}}{ (-1)^{n+r} \Phi_n (\ut y z - \ut x) \Phi_n^{-1}}{ \Phi_n^{-1} \ut y \Phi_r^{-1}}.
\end{align}
L'action de $\LL M$ sur $\nn_i$ ($i = 1,2$) est donc donnée par
\begin{align}\label{eq:RepR1R2}
R_1 \colon \LL M & \to \GL(\uM_{n \times n}(\C))  \cr
\diag(g,a,h) \rtimes w & \mapsto (x \mapsto g x h^{-1}) (x \mapsto (-1)^{n+r+1} \Phi_n \ut x \Phi_n^{-1})^{q(w)} \cr
R_2 \colon \LL M & \to \GL(\uM_{n \times r}(\C) \oplus \uM_{n\times r}(\C)) \simeq \GL_{2nr}(\C) \cr
\diag(g,a,h) \rtimes w & \mapsto \smalldiagTwo {R(g, a)}{R(\ti h, \ti a)} \smallantidiagTwo{S}{S^{-1}}^{q(w)},
\end{align}
où
\begin{align}\label{eq:RepR}
q \colon W_F/W_E & \isomto \Z/2\Z, \quad S \in \GL(\uM_{n \times r}(\C)),\ S(x) := \Phi_n^{-1} x \Phi_r^{-1}, \cr
R \colon \GL_n(\C) \times \GL_r(\C) & \to \GL(\uM_{n \times r}(\C)), \quad (g, a) \mapsto (x \mapsto g \cdot x \cdot a\inv).
\end{align}

\subsection{La fonction $L$ associé a $\nn_1$}
 On a (voir~\eqref{eq:IdentifyLgroupM} pour $\varphi$)
\begin{align*}
R_1 \circ \varphi [(g, h; a) \rtimes w] = (x \mapsto g x (\Phi_n \ti h \Phi_n\inv)\inv)(x \mapsto (-1)^{n+r+1} \Phi_n \ut x \Phi_n\inv)^{q(w)}.
\end{align*}
Donc $R_1 \circ \varphi$ est trivial sur $\GL_r(\C)$ et isomorphe à $\std \otimes \std$ sur $\GL_n(\C)^2 = \widehat G_n$. La représentation $R_1 \circ \varphi|_{\widehat G_n} \simeq \std \otimes \std$ a, à isomorphisme près, exactement deux extensions à une représentation du groupe $\LL G_n$ qui sont triviales sur $W_E$. Ces deux extensions sont appelées $\Asai^\pm$, dans lesquelles les éléments de $W_F \backslash W_E$ agissent par $v_1 \otimes v_2 \mapsto \pm v_2 \otimes v_1$ sur $\std \otimes \std$ (cf. \cite[Eq. (2.2.9)]{Mok}). On a donc $\Asai^- \simeq \omega_{E/F} \otimes \Asai^+$, où $\omega_{E/F} \colon W_F \to \{\pm1\}$ est la caractère quadratique associé a $E/F$. Une façon de distinguer $\Asai^+$ et $\Asai^-$ est
\begin{equation}\label{eq:Asai}
\Tr \Asai^\pm(1 \rtimes w) = \pm n,
\end{equation}
pour $w \in W_F \backslash W_E$. Alors, pour identifier $R_1 \circ \varphi$, on calcule
\begin{align*}
\Tr(1 \rtimes w, R_1 \circ \varphi) &= \Tr(\uM_{n \times n}(\C) \owns x \mapsto (-1)^{n+r+1} \Phi_n \ut x \Phi_n^{-1}) \cr
&= \Tr( [x_{ij}]_{i,j=1}^n \mapsto [(-1)^{n+r+1} (-1)^{i+j} x_{n+1-j, n+1-i}]_{i,j=1}^n ) = (-1)^{r} n.
\end{align*}
Pour la deuxième égalité, on écrit $\Phi_n = sw$ avec $s = \diag(1, -1, \ldots)$ et $w$ le matrice avec $1, \ldots, 1$ sur l'anti-diagonale, et alors, transposer est $(x_{ij}) \mapsto (x_{ji})$, conjuger avec $w$ est $(x_{ij}) \mapsto (x_{n+1-i,n+1-j})$, et conjuger avec $s$ est $(x_{ij}) \mapsto ( (-1)^{i+j} x_{ij})$. Pour la dernière égalité, nous voyons que l'opérateur permute, à signes près, les vecteurs de la base standard. Pour un tel opérateur, la trace est obtenue en additionnant les signes correspondant \`a des vecteurs de base qui ne sont pas déplacés. Ici, il s'agit des entrées sur l'anti-diagonale. Donc si $j = n+1-i$, et avec signe
\[
(-1)^{n+r+1} (-1)^{i+j} = (-1)^{n+r+1} (-1)^{i+(n+1-i)} = (-1)^r.
\]
Alors, nous concluons que $R_1 \simeq \Asai^{(-1)^{r}}$.

\subsection{La function $L$ associe a $\nn_2$}
La function $L$ associe a $\nn_2$ est décrit par le lemme suivant.

\begin{lemma}\label{lem:FunctionLpourR1}
Soit $\pi \otimes \sigma$ une représentation automorphe de $M(\A) \simeq \GL_n(\A_E) \times U_r(\A_F)$, alors
\begin{equation}\label{eq:DescriptionFonctionL_R1}
L(s, \pi, R_2 \circ \varphi) = L(s, \pi \otimes \tu{BC}(\sigma), R)
\end{equation}
où, on rappelle, $\tu{BC}$ est le morphisme de changement de base standard pour $U_r$.
\end{lemma}

On a
\begin{align*}
R_2 \circ \varphi \colon \LL [G_n \times U_r] & \to \GL(\uM_{n \times r}(\C) \oplus \uM_{n\times r}(\C)) \cr
(g,h,a) \rtimes w & \mapsto \diagTwo {R(g, a)}{R(\Phi_n h \Phi_n^{-1} , \ti a)} \antidiagTwo{S}{S^{-1}}^{q(w)}.
\end{align*}
Apr\`es avoir conjugu\'e avec $\smalldiagTwo {S}1$, ceci devient
\begin{align}\label{eq:ConjugerS1}
(g,h,a) \rtimes w  \mapsto & \diagTwo {S R(g, a) S^{-1} }{R(\Phi_n h \Phi_n^{-1} , \ti a)} \antidiagTwo{1}{1}^{q(w)} \cr
= & \diagTwo { R(\Phi_n g \Phi_n^{-1}, \Phi_r a \Phi_r^{-1} )  }{R(\Phi_n h \Phi_n^{-1} , \ti a)} \antidiagTwo{1}{1}^{q(w)} \cr
\simeq & \diagTwo { R(g, a )  }{R(h ,  \Phi_r \ti a  \Phi_r^{-1})} \antidiagTwo{1}{1}^{q(w)},
\end{align}
(dans le ``$\simeq$'' on a conjug\'e $g, h$ par $\Phi_n$ et $a$ par $\Phi_r$). Soit
\begin{align*}
R' \colon \LL [ G_n \times G_r] & \to \GL_{2nr}(\C) \quad (g,h;a,b) \rtimes w  \mapsto \diagTwo {R(g, a)}{R(h, b)} \antidiagTwo 11^{q(w)} \cr
B \colon \LL [G_n \times U_r] &\to \LL [G_n \times G_r] \quad (g,h;a) \rtimes w \mapsto (g,h; a, \autom ra ) \rtimes w.
\end{align*}
Donc $B$ est donn\'e par changement de base standard sur $U_r$ et l'identit\'e sur $G_n$. Alors \eqref{eq:ConjugerS1} donne
\[
R_2 \circ \varphi \simeq R' \circ B.
\]
Pour d\'emontrer Lemme~\ref{lem:FunctionLpourR1} on peut supposer que $E/F$ est local (si $E/F$ est global, la démonstration se fait en considérant chaque place de $F$).
Soient
\begin{align*}
\psi_{\pi \otimes \sigma} \colon L_F  & \to [\GL_n(\C)^2 \times \GL_r(\C) ] \rtimes W_F \cr
\psi_{\pi \otimes \tu{\BC}(\sigma)} \colon L_E & \to [\GL_n(\C) \times \GL_r(\C)] \times W_E
\end{align*}
les param\`etres de $\pi$ et $\pi \otimes \tu{\BC}(\sigma)$. Alors $B\psi_\pi$ et $\psi_{\pi \otimes \BC(\sigma)}$ correspondent via le lemme de Shapiro, et on v\'erifie que
\[
R' B \psi \simeq \ind_{L_E}^{L_F} (R \psi_{\pi \otimes \BC(\sigma)}).
\]
Le lemme s'en d\'eduit.

\subsection{Séries d'Eisenstein}\label{sec:Series} Ici nous supposerons que  $E/F$ est une extension quadratique des corps de nombres (bien entendu, on ne fait pas d'hypothèse que $E/F$ soit de type CM). En plus, on fixe $n, r$ deux entiers, avec $n > 0$ et $r \geq 0$, et
\begin{itemize}
\item $\pi$ une représentation automorphe cuspidale unitaire de $G_n(\A_F)$ ;
\item $\sigma$ une représentation automorphe cuspidale unitaire de $U_r(\A_F)$.
\end{itemize}
On pose $N = 2n+r$. On voit alors $\pi \otimes \sigma$ comme représentation automorphe cuspidale de $M \subset U_N$. Soit $s \in \C$. On pose
\[
I_s = \ind_{P(\A_F)}^{U_N(\A_F)} (\pi |\det|^s_E \otimes \sigma) = \Ind_{P}^{U_N(\A_F)} (\pi |\det|^{s + \tfrac {n+r}2}_E \otimes \sigma),
\]
où, comme d'habitude, $P$ agit par inflation via $P \twoheadrightarrow M$ et l'induction $\ind$ est unitaire. Quand $s = 0$, nous écrivons $I = I_0$.

Nous écrivons $H_{P}$ pour l'application de Harish-Chandra
\[
H_P \colon M(\A_F) \to \ia, \quad \forall \chi \in X_F^*(P), \forall m \in M(\A_F):
\exp(\langle \chi, H_P(m) \rangle) = |\chi(m)|_{\A_F}.
\]
Soit $K \subset U_N(\A_F)$ un sous-groupe compact maximal en bonne position par rapport à $P$ \cite[Section I.1.4]{MoeglinWaldspurgerANENS89}. Alors nous étendons $H_P$ au groupe $U_N(\A_F)$ via la décomposition d'Iwasawa $U_N(\A_F) = K M(\A_F) N(\A_F)$ tel que $H_P$ soit trivial sur $K$ et $N(\A_F)$.

Comme dans les autres chapitres, nous considerons des series d'Eisenstein,
\begin{align*}
I & \to \cA(U_N(F) \backslash U_N(\A_F)) \cr
 f & \mapsto E(f, s)(g) = \sum_{\gamma \in P(F) \backslash U_N(F)} \exp(\langle s + \rho_P, H_P(m) \rangle) f(\gamma g)
\end{align*}
pour $s \in \C$ avec $\Re(s) \gg 1$ (ici on utlise $\til \rho_P$ comme base de $\ia_P^* \simeq \C$). La série du côté droit converge lorsque $\Re(s) \gg 1$, et la série d'Eisenstein est autrement définie par prolongement méromorphe. Si $E(f, s)$ a un pôle en $s = s_0 > 0$ pour un certain $f \in I$, alors il a également un pôle pour les autres $f \in I$ avec $f \neq 0$. Dans ce cas, les résidus de $E(f,s)$ en $s = s_0$ génèrent une représentation automorphe discrète $\Pi(s_0)$ de $U_N$.

\begin{proposition}\label{prop:Lquot}
Soit $\pi \otimes \sigma$ une représentation automorphe cuspidale unitaire générique de $\GL_n(\A_E) \times U_r(\A_F)$. Alors les pôles avec $\Re(s) > 0$ de la série d'Eisenstein $E(f, s)$ pour $f \in I_s$, $f \neq 0$, coïncident avec les pôles avec $\Re(s) > 0$ de
\begin{equation}\label{eq:Lquot}
\frac {L(s, \pi \otimes \BC(\sigma), R)}{L(s+1, \pi \otimes \BC(\sigma), R)}
\frac {L(2s, \pi, \Asai^{(-1)^r})} {L(2s+1, \pi, \Asai^{(-1)^r})}
\end{equation}
où $\BC$, $R$ et $\Asai^\pm$ sont définis dans \eqref{eq:CB},~\eqref{eq:RepR}~et~\eqref{eq:Asai}.
\end{proposition}

Ici on remarque que d’après le résultat de Henniart~\cite[Sect.~5.2]{HenniartIMRN2003} les facteurs locaux $L_v(s, \pi \otimes \BC(\sigma), R)$ et $L_v(2s, \pi, \Asai^{(-1)^r})$, définis par Shahidi, coïncident avec ceux définis par les paramètres locaux de Langlands de $\pi$ et $\rho$.

La démonstration de cette proposition est détaillée dans Grbac-Shahidi~\cite[Thm.2.1]{GrbacShahidi}, avec certaines réserves. Leur argument est écrit dans le cas où $r \leq 1$, car dans ce cas, ils n’ont pas besoin de supposer que $\sigma$ est générique, car pour $\pi$ cela est de toute façon automatique, et lorsque $r \leq 1$, c’est aussi vrai pour $\sigma$. À part cela, la démonstration de notre proposition est identique à leur Théorème~2.1, sauf qu’on remplace leur Lemme~2.3 par les travaux de Kim et Krishnamurthy \cite{KimKrishnamurty1, KimKrishnamurty2}, qui démontrent l’holomorphie et la non-annulation des opérateurs d’entrelacement normalisés. Kim et Krishnamurthy ne posent pas la restriction $r \leq 1$, mais supposent la condition de généricité (satisfaite pour nous) ; voir également la Remarque 2.4, p.~198 de~\cite{GrbacShahidi}.

\subsection{Pôles des fonctions $L$ d'Asai}\label{sec:Asai}
Soit $E/F$ une extension quadratique des corps de nombres. Soit $\rho$ une représentation cuspidale conjuguée autoduale de $\GL_n(\A_E)$. Alors exactement une des deux fonctions d’Asai
\[
L(s, \rho, \Asai^\pm)
\]
admet un pôle en $s = 1$ (voir~\cite[p.~30]{Mok}). Les deux fonctions sont non nulles en $s = 1$.

\begin{definition}\label{def:Parite}
On dit que $\rho$ est \emph{conjuguée orthogonale} ou \emph{de parité ``$\eta(\rho) = +1$''} (resp. \emph{conjuguée symplectique} ou \emph{de parité ``$\eta(\rho) = -1$''}) si $L(s, \rho, \Asai^+)$ (resp. $L(s, \rho, \Asai^-)$) a un pôle en $s = 1$.
\end{definition}

Dans le langage de Mok, $\rho$ définit un paramètre conjugué autodual simple générique dans $\til \Phi_{\text{sim}}(r)$, et est donc, par son Théorème 2.4.2, de la forme $\rho = \xi_{\kappa}(\sigma)$ pour une représentation automorphe $\sigma$ de $U_r$ et pour un signe unique $\kappa = \kappa(\rho) \in \{\pm1\}$. Le Théorème 2.5.4(a) de Mok montre alors que $L(s, \rho, \Asai^{(-1)^{r-1} \kappa})$ admet un pôle en $s = 1$. Donc
\begin{equation}\label{eq:MokIdentity}
\eta(\rho) = \kappa(\rho)(-1)^{r-1}.
\end{equation}

\subsection{L'invariance de la parité}
Nous allons prouver que la parité est invariante sous l'action de $\Aut(\C)$.

\begin{proposition}\label{prop:ParityInvariance}
Soit $\pi$ une représentation automorphe régulier algébrique cuspidale conjuguée autoduale de $G_n$. Nous avons
\[
\eta(\pi) = \eta(a(\pi))
\]
pour tout $a \in \Aut(\C)$.
\end{proposition}

(Nous remarquons ici que $a(\pi)$ existe d’après un résultat de Clozel \cite[Thm.~3.13]{Clozel_AA}.)

Supposons d’abord que $\eta(\pi) = (-1)^n$. Alors $\kappa(\pi) = 1$, donc $\pi = \BC(\sigma)$ pour une certain représentation automorphe  discret $\sigma$ de $U_r$ d’après le résultat de Mok. Nous savons que $\pi$ est cohomologique cuspidale, donc tempérée à l’infini. Mok prouve la compatibilité aux places infinies, et ainsi $\sigma$ est également tempérée à l’infini, ce qui implique que $\sigma$ est cuspidale (selon Wallach \cite{WallachConstantTerm}). Comme le caractère infinitésimal de $\sigma$ coïncide avec celui d’une représentation algébrique, nous savons que $\sigma$ est cohomologique. Par le résultat de Rai et Nair \cite{NairRai}, la cohomologie $L^2$ a une $\Q$-structure, et donc la représentation $a(\sigma)$ existe. Nous affirmons maintenant que les représentations
\begin{equation}\label{eq:BCandGaloisTwist}
\CB(a(\sigma)) \quad \tu{et} \quad a(\BC(\sigma))
\end{equation}
sont isomorphes. Ensuite, \eqref{eq:BCandGaloisTwist} implique que $\kappa(a(\pi)) = 1$. Cela prouve que ``$\eta(\pi) = (-1)^n \Rightarrow \eta(a(\pi)) = (-1)^n$''. Pour l’implication réciproque, on peut argumenter de manière similaire.

Nous prouvons maintenant \eqref{eq:BCandGaloisTwist}. Par la propriété de multiplicité 1 forte, il suffit de le prouver aux places finies non ramifiées. Soit donc $v$ une telle place de $F$, et écrivons
\begin{equation}\label{eq:QuotientBeforeTwist}
\ind_{B}^{U_n}(\chi) \surjects \sigma_v, \quad \quad \chi = (\chi_1, \ldots, \chi_m) \in \Hom_{\tu{unr}}(T(F), \C^\times).
\end{equation}
avec $\chi$ le paramètre de Satake de $\sigma_v$ (les composantes $\chi_i \colon E^\times \to \C^\times$ sont définies via \eqref{eq:ToreMaximal}; l'induction est normalisée). Alors $\BC(\sigma_v)$ a pour paramètre de Satake
\[
(\chi_1, \ldots, \chi_m, \chi_m\inv, \ldots, \chi_1\inv)
, \quad \tu{ou} \quad (\chi_1, \ldots, \chi_m, 1, \chi_m\inv, \ldots, \chi_1\inv)
\]
selon que $n = 2m + \eps$ ($\eps \in \{0, 1\}$) est pair ou impair. En prenant pour $\chi$ le caractère module $\delta_B^{U_n}$, nous voyons d’après le Lemme \ref{lem:modulaire2} que $\BC(\delta_B^{U_n})$ est le caractère module $\delta_n$ de $G_n$.

En utilisant la compatibilité de l’induction non normalisée avec $a(\cdot)$, nous voyons que
\[
\Ind_B^{U_r}(a(\delta_B^{1/2}) a(\chi)) \surjects a(\sigma_v)
\]
Ainsi, le paramètre de Satake de $a(\sigma_v)$ est
$\delta_B^{-1/2} a(\delta_B^{1/2}) a(\chi)$.
En appliquant $\CB$, nous trouvons que le paramètre de Satake de $\CB(a(\chi_v))$ est égal à
\[
\CB(\delta_B^{-1/2} a(\delta_B^{1/2}) a(\chi)) = \delta_n^{-1/2} a(\delta_n^{1/2}) a(\CB(\chi)).
\]
ce qui est le paramètre de Satake de $a (\CB(\sigma_v))$.

\subsection{Le paramètre $\psi$}\label{sec:ParamPsi}
On continue avec les notations $E/F, \rho$ de la section précédente. Nous rappelons brièvement quelques définitions du livre de Mok pour les paramètres $U_N$ (pour plus de détails, voir \cite[Chap. 2]{Mok}). Nous écrivons $\Psi(N)$ pour l’ensemble des paramètres formels
\begin{equation}\label{eq:ParametreGenerale}
\psi^N = \pi_1 \otimes \sp(d_1) \oplus \pi_2 \otimes \sp(d_2) \oplus \ldots \oplus \pi_t \otimes \sp(d_t).
\end{equation}
Nous nous intéressons uniquement aux paramètres auto-duaux elliptiques conjugués ; cela signifie que $\pi_i$ est conjugué autoduale pour chaque $i$. Mok définit deux ensembles supplémentaires
\[
\Psi_2(U_N, \xi_\kappa), \quad\quad \kappa \in {\pm 1}.
\]
Ces deux sous-ensembles peuvent être utilisé pour paramétrer le spectre discret de $U_N$ (voir Théorème 2.5.2 dans \cite{Mok}), et ils correspondent au fait que $U_N$ peut être vu de deux manières comme un groupe endoscopique maximal de $\GL_N$ tordu (via le changement de base standard ou tordu). Le paramètre $\psi^N$ définit un élément de $\Psi_2(U_N, \xi_\kappa)$ si et seulement si\footnote{Moralement, pour que $\psi^N$ appartienne à $\Psi_2(U_N, \xi_\kappa)$, cela signifie que leparamètre global d’Arthur $\cL_F \times \SL_2(\C) \to \LL U_N$ se factorise par $\xi_\kappa$ ; mais comme le groupe de Langlands global $\cL_F$ n’est pas connu d'exister (ou d'avoir les propriétés attendues), Mok ne peut pas formuler la condition de cette manière.} pour chaque $i$, nous avons (cf.~\cite[pp. 20-21]{Mok})
\[
\eta(\pi_i)(-1)^{N + d_i} = \kappa.
\]
Dans notre argument, nous considérerons le paramètre
\[
\psi^N = \pi \otimes \sp(2) \oplus \rho  \in \Psi(N)
\]
où $\pi$ est une réprésentation automorphe cuspidale conjugée autoduale de $G_n$.
Selon ce qui précède, $\psi^N$ définit un élément de $\Psi_2(U_N, \xi_\kappa)$ pour $\kappa \in \{\pm1 \}$, si
\begin{equation}\label{eq:SignCondition}
\eta(\pi) = (-1)^r\kappa \quad \tu{et} \quad \eta(\rho) = (-1)^{r+1}\kappa.
\end{equation}
Remarquez que, indépendamment du signe de $\kappa$, nous avons $\eta(\pi) \neq \eta(\rho)$. En particulier, $\pi$ et $\rho$ ne peuvent pas être de même parité. Dans le cas où leurs parités seraient les mêmes, $\psi^N$ ne définirait pas un paramètre discret pour $U_N$, mais plutôt pour le groupe endoscopique tordu $U_{2n} \times U_r$ de $G_N$.

Nous supposerons jusqu’à la fin de cette section que la condition \eqref{eq:SignCondition} est satisfaite.

Dans son (2.4.3) Mok définit un substitut $\cL_{\psi^N}$ du groupe de Langlands global, pour le paramètre $\psi^N$. De la discussion là, il s'ensuit que dans notre cas
\[
\cL_{\psi^N} = \LL U_n \times_{W_F} \LL U_r = (\GL_n(\C) \times \GL_r(\C)) \rtimes W_F.
\]
Mok définit en outre deux morphismes $\til \psi^N$ et $\til \psi$
\begin{equation}\label{eq:Diagram}
\xymatrix{
&& \LL U_N \ar[d]^{\xi_\kappa} \cr
\cL_{\psi^N} \times \SL_2(\C) \ar[urr]^{\til \psi} \ar[rr]_{\quad\quad\quad\til \psi^N\quad} && \LL G_N
}
\end{equation}
où
\begin{align*}
\til \psi^N \lhk (g, a) \rtimes w, \vierkant xyzv \rhk & = \vierkantdrie {\xi_{\kappa(\pi)}(g)x}{}{\xi_{\kappa(\pi)}(g)y}{}{\xi_{\kappa(\rho)}(a)}{}{\xi_{\kappa(\pi)}(g)z}{}{\xi_{\kappa(\pi)}(g)v} \rtimes w \cr
\til \psi \lhk (g, a) \rtimes w, \vierkant xyzv \rhk & = \vierkantdrie {gx}{}{gy}{}{a}{}{gz}{}{gv} \rtimes w.
\end{align*}

Nous considérons le paramètre $\psi = (\psi^N, \til \psi^N)$ dans $\Psi(U_N, \xi_\kappa)$ (voir [\loccit, Def. 2.4.5]). Nous définissons
\[
\cS_\psi = \pi_0\lhk \tu{Cent}(\tu{Im} \til \psi, \hat U_N) / Z(\hat U_N)^{W_F} \rhk \simeq \{\pm 1\}.
\]
Vérifions qu'effectivement $\cS_\psi  \simeq \{\pm 1\}$. Nous avons
$\tu{Im}(\til \psi)^0 = \lbr \vierkantdrie *{\ }*{\ }*{\ }*{\ }* \in \GL_N(\C) \rbr$ et le centralisateur de ce groupe est
\[
\lbr \vierkantdrie {x 1_n}{\ }{\ }{\ }{y 1_r}{\ }{\ }{\ }{x 1_n} \in \GL_N(\C) \ |\ x,y \in \C^\times \rbr.
\]
Une matrice $M$ de ce type commute avec $1 \rtimes w$ pour $w \in W_F \backslash W_E$ lorsque $\Phi_N \ut M\inv \Phi_N\inv = M$ et donc précisément lorsque $x, y \in \{\pm 1\}$. Comme $Z(\hat U_N)^{W_F}) = \{\pm 1\}$, nous trouvons $\cS_\psi \simeq \{\pm 1\}$.

Pour chaque place $F$ $v$, Mok définit un analogue local $\cS_{\psi_v}$ de $\cS_\psi$ (ci-dessous son (2.2.11)), ainsi que des flèches naturelles $\cS_{\psi} \to \cS_{\psi_v}$. Par son Théorème 2.5.1, le paramètre local $\psi_v$ donne lieu à une paire $(\Pi_{\psi_v}, p)$ consistant en un paquet (le $A$-paquet local) et une application
\[
p \colon \Pi_{\psi_v} \to \Hom(\cS_{\psi_v}, \C^\times), \quad \tau \mapsto \langle \cdot, \tau \rangle.
\]
Le caractère global de signe $\eps_\psi$ est défini comme suit. Écrivons $\cL_\psi = \cL_{\psi^N}$. Mok définit
\[
\tau_\psi \colon \cS_\psi \times \cL_\psi \times \SL_2(\C) \to \GL(\hat \ig), \quad (s, g, h) \mapsto \ad(s \til \psi(g, h)).
\]
Alors nous avons la décomposition en irréductibles $\hat \ig = \bigoplus_\alpha \lambda_\alpha \otimes \mu_\alpha \otimes \nu_\alpha$ où $\lambda_\alpha$, $\mu_\alpha$, $\nu_\alpha$ sont des représentations irréductibles de $\cS_\psi$, $\cL_\psi$, $\SL_2(\C)$ (respectivement). Alors
\[
\eps_\psi = \prod_\alpha \lambda_\alpha
\]
où le produit porte sur tous les indices $\alpha$ tels que $\mu_\alpha$ soit symplectique et $\eps(1/2, \mu_\alpha) = -1$ ([\loccit , (2.5.6)]). On pourrait déterminer directement la forme de $\eps_\psi$ à partir de la définition de notre paramètre explicite $\psi^N$, mais nous n’en aurons pas besoin.

Le paquet global $\Pi_\psi(\eps_\psi)$ est alors l’ensemble des $\otimes’ \pi_v$ admissibles tels que $\langle \cdot, \pi_v \rangle = 1$ pour presque tout $v$, et $\prod_v \langle \cdot, \pi_v \rangle|_{\cS_\psi} = \eps_\psi$. Nous obtenons un sous-espace \cite[Thm. 2.5.2]{Mok}
\begin{equation}\label{eq:H_psi}
\cH_\psi := \bigoplus_{\pi \in \Pi_\psi(\eps_\psi)} \pi \subset L^2_{\disc}(U_N(F) \backslash U_N(\A_F)).
\end{equation}

\subsection{La version unitaire de ``superrégularité''} Comme précédemment, nous considérons le $A$-paramètre $\psi = \pi \otimes \sp(2) \oplus \rho$. Nous cherchons une hypothèse naturelle de régularité sur les caractères infinitésimaux de $\pi$ et $\rho$ (pour $v | \infty$) de sorte que toutes les représentations dans les paquets archimédiens associés à $\psi_v$ soient cohomologiques.

À ce moment, nous faisons uniquement les hypothèses sur $\pi$, $\rho$ et $\sigma$ que $\rho = \BC(\sigma)$, et que $\sigma$ et $\pi$ sont automorphes unitaires.
Nous passons à la notation locale : Soit $v | \infty$ et écrivons $\pi, \sigma, F, E$ pour $\pi_v, \rho_v, F_v, E_v$. On commence avec le cas $F = \R$ et $E \simeq \C$.

Le caractère infinitésimal de $\pi$ est donné par un élément dans $X^*(T(n))_{\C}^+$, où
$T(n) = \Res_{E/F} \Gm^n \subset G_n$ est le tore diagonal et les caractères $X^*$ sont sur $E$, le $+$ indique la chambre dominante, définie par le Borel triangulaire supérieur. Ainsi, explicitement, le caractère infinitésimal est de la forme
\[
(p_1, \ldots, p_n) \times (p_1’, \ldots, p_n’) \in \C^n \times \C^n = X^*(T(n))_{\C}
\]
avec $p_1 \geq p_2 \geq \ldots \geq p_n$ et $p_1’ \geq p_2’ \geq \ldots \geq p_n’$. Par autodualité conjuguée, il est dans l’image de $\xi_\kappa \colon \LL U_n \to \LL G_n$, et donc il satisfait
\[
(p_1’, p_2’, \ldots, p_n’) = (-p_n, -p_{n-1}, \ldots, -p_1).
\]
Écrivons, pour le caractère infinitésimal de $\rho$, de manière similaire
\[
(q_1, \ldots, q_r) \times (q_1’, \ldots, q_r’) \in \C^r \times \C^r = X^*(T(r))_{\C}
\]
avec $q_1 \geq q_2 \geq \ldots \geq q_r$ et $q_1’ \geq q_2’ \geq \cdots \geq q_r’$. Le caractère infinitésimal de
\[
I_s = \tu{ind}_M^{U_N} (\pi |\det |_E^{1/2} \otimes \sigma)
\]
est alors donné par l'ensemble
\begin{equation}\label{eq:Differentes}
\lhk \{ p_i \pm 1/2 \} \cup \{q_j\}\rhk \times \lhk \{ p_i' \pm 1/2 \} \cup \{q_j'\} \rhk \in (\C^N / \iS_N)^2
\end{equation}
où $1 \leq i \leq n$ et $1 \leq j \leq r$.

Les nombres dans cet ensemble doivent tous être différents, ce qui équivaut à
\begin{equation}\label{eq:SuperregularityDisjointness}
\begin{cases}
p_1 + 1/2 > p_1 - 1/2 >  \ldots > p_n + 1/2 > p_n - 1/2 & \tu{($\pi$ superrégulier)} \cr
q_1 > q_2 > \ldots > q_r & \tu{($\rho$ régulier)} \cr
\forall i \forall j : p_i \pm 1/2 \neq q_j & \tu{($\pi$, $\rho$ disjoints).}
\end{cases}
\end{equation}
La condition de superrégularité sur $\pi$ peut aussi être formulée comme $p_i > p_{i+1} + 1$ pour tout $i \leq n-1$. 

Pour être cohomologique, le caractère infinitésimal doit non seulement être régulier, mais aussi entier. Par~\eqref{eq:Differentes} cela revient à
\[
p_i \pm 1/2 \in \tfrac {N-1}2 + \Z \quad \text{et} \quad q_j \in \tfrac {N-1}2 + \Z
\]
pour tout $i,j$ avec $1 \leq i \leq n$, $1 \leq j \leq r$. Comme
\[
\tfrac {N-1}2 + \Z = \tfrac {2n+r-1}2 + \Z = \tfrac {r-1}2 + \Z
\]
nous trouvons que $q_j \in \tfrac {r-1}2 + \Z$, ce qui signifie que $\rho$ (et donc $\sigma$) est algébrique. Pour $\pi$, la condition est que
\[
\forall i : \quad p_i \in \tfrac {N}2 + \Z = \tfrac r2 + \Z.
\]
La représentation $\pi$ est algébrique si et seulement si les $p_i$ appartiennent à $\tfrac {n-1}2 + \Z$. Ainsi, la condition sur $\pi$ est que
\begin{equation}\label{eq:Algebraic_or_L_Algebraic}
\begin{cases}
\text{Si $n + r$ est impair alors $\pi$ doit être algébrique.} \cr
\text{Si $n + r$ est pair alors  $\pi$ doit être demi-algébrique}.
\end{cases}
\end{equation}

Dans les autres cas, à savoir $E/F$ est $\R/\R$ ou $\C/\C$, les groupes $U_{n,F}$ et $U_{r,_F}$ sont déployés et isomorphes à $\GL_n$ et $\GL_r$. Donc, dans ces cas, le caractère infinitésimal de $\pi$ et $\rho$ est donné par $p_1 \geq \ldots \geq p_n$ et $q_1 \geq \ldots \geq q_r$, et il n'y a pas de $p_i'$, $q_j'$; les conditions d’intégralité, de super-régularité et de disjonction sont autrement les mêmes que dans \eqref{eq:SuperregularityDisjointness} et \eqref{eq:Algebraic_or_L_Algebraic}.

Nous revenons maintenant à la notation globale de la section précédente.

\begin{lemma}\label{lem:EstCohomologique}
Supposons que $\pi$ soit algébrique ou $L$-algebrique selon la recette dans \eqref{eq:Algebraic_or_L_Algebraic}, $\pi$ soit superrégulier, $\rho$ est algébrique régulier et que $\pi$ et $\rho$ sont disjoints en chaque place infinie de $F$. Alors chaque représentation automorphe $\tau \subset \cH_\pi$ est cohomologique.
\end{lemma}

D'après l'équation~\eqref{eq:H_psi}, $\tau$ apparaît dans le spectre discret de $U_N$ ; en particulier, elle est unitaire. Son caractère infinitésimal est entier par algébricité et régulier par l'argument de \eqref{eq:Differentes}--\eqref{eq:Algebraic_or_L_Algebraic}. Ainsi, $\tau_v$ est cohomologique pour chaque place infinie selon \cite[Thm.~1.8]{SalamancaRiba}.

\begin{theorem}\label{thm:IsEisenstein}
Soit $\pi_G$ une représentation de $\cH_\psi$, résiduelle ($\psi = \pi \otimes \sp(2) \oplus \rho$), provenant par la formation de résidus et de valeurs principales d’une représentation cuspidale $\pi_{M’}$ d’un sous-groupe de Levi de $G$. Alors $M’ = M = G_n \times U_r$, $\pi_M = \pi \otimes \sigma$ ($\sigma$ associé à $\rho$ par changement de base), et $s = 1/2$.
\end{theorem}
\begin{proof}
On a
\begin{align*}
M &= \prod_{i=1}^I G_{n_i} \times U_{r'} \quad \quad \textup{($I$ peut être $0$)}\cr
N & = r' + 2 \sum n_i
\end{align*}
et $\pi_G$ provient de $\pi_1 \otimes \pi_2 \otimes \cdots \otimes \pi_I \otimes \sigma'$ ($\pi_i$ cuspidale, $\omega_{\pi_i} = 1$ sur $A_{G_{n_i}}$, $\sigma'$ cuspidale). Soit $v$ une place finie non-ramifiée. Alors $\sigma'$ définit, par changement de base stable, une représentation $\sigma'_{E_w}$ de $\GL_{r', E_w}$ ($w | v$). Soit $t_w(\sigma')$ la matrice de Hecke de $\sigma'_{E_w}$ (voir \cite[p.~403]{MinguezParisBookProject}).

D'après Mok's Theorem~2.5.2 (avec $\kappa = 1$) il existe $\psi' \in \Psi_2(r', \xi_1)$ tel que $\sigma' \subset \cH_{\psi'}$. Alors
\[
\psi' = \rho_1 \otimes \sp(m_1) \oplus \cdots \oplus \rho_j \otimes \sp(m_j),
\]
(avec les conditions de Mok, p. 131). Pour $w$ finie on a $t_w(\sigma') = t_w(\psi')$. 

Supposons que $\pi_G$ provient de $\pi_M$. Alors
\begin{equation}\label{eq:C_1}
t_{\pi_G, w} = t_{\pi_1, w} q^{s_1} \oplus \cdots \oplus t_{\pi_{I, w}} q^{s_I} \oplus t_{\pi_1, w}\inv q^{-s_1} \oplus \cdots \oplus t_{\pi_I, w} q^{-s_I} \oplus t_{\sigma', w}
\end{equation}
Puisque $\pi_G \subset \cH_{\psi}$,
\begin{equation}\label{eq:C_2}
t_{\pi_G, w} = t_{\pi_w} q^{1/2} \oplus t_{\til \pi, w} q^{-1/2} \oplus t_{\rho,w}.
\end{equation}
On a $\omega_\pi(z \li z) = 1$ et donc $\omega_\pi$ est triviale sur $A_M$. Maintenant \eqref{eq:C_2} contient $3$ représentations cuspidales et \eqref{eq:C_1} contient $2I + \sum m_j r_j$ cuspidales. Donc
\begin{itemize}
\item[(a)] $I = 0$, $\pi_G$ est cuspidale (exclu),
\item[(b)] $I = 1$, donc $\sum m_j r_j  = 1$ et
\[
\pi |\ |^{1/2} \boxplus \til \pi |\ |^{-1/2} \boxplus \rho = \pi_1 |\ |^{s_1} \boxplus \til \pi_1 |\ |^{-s_1} \boxplus \rho_1
\]
avec $\rho_1$ cuspidale. Les degrés sont $(n,n,r)$ et $(n_1, n_1, r')$.
\end{itemize}
On sait que $n \neq r$ (parité). Donc $n = n_1$ et $r = r'$, $\pi |\ |^{1/2}$ ou $\til \pi |\ |^{-1/2}$ égale à $\pi_1 |\ |^{s_1} \ldots$ d’où $s_1 = 1/2$, $\pi = \pi_1$, $\rho = \rho_1$.
\end{proof}

\subsection{L'argument final}
Soient $n,r \in \Z_{>0}$. On pose $N = 2n + r$ et $G = U_N$. Soient $\pi$ et $\rho$ des représentations automorphes régulières, conjuguées autoduales cuspidales de $G_n$ et $G_r$ respectivement, telles que
\begin{itemize}
\item Si $n+r$ est impair, nous supposons que $\pi$ est algébrique et que $\rho$ est algébrique.
\item Si $n+r$ est pair, nous supposons que $\pi[1/2]$ est algébrique et que $\rho$ est algébrique.
\end{itemize}
On suppose aussi:
\begin{itemize}
\item $\pi$ est superrégulière et les caractères infinitésimaux du couple $\pi, \rho$ sont disjoints pour chaque place archimédienne (voir Équation~\eqref{eq:SuperregularityDisjointness});
\item $\eta(\pi) = (-1)^r$ et $\eta(\rho) = (-1)^{r-1}$ (voir Définition~\ref{def:Parite} et Équation~\eqref{eq:SignCondition}).
\end{itemize}

\begin{theorem}\label{thm:TheoremeUnitaire}
On a
\[
L(1/2, \pi \otimes \rho, R) \neq 0 \Longleftrightarrow L(1/2, \til a(\pi) \otimes a(\rho), R) \neq 0
\]
pour tout $a \in \Aut(\C)$.
\end{theorem}

(On rappelle que la représentation $R$ est $\std \otimes \std^\vee$, voir l’Équation~\eqref{eq:RepR}; la notation $\til a$ est définie dans Définition~\ref{def:tilde_action}.) 

Il suffit de prouver ``$\Rightarrow$''. Alors supposons que $L(1/2, \pi \otimes \rho, R) \neq 0$. Pour que $\rho = \BC(\sigma)$ pour une certaine représentation automorphe cuspidale $\sigma$ de $U_r$, il faut que $+1 = \kappa(\rho) = \eta(\rho) (-1)^{r-1}$ (\S\ref{sec:Asai}), et donc que
\begin{equation}\label{eq:eta_rho}
\eta(\rho) = (-1)^{r-1}.
\end{equation}
Pour l’argument ci-dessous, il faut que
\begin{equation}\label{eq:Has_Pole}
L(2s, \pi, \Asai^{(-1)^r})
\end{equation}
ait un pôle en $s = 1/2$. Cela donne (\S\ref{sec:Asai}) la condition
\begin{equation}\label{eq:eta_pi}
\eta(\pi) = (-1)^r.
\end{equation}

Soit $\psi$ le paramètre $\pi \otimes \sp(2) \oplus \rho$, comme considéré dans la Section~\ref{sec:ParamPsi}. Puisque la condition dans \eqref{eq:SignCondition} est vraie, $\psi$ définit, via soit $\xi_+$ soit $\xi_-$, une représentation du spectre discret de $G$.

Un résultat de Soudry~\cite{SoudryAsterique2005} (cf. \cite[Thm. 6.1, p. 134]{CoPiSh}) assure que $\rho = \CB(\sigma)$ pour un représentation automorphe discrète générique $\sigma$ de $U_r$, à condition que la fonction $L$ partielle
\begin{equation}\label{eq:Lpartielle}
L^S(\rho, s, \Asai^{\eps})
\end{equation}
ait un pôle en $s=1$, pour $\eps = (-1)^{r-1}$ et $S$ un ensemble suffisamment grand de places de $F$ incluant les places infinies (que $\rho = \CB(\sigma)$ puisse presque être déduit des résultats de Mok, mais pas tout à fait, car on ne saurait pas si $\sigma$ peut être pris générique).

Vérifions \eqref{eq:Lpartielle}. D'après \eqref{eq:eta_rho}, la fonction $L$ complète $L(s, \rho, \Asai^\eps)$ a un pôle en $s=1$. Donc nous devons établir le même fait pour la fonction $L$ partielle. Comme $\rho \simeq \rho^{c, \vee}$, la Proposition~3.6(ii) de \cite{JacquetShalika_Euler2} implique que la fonction $L$ partielle
\begin{equation}\label{eq:LpartiellePole}
L^{S_\ff}(s, \rho \times \rho^c) = L^{S_{\ff}}(\rho, s, \Asai^+) L^{S_{\ff}}(\rho, s, \Asai^-)
\end{equation}
a un pôle simple en $s = 1$, où $S_{\ff} \subset S$ est l'ensemble des places finies dans $S$. Ainsi, exactement l'une des fonctions $L^{S_{\ff}}(\rho, s, \Asai^\pm)$ a un pôle en $s = 1$. Si $L^{S_{\ff}}(\rho, s, \Asai^{-\eps})$ a un pôle, alors, comme les facteurs locaux $L_v(\rho, s, \Asai^{-\eps})$ ne s'annulent pas en $s =1$, nous trouvons que la fonction complète $L(s, \rho, \Asai^{-\eps})$ a un pôle en $s = 1$, et par conséquent, le pôle de
\[
L(s, \rho \times \rho^c) = L(s, \rho, \Asai^+) L(s, \rho, \Asai^-)
\]
en $s = 1$ n'est pas simple. Cela contredit~\cite{JacquetShalika_Euler2}. Ainsi, $L^{S_{\ff}}(s, \rho, \Asai^{\eps})$ a un pôle en $s = 1$. Comme les $L_v(s, \rho, \Asai^\eps)$ sont non nuls en $s=1$ pour $v$ infini, nous constatons que $L^S(s, \rho, \Asai^{\eps})$ a également un pôle. Cela signifie que la condition de~\cite[Thm.~6.1]{CoPiSh} est satisfaite. Donc $\sigma$ existe et $\rho_v$ est associée a $\sigma_v$ pour $v|\infty$ et $v$ non-ramifiée~\cite[p.~122]{CoPiSh}.

Nous allons maintenant démontrer que $\sigma$ est en fait cuspidale. Il suffit de montrer que $\sigma_v$ appartient à la série discrète pour $v|\infty$. Pour un tel $v$, nous savons que $\rho_v$ est associé à $\sigma_v$. Si l’extension $E \otimes_F F_v / F_v$ est isomorphe à $\C/\C$ ou $\R/\R$, l’énoncé est clair. Sinon, l’extension est $\C/\R$. Dans ce cas, soit
\begin{equation}\label{eq:phi_rho}
\varphi_\rho = \lhk (z/\li z)^{q_1}, \ldots, (z/\li z)^{q_r} \rhk, \quad \quad (q_i \equiv (r-1)/2 \mod 1 )
\end{equation}
le paramètre de $\rho$. Le caractère infinitésimal est $(q, -q)$ modulo $W \times W$ avec $q = (q_1, \ldots, q_r)$, voire Proposition~\ref{prop:Clozel1_1}. Le caractère infinitésimal de $\sigma$ est alors $q$ ($\Lie T_{U_r} \otimes \C = \C^r$). Maintentant $\sigma$ donne une représentation $\varphi_\sigma$ de $\C^\times$, étendue à $W_{\R}$, avec $q_i \in \frac {r-1}2 + \Z$, $q_i \neq q_j$. Donc $\sigma_v$ est une série discrète (cf.~\cite[Prop.~4.3.2]{BergeronClozelAsterisque}). En particulier $\sigma$ est cuspidale.

Nous visons maintenant à utiliser la Proposition 8.4 pour montrer que la série d’Eisenstein
\[
E(-, s) \colon I_s = \ind_P^G(\pi|\ |^s \otimes \sigma) \to \cA(G).
\]
admet un résidu en $s = 1/2$. D’après la Proposition, nous devons vérifier que le quotient
\[
\frac {L(s, \pi \otimes \BC(\sigma), R)L(2s, \pi, \Asai^{(-1)^r})}{L(s+1, \pi \otimes \BC(\sigma), R)L(2s+1, \pi, \Asai^{(-1)^r})}
\]
a un pôle en $s = 1/2$. Notez que le dénominateur est holomorphe autour de $s = 1/2$, et que le numérateur a un pôle (car le premier facteur n’est pas nul par l’hypothèse, le second facteur a un pôle par \eqref{eq:eta_pi} et la Définition~\ref{def:Parite}). Ainsi, $E(-, s)$ a bien un résidu en $s = 1/2$. Nous pouvons donc définir $\pi_G$ comme étant un facteur irréductible, non ramifié en dehors de $S$, de l’image de $I_{1/2} \to \cA^2(G)$ définie en évaluant $(s-1/2)E(-, s)$ en $s = 1/2$.

Pour $v|\infty$, les représentations $\pi_v$, $\sigma_v$ dans $\ind(\pi_v|\ |^s \otimes \sigma_v)$ sont tempérées, et par conséquent $\pi_{G, v}$ est le quotient de Langlands de $\ind(\pi_v|\ |^s \otimes \sigma_v)$. Par superrégularité nous constatons que $\pi_{G, v}$ est cohomologique (voir la Lemme~\ref{lem:EstCohomologique}). Ainsi, 
\[
\pi_G \tu{ est résiduelle et cohomologique}.
\] 

Nous utilisons maintenant l’argument consistant à restreindre les classes de cohomologie aux composantes de bord de la compactification de Borel-Serre. Plus précisément, comme dans la discussion avant~\eqref{eq:decomp_g_K_coh_V_iota}, nous trouvons une représentation algébrique $V$ de $\Res_{F/\Q} G$, $q \in \Z$ bien choisi, telle que $\pi_{G, \ff}$ intervienne dans
\[
H^q(\ig, K; \cA^2(G) \otimes_{\C} V_\C).
\]
En répétant les arguments après~\eqref{eq:decomp_g_K_coh_V_iota}, il existe un $(\ig, K) \times G(\A_{F, \ff})$-module admissible $\pi_G’ = \pi_{G, \infty}’ \otimes \pi’_{G, \ff}$ tel que $\pi_G’$ soit résiduel (discret et non cuspidal), $\pi_{G, \ff}’ \simeq a(\pi_{G, \ff})$ et $\pi_{G, \infty}’$ soit cohomologique pour $\C \otimes_{a\inv \circ \iota_0, E} V$ (pour la notation $\iota_0$ et $E$, voir ci-dessus~\eqref{eq:decomp_g_K_coh_V_iota}).

\begin{lemma}\label{lem:PsiPrime}
Le paramètre d’Arthur de $\pi_{G, \ff}’$ est $\psi_{\til a(\pi), a(\rho)}$.
\end{lemma}

Nous travaillons avec les $L$-groupes $\LG_{\Q}$, $\LL G_{N,\Q}$ définis sur $\Q$ comme expliqué dans la Section~\ref{sec:alg_Satake}. Notez également que $\BC  \colon \LG_{\Q} \to \LL G_{N,\Q}$ est défini sur $\Q$. Par le même argument que celui au début de la preuve du Lemme~\ref{lem:a_sigma_cusp} il suffit de vérifier que nous avons

\begin{equation}\label{eq:Lem8_10_goal}
\BC(c(\pi'_{G, v})) = \diag(q_v^{1/2}, q_v^{1/2}) \otimes c(\til a(\pi)_v) \oplus c(a(\rho)_v),
\end{equation}
Pour $a \in \Aut(\C)$ et $m \in \Z/2\Z$ nous définissons le signe $\eps_m \in \{\pm 1\}$ par
\[
\eps_m = a\lhk q_v^{ \tfrac{m-1}2} \rhk  \cdot q_v^{- \tfrac{m-1}2}.
\]
Nous avons alors les formules $a(c(\pi_{G, v})) = \eps_N c(a(\pi_{G, v})$ (par Lemme~\ref{lem:modulaire2}), $a(c(\pi_v)) = \eps_n c(a(\pi_v))$ et $a(c(\rho_v)) = \eps_r c(a(\rho_v))$ (pour les deux derniers, voir Exemple~\ref{exa:fonct_Sat_tordu}). En utilisant ces formules, nous calculons 
\begin{align}\label{eq:Lem8_10_computation}
\BC(c(\pi'_{G, v})) %
&= \BC (c(a(\pi_{G, v}))) \cr
&= \eps_N \BC( a(c(\pi_{G, v}))) \cr
&= \eps_N a(\BC(c(\pi_{G, v}))) \cr
&= \eps_N  a \lhk \diag \lhk q_v^{1/2}, q_v^{-1/2} \rhk \otimes c(\pi_v) \oplus \BC(c(\sigma_v))  \rhk \cr
&= \eps_N  a \lhk \diag \lhk q_v^{1/2}, q_v^{-1/2} \rhk \otimes a(c(\pi_v)) \oplus a(c(\rho_v)) \rhk \cr
&= \eps_N  a \lhk \diag \lhk q_v^{1/2}, q_v^{-1/2} \rhk \otimes \eps_n c(a(\pi_v)) \oplus \eps_r c(a(\rho_v))\rhk  \cr
 &=  \diag \lhk q_v^{1/2}, q_v^{-1/2} \rhk \otimes \eps_N \eps_n \eps_0 c(a(\pi_v)) \oplus \eps_N \eps_r c(a(\rho_v)).  
\end{align}
En utilisant $N = 2n + r$, nous calculons
\begin{align*}
a \lhk q_v^{ \tfrac {N-1} 2} q_v^{ \tfrac {n-1} 2} q_v^{-\tfrac {1} 2} \rhk 
q_v^{ - \tfrac {N-1} 2}
q_v^{ - \tfrac {n-1} 2}
q_v^{ \tfrac {1} 2} & = a \lhk q_v^{ \tfrac {n+r}2} \rhk q_v^{ - \tfrac {n+r}2} \cr
a\lhk q_v^{ \tfrac {N-1}2 } q_v^{ \tfrac {r-1}2 }\rhk 
q_v^{ -\tfrac {N-1}2 } q_v^{ -\tfrac {r-1}2 }
& = 1
\end{align*}
et donc $\eps_N \eps_n \eps_0 = \eps_{n + r}$  et $\eps_N \eps_r = 1$. Ainsi,
\[
\eps_N \eps_n \eps_0 c(a(\pi_v)) = \til a(c(\pi_v)), \quad \eps_N \eps_r c(a(\rho_v)) = a(c(\rho_v)).
\]
En remplaçant cela à la fin de \eqref{eq:Lem8_10_computation}, nous trouvons \eqref{eq:Lem8_10_goal}, d’où le Lemme~\ref{lem:PsiPrime}.

D’après le résultat de Soudry (voir~\eqref{eq:Lpartielle}), la représentation $a(\rho)$ est de la forme $\BC(\sigma’)$, où $\sigma’$ est une représentation automorphe discrète et générique de $U_r$. Nous rappelons que cela nécessite que la fonction $L$ partielle $L^S(s, a(\rho), \Asai^\eps)$ ait un pôle en $s = 1$ pour $\eps = (-1)^{r-1}$. Pour le voir : la fonction $L$ complète $L(s, a(\rho), \Asai^\eps)$ a un pôle par la Proposition~\ref{prop:ParityInvariance}. Ensuite, l’existence d’un pôle pour la fonction $L$ partielle en découle de la même manière que l’argument après~\eqref{eq:LpartiellePole}. En fait, $\sigma’$ est cuspidale par l’argument avant l’Équation~\eqref{eq:phi_rho}. En résumé:
\[
a(\rho) = \BC(\sigma’) \quad \textup{avec $\sigma’$ cuspidale et générique}.
\]
Par le Lemme~\ref{lem:PsiPrime}, $\pi_G’$ apparaît dans $\cH_{\psi_{\til a(\pi), a(\rho)}}$. Ainsi, le Théorème~\ref{thm:IsEisenstein} s’applique pour $\psi = \psi_{\til a(\pi), a(\rho)}$, donc $\pi_G’$ apparaît dans l’image du résidu en $s = 1/2$ de la série d’Eisenstein
\[
E(-, s) \colon \ind_P^G (\til a(\pi)|\ |^s \otimes \sigma’) \to \cA(G).
\]
Par la Proposition~\ref{prop:Lquot}, le quotient
\begin{equation}\label{eq:QuotientHasAPole}
\frac {L(s, \til a(\pi) \otimes \BC(\sigma'), R)}{L(s+1, \til a(\pi) \otimes \BC(\sigma'), R)}
\frac {L(2s, \til a(\pi), \Asai^{(-1)^r})} {L(2s+1, \til a(\pi), \Asai^{(-1)^r})}
\end{equation}
a un pôle en $s = 1/2$. 

Nous affirmons que
\begin{equation}\label{eq:ParityInvariance}
L(2s, \til a(\pi), \Asai^{(-1)^r})
\end{equation}
a un pôle en $s = 1/2$. 
Si $\pi$ est cohomologique, alors $\til a(\pi) = a(\pi)$, et comme $L(2s, \pi, \Asai^{(-1)^r})$ a un pôle, nous pouvons appliquer la propriété d’invariance du Proposition~\ref{prop:ParityInvariance} pour obtenir l’affirmation. Si $\pi[1/2]$ est cohomologique, on a $\til a(\pi) = a(\pi[1/2])[-1/2]$. Dans ce cas, nous calculons
\[
\eta(\pi) = \eta(\pi[1/2]) = \eta(a(\pi[1/2])) = \eta(a(\pi[1/2])[-1/2]).
\]
Les première et dernière égalités sont vraies parce que le caractère $[\pm1/2] \colon \A^\times_E/E^\times \to \C^\times$ est trivial sur $\A_F^\times/F^\times$, et donc leur torsion ne change pas la parité. L’égalité du milieu est la Proposition~\ref{prop:ParityInvariance}. Ainsi, \eqref{eq:ParityInvariance} est vrai.

Comme le dénominateur dans \eqref{eq:QuotientHasAPole} est holomorphe autour de $s = 1/2$, nous trouvons alors que $L(s, \til a(\pi) \otimes a(\rho), R) \neq 0$.

\section{Invariance des facteurs \(\epsilon\)}\label{sec:Eps}

\subsection{Cas symplectique}

Soit \(F\) un corps de nombres, \(r,t \geq 1\) des entiers.
Soit \(\pi\) (resp.\ \(\rho\)) une représentation automorphe cuspidale autoduale pour \(\GL_{r,F}\) (resp.\ \(\GL_{t,F}\)).
Rappelons que le ``root number'' \(\epsilon(1/2, \pi \times \rho)\) est dans \(\{\pm 1\}\) car on a
\[ \epsilon(1/2, \pi \times \rho) \epsilon(1/2, \pi^\vee \times \rho^\vee) = 1 \]
et \(\pi\) et \(\rho\) sont auto-duales par hypothèse.
On sait \cite[Theorem 1.5.3 (b)]{ArthurBook} que si \(\pi\) et \(\rho\) sont du même type (symplectique ou orthogonal) alors \(\epsilon(1/2, \pi \times \rho) = +1\).
Supposons donc que \(\pi\) et \(\rho\) ne sont pas du même type.
Supposons également que chacune de ces deux représentations est algébrique régulière ou demi-algébrique régulière et qu'en toute place archimédienne \(v\) de \(F\), soit les poids de \(\pi_v\) sont dans \(\tfrac{1}{2}+\Z\) et ceux de \(\rho_v\) dans \(\Z\), soit vice-versa.

Pour une place non-archimédienne \(v\) de \(F\) on note \(\LLC\) la correspondance de Langlands, bijection entre classes d'isomorphismes de représentations lisses irréductibles de \(\GL_N(F_v)\) à coefficients complexes et classes d'isomorphismes de représentations de Weil-Deligne Frobenius-semi-simples \cite[Definition 4.1.2]{Tate_corvallis} de dimension \(N\) à coefficients complexes, normalisée unitairement (c'est-à-dire que les représentations à caractère central unitaire correspondent aux représentations de Weil-Deligne de déterminant unitaire).
Aux places archimédiennes on note encore \(\LLC\) la bijection entre \((\gfr,K)\)-modules irréductibles (où \((\gfr,K)=(\mathfrak{gl}_N(\C), \mathrm{O}(n))\) si \(F_v \simeq \R\) et \((\gfr,K) \simeq (\mathfrak{gl}_N(\C) \times \mathfrak{gl}_N(\C), \mathrm{U}(N))\) si \(F_v \simeq \C\)) et représentations irréductibles semi-simples de dimension \(N\) de \(W_{F_v}\).

\begin{lemma} \label{lem:equiv_LL_tens}
  \begin{enumerate}
  \item Soit \(\sigma\) une représentation cuspidale pour \(\GL_{N,F}\).
    \begin{enumerate}
    \item On suppose \(\sigma\) L-algébrique régulière, c'est-à-dire que \(\sigma\) (resp.\ \(\sigma[1/2]\)) est algébrique régulière si \(N\) est impair (resp.\ pair).
      Alors pour toute place non-archimédienne \(v\) de \(F\) on a \(a(\LLC(\sigma_v)) = \LLC(\tilde{a}(\sigma)_v)\).
    \item On suppose \(\sigma\) demi-L-algébrique régulière, c'est-à-dire que \(\sigma\) (resp.\ \(\sigma[1/2]\)) est algébrique régulière si \(N\) est pair (resp.\ impair).
      Alors pour toute place non-archimédienne \(v\) de \(F\) on a \(a(|\cdot|_v^{1/2} \LLC(\sigma_v)) = |\cdot|_v^{1/2} \LLC(\tilde{a}(\sigma)_v)\).
    \end{enumerate}
  \item Pour toute place non-archimédienne \(v\) de \(F\) on a
    \[ a(|\cdot|_v^{1/2} \LLC(\pi_v) \otimes \LLC(\rho_v)) = |\cdot|_v^{1/2} \LLC(\tilde{a}(\pi)_v) \otimes \LLC(\tilde{a}(\rho)_v). \]
  \end{enumerate}
\end{lemma}
\begin{proof}
  \begin{enumerate}
  \item On sait que pour une représentation lisse irréductible \(\delta\) de \(\GL_N(F_v)\) on a \(a(\LLC(\delta)) = \LLC(a(\delta))\) (resp.\ \(a(|\cdot|_v^{1/2} \LLC(\delta)) = |\cdot|_v^{1/2} \LLC(\delta)\)) si \(N\) est impair (resp.\ pair).
    Les deux points s'en déduisent au cas par cas.
  \item Cela se déduit du point suivant: parmi \(\pi\) et \(\rho\), l'une est L-algébrique régulière et l'autre est demi-L-algébrique régulière.
  \end{enumerate}
\end{proof}

\begin{theorem} \label{thm:inv_epsilon_symp}
  Pour tout \(a \in \Aut(\C)\) on a
  \[ \epsilon(1/2, \tilde{a}(\pi) \times \tilde{a}(\rho)) = \epsilon(1/2, \pi \times \rho). \]
\end{theorem}
\begin{proof}
  On va montrer l'équivariance pour \(\Aut(\C)\), c'est-à-dire
  \[ \epsilon(1/2, \tilde{a}(\pi) \times \tilde{a}(\rho)) = a \left( \epsilon(1/2, \pi \times \rho) \right) \]
  par voie locale.
  Soit \(\psi: \A_F/F \to \C^\times\) un caractère continu non trivial.
  Pour fixer les idées on prend \(\psi = \psi_\Q \circ \Tr_{F/\Q}\) où \(\psi_\Q: \A/\Q \to \C^\times\) est le caractère trivial sur \(\Zhat\) et de restriction \(x \mapsto \exp(2i\pi x)\) à \(\R\).
  On fixe également une mesure de Haar \(dx_v\) sur \(F_v\) pour toute place \(v\) de \(F\) comme suit.
  \begin{itemize}
  \item Si \(v\) est réelle on prend la mesure de Lebesgue \(dx_v\) sur \(F_v \simeq \R\).
  \item Si \(v\) est complexe on prend la mesure \(2 da_v db_v\) où \(x_v = a_v + ib_v\).
  \item Si \(v\) est non-archimédienne on prend la mesure de Haar telle que \(\int_{\cO_{F_v}} dx_v=1\).
  \end{itemize}
  Cela donne (voir par exemple \cite[Proposition I.5.2]{Neukirch_ANT}) \(\vol(\A_F/F) = |d_F|^{1/2}\) où \(d_F\) est le discriminant de \(F\).
  On a par définition
  \[ \epsilon(1/2, \pi \times \rho) = \vol(\A_F/F)^{-rt} \prod_v \epsilon(1/2, \pi_v \times \rho_v, \psi_v, dx_v). \]
  Le terme \(\vol(\A_F/F)^{-rt}\) est rationnel car \(rt\) est pair.
  Par compatibilité de la correspondance de Langlands locale aux facteurs epsilon on a en toute place \(v\) de \(F\)
  \[ \epsilon(1/2, \pi_v \times \rho_v, \psi_v, dx_v) = \epsilon(1/2, \LLC(\pi_v) \times \LLC(\rho_v), \psi_v, dx_v) \]
  et nous allons calculer en terme des paramètres de Langlands.

  Montrons qu'à toute place non-archimédienne \(v\) de \(F\) on a
  \begin{equation} \label{eq:equiv_eps_nonarch_sympl}
    a (\epsilon(1/2, \pi_v \times \rho_v, \psi_v, dx_v)) = \epsilon(1/2, \tilde{a}(\pi)_v \times \tilde{a}(\rho)_v, a(\psi_v), dx_v).
  \end{equation}
  D'après la fonctorialité en le corps des coefficients pour les facteurs epsilon côté Galois (voir \cite[p.18]{Tate_corvallis} pour le cas où la monodromie est nulle; avec de la monodromie les facteurs epsilon sont définis \cite[p.21]{Tate_corvallis}) on a
  \begin{align*}
    a (\epsilon(1/2, \LLC(\pi_v) \times \LLC(\rho_v), \psi_v, dx_v))
    &= a(\epsilon(|\cdot|_v^{1/2} \LLC(\pi_v) \times \LLC(\rho_v), \psi_v, dx_v)) \\
    &= \epsilon(a(|\cdot|_v^{1/2} \LLC(\pi_v) \times \LLC(\rho_v)), a(\psi_v), dx_v)
  \end{align*}
  et le deuxième point du Lemme \ref{lem:equiv_LL_tens} nous montre que cela est égal à
  \[ \epsilon(|\cdot|_v^{1/2} \LLC(\tilde{a}(\pi)_v) \times \LLC(\tilde{a}(\rho)_v), a(\psi_v), dx_v) = \epsilon(1/2, \tilde{a}(\pi)_v \times \tilde{a}(\rho)_v, a(\psi_v), dx_v). \]
  Soit \(p\) la place de \(\Q\) sous \(v\), et notons \(a_p\) l'élément de \(\Zp^\times\) tel que \(a(\zeta) = \zeta^{a_p}\) pour \(\zeta^{p^N}=1\).
  On a donc \(a(\psi_v(x_v)) = \psi_v(a_p x_v)\) pour tout \(x_v \in F_v\), et \cite[(3.4.4)]{Tate_corvallis} nous donne
  \[ \epsilon(1/2, \tilde{a}(\pi)_v \times \tilde{a}(\rho)_v, a(\psi_v), dx_v) = \det(\LLC(\tilde{a}(\pi)_v) \otimes \LLC(\tilde{a}(\rho)_v))(a_p) \times \epsilon(1/2, \tilde{a}(\pi)_v \times \tilde{a}(\rho)_v, \psi_v, dx_v). \]
  Or \(\LLC(\pi_v) \otimes \LLC(\rho_v)\) se factorise par le groupe symplectique (c'est une conséquence de \cite[Theorem 1.4.2]{ArthurBook}) donc est de déterminant trivial.
  On en déduit
  \[ \epsilon(1/2, \tilde{a}(\pi)_v \times \tilde{a}(\rho)_v, a(\psi_v), dx_v) = \epsilon(1/2, \tilde{a}(\pi)_v \times \tilde{a}(\rho)_v, \psi_v, dx_v). \]

  Considérons maintenant les places archimédiennes.
  Pour \(a \in \frac{1}{2} \Z_{\geq 0}\) on note \(I_a = \Ind_{\C^\times}^{W_\R}(z \mapsto (z/|z|)^{2a})\).
  On déduit des formules \cite[(3.2.4) et (3.2.5)]{Tate_corvallis} pour \(F_v \simeq \R\)
  \[ \epsilon(|\cdot|^s I_a, \psi_v, dx) = i^{2a+1} \]
  et pour \(\iota: F_v \simeq \C\) et \(a,b \in \C\) tels que \(a-b \in \Z\)
  \[ \epsilon(z \mapsto \iota(z)^a \ol{\iota(z)}^b, \psi_v, dx_v) = i^{|a-b|} \]
  d'où l'on tire
  \[ \epsilon(|\cdot|^s I_a|_{F_v^\times}, \psi_v, dx_v) = (-1)^{2a}. \]
  Ces formules dans le cas \(a \in 1/2 + \Z\) nous suffiront pour calculer \(\epsilon(1/2, \pi_v \times \rho_v, \psi_v, dx_v)\).
  Notons la restriction du paramètre de Langlands de \(\pi_v\) (resp.\ \(\rho_v\)) à \(\C^\times\)
  \[ z \mapsto \diag((z/\ol{z})^{p_1}, \dots, (z/\ol{z})^{p_r}) \ \ \text{resp. } \diag((z/\ol{z})^{q_1}, \dots, (z/\ol{z})^{q_t}). \]
  Alors
  \[ \LLC(\pi_v) \otimes \LLC(\rho_v) \simeq \bigoplus_{\substack{1 \leq i \leq r \\ 1 \leq j \leq t \\ p_i+q_j>0}} I_{p_i+q_j}. \]
  Noter que nos hypothèses impliquent \(p_i+q_j \in 1/2 + \Z\).
  On en déduit
  \[ \epsilon(1/2, \pi_v \times \rho_v, \psi_v, dx_v) = \prod_{\substack{1 \leq i \leq r \\ 1 \leq j \leq t \\ p_i+q_j>0}}
    \begin{cases}
      (-1)^{p_i+q_j+1/2} & \text{ si } F_v \simeq \R, \\
      (-1)^{2p_i+2q_j} & \text{ si } F_v \simeq \C.
    \end{cases}
  \]

  Considérons maintenant les places archimédiennes dans leur ensemble.
  Rappelons que l'on identifie le caractère infinitésimal de \(\pi_v\) à une famille à un ou deux éléments \((p_\iota)_\iota\) où
  \begin{itemize}
  \item dans le cas réel où \(v\) correspond à un plongement \(\iota: F \hookrightarrow \C\) d'image contenue dans \(\R\), \(p_\iota = (p_1, \dots, p_r)\),
  \item dans le cas complexe où le choix de \(F_v \simeq \C\) correspond à \(\iota: F \hookrightarrow \C\), \(p_\iota = (p_1, \dots, p_r)\) et \(p_{\ol{\iota}} = p_\iota\) (l'égalité entre \(p_\iota\) et \(p_{\ol{\iota}}\) est particulière au cas des représentations autoduales).
  \end{itemize}
  De même pour \(\rho_v\), définissant \((q_\iota)_\iota\).
  La famille \((p_\iota)_\iota\) où \(\iota\) parcourt tous les plongements de \(F\) dans \(\C\) décrit le caractère infinitésimal de \(\otimes_{v|\infty} \pi_v\).
  On a
  \[ \prod_{v|\infty} \epsilon(1/2, \pi_v \times \rho_v, \psi_v, dx_v) = (-1)^{crt/2} \prod_{\iota: F \hookrightarrow \C} \prod_{\substack{1 \leq i \leq r \\ 1 \leq j \leq t \\ p_{\iota_i}+q_{\iota_j}>0}} (-1)^{p_{\iota,i}+q_{\iota_j}+1/2} \]
  où \(c\) désigne le nombre de places complexes de \(F\), en particulier le produit à gauche est dans \(\{\pm 1\}\).
  On sait que le caractère infinitésimal de \(\otimes_{v|\infty} \tilde{a}(\pi)_v\) est donné par \((p_{a^{-1} \iota})_\iota\), et de même pour \(\rho\), il est clair sur la formule ci-dessus que l'on a
  \begin{align*}
    \prod_{v|\infty} \epsilon(1/2, \tilde{a}(\pi)_v \times \tilde{a}(\rho)_v, \psi_v, dx_v)
    &= \prod_{v|\infty} \epsilon(1/2, \pi_v \times \rho_v, \psi_v, dx_v) \\
    &= a \left( \prod_{v|\infty} \epsilon(1/2, \pi_v \times \rho_v, \psi_v, dx_v) \right).
  \end{align*}
  En combinant avec \eqref{eq:equiv_eps_nonarch_sympl} on conclut
  \[ \epsilon(1/2, \tilde{a}(\pi) \times \tilde{a}(\rho)) = a(\epsilon(1/2, \pi \times \rho)) = \epsilon(1/2, \pi \times \rho). \]
\end{proof}

\begin{corollary} \label{cor:inv_pari_ord_sympl}
  On a
  \[ \mathrm{ord}_{s=1/2} L(s, \pi \times \rho) \equiv \mathrm{ord}_{s=1/2} L(s, \tilde{a}(\pi) \times \tilde{a}(\rho)) \mod 2. \]
\end{corollary}
\begin{proof}
  Rappelons qu'il existe \(f \in \R_{>0}\) tel que \(\epsilon(s, \pi \times \rho) = \epsilon(1/2, \pi \times \rho) f^{1/2-s}\), et l'équation fonctionnelle
  \[ \Lambda(s, \pi \times \rho) = \epsilon(s, \pi \times \rho) \Lambda(1-s, \pi^\vee \times \rho^\vee) \]
  où \(\Lambda\) désigne la fonction \(L\) complète (incluant les facteurs archimédiens).
  Comme \(\pi\) et \(\rho\) sont auto-duales, en notant \(g(s) = f^{s/2} \Lambda(s, \pi \times \rho)\) on a
  \[ g(s) = \epsilon(1/2, \pi \times \rho) g(1-s) \]
  et il en résulte que l'ordre d'annulation en \(s=1/2\) de \(\Lambda(s, \pi \times \rho)\) est pair (resp.\ impair) si \(\epsilon(1/2, \pi \times \rho) = +1\) (resp.\ \(-1\)).
  Les facteurs archimédiens sont holomorphes non nuls en \(s=1/2\), donc les ordres d'annulations de \(\Lambda(s, \pi \times \rho)\) et \(L(s, \pi \times \rho)\) en \(s=1/2\) sont égaux.
\end{proof}

\subsection{Cas conjugué autodual}

Soit \(E/F\) une extension quadratique de corps de nombres.
On note \(\{1,c\} = \Gal(E/F)\).
Soient \(\pi\) et \(\rho\) des représentations automorphes cuspidales conjuguées autoduales (\(\pi^\vee \simeq \pi^c\) et \(\rho^\vee \simeq \rho^c\)) pour \(\GL_{r,E}\) et \(\GL_{t,E}\).
On a encore \(\epsilon(1/2, \pi \times \rho^\vee) \in \{\pm 1\}\) car\footnote{Cette égalité est assez formelle: plus généralement pour un isomorphisme de corps de nombres \(a: E_1 \simeq E_2\) et des représentations automorphes cuspidales \(\pi_1\) et \(\pi_2\) pour des groupes linéaires sur \(E_2\) on a \(\epsilon(s, \pi_1^c \times \pi_2^c) = \epsilon(s, \pi_1 \times \pi_2)\).} \(\epsilon(s, \pi^c \times \rho) = \epsilon(s, \pi \times \rho^c)\).
Si \(\pi\) et \(\rho\) sont de même signe on sait \cite[Theorem 2.5.4 (b)]{Mok} que ce signe vaut \(+1\).

Nous aurons besoin du lemme suivant.

\begin{lemma} \label{lem:res_eta_E_F}
  Soit \(\eta_{E/F}: \A_F^\times/F^\times \to \{\pm 1\}\) le caractère d'ordre \(2\) correspondant à \(E/F\), c'est-à-dire de noyau l'image de \(N_{E/F}: \A_E^\times/E^\times \to \A_F^\times/F^\times\).
  Alors sa restriction \(\eta_{E/F}|_{\A^\times/\Q^\times}\) correspond à l'extension (quadratique ou triviale) \(\Q(\sqrt{d_E})/\Q\).
\end{lemma}
\begin{proof}
  On a un diagramme commutatif \cite[p.4]{Tate_corvallis}
  \[
    \begin{tikzcd}
      \A^\times/\Q^\times \ar[d, hook] \ar[r, "{\sim}" above, "{r_\Q}" below] & W_{\Q}^{\mathrm{ab}} \ar[d, "{t}"] \\
      \A_F^\times/F^\times \ar[r, "{\sim}" above, "{r_F}" below] & W_F^{\mathrm{ab}}
    \end{tikzcd}
  \]
  où les flèches horizontales sont les isomorphismes de réciprocité de la théorie du corps de classe et la flèche verticale \(t\) est le transfert.
  Notons \(s_{F/\Q}\) le signe de la représentation de permutation de \(W_{\Q}\) (ou \(\Gal(\Qbar/\Q)\)) sur
  \[ W_{\Q}/W_F \simeq \Gal(\Qbar/\Q) / \Gal(\Qbar/F) \simeq \Hom(F,\Qbar). \]
  Ce caractère à valeurs dans \(\{\pm 1\}\) est aussi le déterminant de \(\Ind_{W_F}^{W_\Q} 1\).
  On sait de plus \cite[Proposition 1.2]{Deligne_epsilon}
  \[ \det \left( \Ind_{W_F}^{W_\Q}(\eta_{E/F} \circ r_F^{-1}) \right) = (\eta_{E/F} \circ r_F^{-1} \circ t) \times s_{F/\Q} \]
  et comme \(\Ind_{W_E}^{W_F} 1 \simeq 1 \oplus (\eta_{E/F} \circ r_F^{-1})\) on déduit
  \begin{align*}
    (\eta_{E/F} |_{\A^\times/\Q^\times}) \circ r_\Q^{-1}
    &= \eta_{E/F} \circ r_F^{-1} \circ t \\
    &= s_{E/F} \times \det \left( \Ind_{W_F}^{W_\Q} \eta_{E/F} \right) \\
    &= \det \left( \Ind_{W_E}^{W_\Q} 1 \right) \\
    &= s_{E/\Q}
  \end{align*}
  et il ne reste plus qu'à observer que \(s_{E/Q}\), vu comme caractère de \(\Gal(\Qbar/\Q)\), a pour noyau \(\Gal(\Qbar/\Q(\sqrt{d_E}))\).
  En effet si on choisit \(\alpha \in E\) tel que \(E=\Q(\alpha)\) et un ordre \(\{\iota_1,\dots,\iota_d\} = \Hom(E,\Q)\) alors on a modulo les carrés de \(\Q^\times\)
  \[ d_E \sim \det (\iota_i (\alpha)^{j-1})_{1 \leq i,j \leq d}^2 = \prod_{1 \leq i < j \leq d} (\iota_i(\alpha) - \iota_j(\alpha))^2. \]
\end{proof}

\begin{theorem} \label{thm:inv_epsilon_conjautodual}
 Supposons que \(\pi\) et \(\rho\) sont de signes opposés, qu'elles sont algébriques ou demi-algébriques régulières et qu'en toute place archimédienne \(v\) de \(E\) les poids de \(\LLC(\pi_v) \otimes \LLC(\rho_v)\) sont demi-entiers.
  Alors pour tout \(a \in \Aut(\C)\) on a
  \[ \epsilon(1/2, \tilde{a}(\pi) \times \tilde{a}(\rho)^\vee) = \epsilon(1/2, \pi \times \rho^\vee). \]
\end{theorem}
\begin{proof}
  La preuve est similaire à celle du Théorème \ref{thm:inv_epsilon_symp}, mais un peu plus délicate dans le cas où \(rt\) est impair.
  On choisit \(\psi: \A_E/E \to \C^\times\) et les mesures de Haar \(dx_v\) sur les localisés \(E_v\) comme dans cette preuve, de sorte qu'on a encore
  \begin{equation} \label{eq:eps_prod_conjautodual}
    \epsilon(1/2, \pi \times \rho^\vee) = |d_E|^{-rt/2} \prod_v \epsilon(1/2, \pi_v \times \rho_v^\vee, \psi_v, dx_v).
  \end{equation}
  Aux places non-archimédiennes de \(E\) le début du calcul est identique: pour une telle place \(v|p\) on trouve
  \[ a(\epsilon(1/2, \pi_v \times \rho_v^\vee, \psi_v, dx_v)) = \det(\LLC(\tilde{a}(\pi)_v) \otimes \LLC(\tilde{a}(\rho)_v)^\vee)(a_p) \times \epsilon(1/2, \tilde{a}(\pi)_v \times \tilde{a}(\rho)_v^\vee, \psi_v, dx_v) \]
  et on en déduit (utilisant \(\hat{\Z}^\times \simeq \A^\times/\R_{>0} \Q^\times\))
  \begin{equation} \label{eq:equiv_epsilon_conjautodual_nonarch_compl}
    \frac{a \left( \prod_{v \nmid \infty} \epsilon(1/2, \pi_v \times \rho_v^\vee, \psi_v, dx_v) \right)}{\prod_{v \nmid \infty} \epsilon(1/2, \tilde{a}(\pi)_v \times \tilde{a}(\rho)_v^\vee, \psi_v, dx_v)} = \omega_{\tilde{a}(\pi)}^t \omega_{\tilde{a}(\rho)}^{-r} (\ol{r_\Q}^{-1}(a|_{\Q^\mathrm{ab}}))
  \end{equation}
  où \(\omega\) désigne le caractère central et \(\ol{r_\Q}: \A^\times/\R_{>0} \Q^\times \simeq \Gal(\Q^\mathrm{ab}/\Q)\) est l'isomorphisme de réciprocité.
  Le caractère \(\omega_{\tilde{a}(\pi)}^t \omega_{\tilde{a}(\rho)}^{-r}\) de \(\A_E^\times/E^\times\) est conjugué auto-dual de signe \((-1)^{rt}\):
  \begin{itemize}
  \item Soit \(\eta(\pi)\) le signe de \(\pi\), alors le caractère central \(\omega_\pi\) est également conjugué auto-dual, de signe \(\eta(\pi)^r\): cela résulte de \cite[Corollary 2.4.11]{Mok} en choisissant une place de \(F\) non déployée dans \(E\).
  \item On a \(\omega_{\tilde{a}(\pi)} = \tilde{a}(\omega_\pi)\), donc \(\tilde{a}(\omega_\pi)\) est également de signe \(\eta(\pi)^r\).
  \item Les deux premiers points s'appliquent également à \(\rho\) donc \(\omega_{\tilde{a}(\pi)}^t \omega_{\tilde{a}(\rho)}^{-r}\) est de signe \(\eta(\pi)^{rt} \eta(\rho)^{-rt} = (\eta(\pi) \eta(\rho))^{rt} = (-1)^{rt}\).
  \end{itemize}
  Si \(rt\) est pair, cela signifie que la restriction de \(\omega_{\tilde{a}(\pi)}^t \omega_{\tilde{a}(\rho)}^{-r}\) à \(\A_F^\times/F^\times\) est triviale, en particulier le terme de droite dans \eqref{eq:equiv_epsilon_conjautodual_nonarch_compl} vaut \(1\).
  Si \(rt\) est impair sa restriction à \(\A_F^\times/F^\times\) est le caractère \(\eta_{E/F}\) d'ordre \(2\) correspondant à l'extension quadratique \(E/F\) et le Lemme \ref{lem:res_eta_E_F} nous dit que le terme de droite dans \eqref{eq:equiv_epsilon_conjautodual_nonarch_compl} vaut \(a(\sqrt{d_E})/\sqrt{d_E}\).
  On simplifie donc \eqref{eq:equiv_epsilon_conjautodual_nonarch_compl} en
  \begin{equation} \label{eq:equiv_epsilon_conjautodual_nonarch}
    \frac{a \left( \prod_{v \nmid \infty} \epsilon(1/2, \pi_v \times \rho_v^\vee, \psi_v, dx_v) \right)}{\prod_{v \nmid \infty} \epsilon(1/2, \tilde{a}(\pi)_v \times \tilde{a}(\rho)_v^\vee, \psi_v, dx_v)} = (a(\sqrt{d_E})/\sqrt{d_E})^{rt}.
  \end{equation}

  Considérons maintenant les places archimédiennes.
  \begin{itemize}
  \item Soit \(v\) une place complexe de \(E\).
    Alors \(\LLC(\pi_v) \otimes \LLC(\rho_v)\) est isomorphe à
    \[ z \in E_v^\times \mapsto \bigoplus_{\substack{1 \leq i \leq r \\ 1 \leq j \leq t}} \left( \iota(z)/\ol{\iota(z)} \right)^{p_{\iota,i}+q_{\iota,j}} \]
    et on trouve
    \[ \epsilon(1/2, \pi_v \times \rho_v, \psi_v, dx_v) = \prod_{\substack{1 \leq i \leq r \\ 1 \leq j \leq t}} i^{2|p_{\iota,i}+q_{\iota,j}|}. \]
  \item Toute place réelle de \(E\) fait partie d'une paire \(\{v_1,v_2\}\) de place réelles au-dessus d'une même place réelle de \(F\), et \(\LLC(\pi_{v_2}) \otimes \LLC(\rho_{v_2})\) est duale de (donc isomorphe à) \(\LLC(\pi_{v_1}) \otimes \LLC(\rho_{v_1})\).
    Comme \(\LLC(\pi_{v_1}) \otimes \LLC(\rho_{v_1})\) est isomorphe à une somme directe de représentations irréductibles \(I_a\) de \(W_\R\) avec \(a \in 1/2 + \Z_{\geq 0}\), et que \(\epsilon(I_a, \psi_{v_1}, dx_{v_1}) \in \{\pm 1\}\), on en déduit
    \[ \prod_{i=1}^2 \epsilon(1/2, \pi_{v_i} \times \rho_{v_i}, \psi_{v_i}, dx_{v_i}) = 1. \]
    Noter que l'existence d'une place réelle de \(E\) implique que \(rt\) est pair.
  \end{itemize}
  La contribution des places complexes est
  \[ \prod_{v \text{ complexe}} \epsilon(1/2, \pi_v \times \rho_v, \psi_v, dx_v) = \exp \left( \frac{i \pi}{2} \sum_{\iota \text{ complexe}} \sum_{\substack{1 \leq i \leq r \\ 1 \leq j \leq t}} |p_{\iota,i}+q_{\iota,j}| \right) \]
  et si \(d_\C\) est le nombre de places complexes de \(E\) ce nombre vaut \(\pm i^{d_\C rt}\).
  La contribution des places réelles est
  \[ \prod_{v \text{ réelle}} \epsilon(1/2, \pi_v \times \rho_v, \psi_v, dx_v) = 1. \]
  On en déduit que
  \[ \prod_{v|\infty} \epsilon(1/2, \pi_v \times \rho_v, \psi_v, dx_v) = i^{-d_\R rt} \exp \left( \frac{i \pi}{2} \sum_\iota \sum_{\substack{1 \leq i \leq r \\ 1 \leq j \leq t}} |p_{\iota,i}+q_{\iota,j}| \right), \]
  où \(d_\R\) est le nombre de places réelles de \(E\), reste inchangé lorsqu'on remplace \((\pi,\rho)\) par \((\tilde{a}(\pi), \tilde{a}(\rho))\).
  (En fait \(d_\R\) est pair car \(d_\R + 2d_\C = [E:\Q] = 2[F:\Q]\), et \(d_\R>0\) implique \(rt\) pair, donc \(d_\R r t\) est toujours divisible par \(4\).)
  On reformule
  \begin{equation} \label{eq:equiv_epsilon_conjautodual_arch}
    \frac{a \left( \prod_{v|\infty} \epsilon(1/2, \pi_v \times \rho_v^\vee, \psi_v, dx_v) \right)}{\prod_{v|\infty} \epsilon(1/2, \tilde{a}(\pi)_v \times \tilde{a}(\rho)_v^\vee, \psi_v, dx_v)} = (a(\sqrt{-1})/\sqrt{-1})^{d_\C r t}.
  \end{equation}

  Combinant \eqref{eq:eps_prod_conjautodual}, \eqref{eq:equiv_epsilon_conjautodual_nonarch} et \eqref{eq:equiv_epsilon_conjautodual_arch} on obtient
  \[ \frac{a(\epsilon(1/2, \pi \times \rho^\vee))}{\epsilon(1/2, \tilde{a}(\pi) \times \tilde{a}(\rho)^\vee)} = \frac{a(\sqrt{s_E (-1)^{d_\C}})}{\sqrt{s_E (-1)^{d_\C}}} \]
  où \(s_E \in \{\pm 1\}\) est le signe du discriminant \(d_E\).
  Enfin on a \(s_E = (-1)^{d_\C}\) (voir par exemple \cite[Lemma 2.2]{Washington}).
\end{proof}

\begin{corollary} \label{cor:inv_pari_ord_conjautodual}
  Sous les hypothèses du Théorème \ref{thm:inv_epsilon_conjautodual} on a
  \[ \mathrm{ord}_{s=1/2} L(s, \pi \times \rho^\vee) \equiv \mathrm{ord}_{s=1/2} L(s, \tilde{a}(\pi) \times \tilde{a}(\rho)^\vee) \mod 2. \]
\end{corollary}
\begin{proof}
  La preuve est presque identique à celle du Corollaire \ref{cor:inv_pari_ord_sympl}, en utilisant l'égalité (formelle)
  \[ \Lambda(s, \pi^c \times \rho^{c,\vee}) = \Lambda(s, \pi \times \rho^\vee). \]
\end{proof}

\newpage

\appendix

\section{Holomorphie et non-annulation de certains opérateurs d'entrelacement }
\label{sec:appendice}
\subsection{Généralités}\label{generalites}

Soit $F$ un corps local non-archimédien de caractéristique nulle.
On note $\vert.\vert$ sa valeur absolue usuelle.

Soit $n\geq1$ un entier.
Considérons le groupe  $\GL_n$ défini sur $F$.
On note $\theta$ son automorphisme défini par $\theta(g)=J{^tg}^{-1}J^{-1}$ où $J$ est la matrice antidiagonale dont les coefficients non nuls sont les  $J_{k,n+1-k}=(-1)^k$ pour $k=1,\dots,n$.
Il est involutif.
Pour une représentation admissible $\pi$ de $\GL_n(F)$, on note $\theta(\pi)$ la représentation $g\mapsto \pi(\theta(g))$.
On note $\check{\pi}$ la contragrédiente de $\pi$.
Si $\pi$ est irréductible, on sait que $\check{\pi}\simeq \theta(\pi)$.
Pour un nombre complexe $s$, on note $\pi\vert.\vert^s$ la représentation $g\mapsto \vert det(g)\vert^s\pi(g)$.
Fixons la paire de Borel habituelle $(B^{\GL_n},T^{\GL_n})$: $B^{\GL_n}$ est le sous-groupe des matrices triangulaires supérieures et $T^{\GL_n}$ celui des matrices diagonales.
Un groupe de Levi standard $M$ de $\GL_n$ s'écrit $M=\GL_{n_1} \times \dots \times \GL_{n_t}$ avec $n_1+\dots+n_t=n$.
Notons $P$ l'unique sous-groupe parabolique standard de composante de Levi $M$.
Pour des représentations admissibles $\pi_1,\dots,\pi_t$ de $\GL_{n_1}(F),\dots,\GL_{n_t}(F)$, on pose $\pi_1\times\dots \times \pi_t = \Ind_P^{\GL_n}(\pi_1 \otimes \dots \otimes \pi_t)$.
Nous appellerons représentation irréductible quasi-tempérée une représentation admissible irréductible $\pi$ de $\GL_n(F)$ qui, pour un bon choix de $M$, est de la forme $\pi_1\vert.\vert^{r_1}\times\dots \times \pi_t\vert.\vert^{r_t}$ où les $\pi_i$ sont irréductibles et de la série discrète (donc unitaires) et les $r_i$ sont des réels vérifiant $\vert r_i\vert< \frac{1}{2}$.
Signalons qu'inversement, il est connu qu'une telle induite est irréductible et que sa classe d'équivalence ne dépend pas de l'ordre des facteurs.

Soit $n\geq1$ un entier.
On note $G$ le groupe $Sp_{2n}$ défini sur $F$.
On en fixe une paire de Borel $(B,T)$.
Un groupe de Levi standard $M$ de $G$ s'identifie à un produit $M=\GL_{r_1} \times \dots \times \GL_{r_t}\times G_0$, où $G_0=Sp_{2m}$ et $r_1+ \dots + r_t+m=n$.
Pour des représentations admissibles $\pi_1,\dots,\pi_t$ de $\GL_{r_1}(F)$, \dots, $\GL_{r_t}(F)$ et $\sigma$ de $G_0(F)$, on définit comme dans le cas de $\GL_n$ la représentation induite $\Pi = \pi_1 \times \dots \times \pi_t \rtimes \sigma$ de $G(F)$.
Pour un élément $s=(s_1,\dots,s_t)\in {\mathbb C}^t$, on pose $\Pi(s) = \pi_1 \vert.\vert^{s_1} \times \dots \times \pi_t \vert.\vert^{s_t} \rtimes \sigma$.

Conservons le groupe de Levi $M$.
Notons $W_t$ le groupe des permutations $w$ de $\{\pm 1,\dots,\pm t\}$ telles que $w(-i)=-w(i)$ pour tout $i$. Soit $w\in W_t$. On note
$^wM$ le Levi standard tel que $^wM=\GL_{r_{w^{-1}(1)}}\times\dots\times \GL_{r_{w^{-1}(n)}}\times G_0$, où $r_{-i}=r_i$ pour tout $i=1,\dots,t$.
Pour $g=(g_1,\dots,g_t,g_0) \in M(F)$, on pose $^wg=(g_{w^{-1}(1)},\dots,g_{w^{-1}(t)},g_0)$, où $g_{-i}=\theta(g_i)$ pour tout $i=1,\dots,t$; c'est un élément de $^wM(F)$.
On fixe, ainsi qu'il est loisible, un élément $\dot{w}\in G(F)$ qui vérifie les deux conditions suivantes:
\begin{itemize}
\item $\dot{w}M\dot{w}^{-1}={^wM}$;
\item $\dot{w}g\dot{w}^{-1}={^wg}$ pour tout $g\in M(F)$.
\end{itemize}

\begin{remark}
  Notons $\mathcal{N}(M)$ l'ensemble des $g\in G(F)$ tels que $gMg^{-1}$ soit un groupe de Levi standard.
  On vérifie que l'application $w\mapsto \dot{w}$ se quotiente en une bijection de $W_t$ sur $\mathcal{N}(M)/M(F)$.
\end{remark}

Considérons des représentations $\pi_1,\dots,\pi_t,\sigma$ comme ci-dessus, supposons qu'elles sont de longueur finie. Soit $s=(s_1,\dots,s_t)\in {\mathbb C}^t$. Pour $w\in W_t$, posons
\begin{itemize}
\item $^ws=(s_{w^{-1}(1)},\dots,s_{w^{-1}(t)})$ où $s_{-i}=-s_i $ pour tout $i=1,\dots,t$;
\item $^w\pi_i=\pi_{w^{-1}(i)}$ pour $i=1,\dots,t$, où $\pi_{-i}=\theta(\pi_i)$ pour un tel $i$;
\item $^w\Pi = {^w\pi}_1 \times \dots \times {^w\pi}_t \rtimes \sigma$.
\end{itemize}

L'élément $\dot{w} $ associé à $w$ permet de définir un opérateur d'entrelacement non normalisé de $\Pi(s) $ dans  $(^w\Pi)(^ws)$.
Comme on le sait, il n'est défini qu'au sens des opérateurs méromorphes: il est bien défini par une intégrale convergente dans un domaine de la forme $\Re(s_1) \gg \Re(s_2) \gg \dots \gg \Re(s_t) \gg 0$ et se prolonge méromorphiquement à tout ${\mathbb C}^t$.
Cet opérateur dépend du choix de $\dot{w}$, qui n'est pas unique (il l'est seulement modulo le centre de  $M(F)$).
Pour éviter la question technique du choix de $\dot{w}$, nous dirons que deux opérateurs dépendant de variables complexes sont équivalents s'ils diffèrent par multiplication par une fonction holomorphe et sans zéros de ces variables.
Nous noterons $\sim$ cette relation d'équivalence.
L'opérateur ci-dessus est bien défini modulo équivalence, ce qui autorise à le noter simplement $M(w,\Pi,s)$ dans la mesure où seule sa classe d'équivalence nous importe.

Le groupe $W_t$ est un groupe de Coxeter, donc muni d'une fonction longueur $\lng$.
On  a l'équivalence d'opérateurs d'entrelacements
\begin{equation} \label{eq:decomp_op_entrel_long}
  M(w_1w_2,\Pi,s)\sim M(w_1,(^{w_2}\Pi),{^{w_2}s})M(w_2,\Pi,s) \text{ quand } \lng(w_1w_2) = \lng(w_1) \lng(w_2).
\end{equation}

\subsection{Position du problème et énoncé principal}
\label{sec:enonce_appendice}

On s'intéresse à un Levi standard propre maximal $M=\GL_r\times G_0$ de $G$, où $G_0=Sp_{2m}$ et $r+m=n$. On considère des représentations admissibles irréductibles $\pi$ de $\GL_r(F)$ et $\sigma$ de $G_0(F)$. L'origine de notre problème vient d'une situation globale où
\begin{itemize}
\item $F$ est le complété en une place $v$ d'un corps de nombres $\dot{F}$;
\item $\pi$ est la composante en $v$ d'une représentation automorphe irréductible,  cuspidale  et autoduale $\dot{\pi}$ de $\GL_r({\mathbb A}_{\dot{F}})$ (où ${\mathbb A}_{\dot{F}}$ est l'anneau des adèles de $\dot{F}$);
\item $\sigma$ est la composante en $v$  d'une représentation automorphe irréductible et cuspidale $\dot{\sigma}$ de $G_0({\mathbb A}_{\dot{F}})$.
\item par la description due à Arthur du spectre discret des groupes symplectiques, $\dot{\sigma}$ correspond à une représentation automorphe irréductible,  cuspidale  et autoduale $\dot{\rho}$ de $\GL_{2m+1}({\mathbb A}_{\dot{F}})$.
\end{itemize}

De telles données impliquent sur $\pi$ et $\sigma$ les hypothèses suivantes, qui sont les seules que nous retiendrons:

 (Hyp1) $\pi$ est quasi-tempérée;

 (Hyp2) $\sigma$ appartient à un paquet d'Arthur (qui est d'ailleurs un paquet de Langlands) associé à une représentation admissible irréductible $\rho$ de $\GL_{2m+1}(F)$, laquelle est autoduale et quasi-tempérée.

Soit ici $w$ l'unique élément non trivial de $W_1$. Pour $s\in {\mathbb C}$, on a défini l'opérateur $M(w,\pi \rtimes \sigma,s):\pi\vert.\vert^s \rtimes \sigma \to \theta(\pi) \vert.\vert^{-s} \rtimes \sigma$.
On définit un facteur de normalisation
\[ r(w, \pi \rtimes \sigma,s) = \frac{L(s, \pi \times \rho) L(2s, \pi, \wedge^2)}{L(s+1, \pi \times \rho) L(2s+1, \pi, \wedge^2)} \]
et l'opérateur normalisé
\[ N(w, \pi \rtimes \sigma, s) = r(w, \pi \rtimes \sigma, s)^{-1} M(w, \pi \rtimes \sigma, s). \]
Les fonctions $L$ sont définies via la correspondance de Langlands locale pour les groupes linéaires mais on sait qu'elles coïncident avec les fonctions définies antérieurement par Shahidi ou Jacquet, Piatetski-Shapiro et Shalika, cf. \cite{HenniartIMRN2003}.

\begin{remark}
  Pour définir ce que l'on appelle habituellement des opérateurs normalisés, il faudrait  définir plus précisément les relèvements $\dot{w}$ utilisés plus haut et ajouter dans le facteur de normalisation des facteurs $\epsilon$ convenables.
  On sait que c'est possible et de tels opérateurs normalisés sont équivalents aux opérateurs définis ci-dessus.
\end{remark}

Le but de l'appendice est de prouver le résultat suivant.

\begin{proposition} \label{pro:entrelac_hol_nonnul}
  Sous les hypothèses (Hyp1) et (Hyp2), l'opérateur $N(w, \pi \rtimes \sigma,s)$ est holomorphe pour $\Re(s) \geq \frac{1}{2}$, et non identiquement nul pour \(\Re(s) = \frac{1}{2}\).
\end{proposition}

\begin{remark}
  Dans le cas où \(\pi\) est essentiellement de la série discrète et \(\sigma\) est tempérée la première partie (holomorphie) de la Proposition \ref{pro:entrelac_hol_nonnul} est un cas particulier du théorème 3.2.1 de \cite{Moeglin_holomorphy}.
\end{remark}

\subsection{Rappels concernant $\GL_n$}
\label{sec:rappels_GL}

Revenons à un groupe $\GL_n$ et considérons un Levi standard propre maximal $M=\GL_{n_1}\times \GL_{n_2}$.
Soient $\pi_1$ et $\pi_2$ des représentations admissibles irréductibles de $\GL_{n_1}(F)$ et $\GL_{n_2}(F)$.
On pose $\Pi = \pi_1 \times \pi_2$.
Les opérateurs d'entrelacement se définissent comme dans le cas du groupe symplectique, le  groupe $W_t$ étant remplacé par le groupe $\mathfrak{S}_2$ des permutations de l'ensemble $\{1,2\}$.
En notant ${\bf w}$ l'unique élément non trivial de ce groupe, on définit le Levi standard $^{{\bf w}}M = \GL_{n_2} \times \GL_{n_1}$, la représentation $^{\bf w}\Pi = \pi_2 \otimes \pi_1$ et pour ${\bf s} = (s_1,s_2) \in {\mathbb C}^2$, l'élément $^{\bf w}{\bf s} = (s_2,s_1)$ ainsi que les représentations $\Pi({\bf s}) = \pi_1\vert.\vert^{s_1} \times \pi_2 \vert.\vert^{s_2}$ et $(^{\bf w}\Pi)(^{\bf w}{\bf s}) = \pi_2\vert.\vert^{s_2}\times \pi_1\vert.\vert^{s_1}$.
On a alors un opérateur d'entrelacement $M({\bf w},\Pi,{\bf s}): \Pi({\bf s}) \to (^{{\bf w}}\Pi)(^{{\bf w}}{\bf s})$.
On définit le facteur de normalisation
$$r({\bf w}, \Pi, {\bf s}) = \frac{L(s_1-s_2, \pi_1 \times \check{\pi}_2)}{L(s_1-s_2+1, \pi_1 \times \check{\pi}_2)}$$
et l'opérateur normalisé $N({\bf w}, \Pi, {\bf s}) = r({\bf w}, \Pi, {\bf s})^{-1} M({\bf w}, \Pi, {\bf s})$.

Supposons que $\pi_1$ et $\pi_2$ soient de la série discrète.
Montrons que
\begin{enumerate}
\item \label{it:N_GL_hol_iso}
  l'opérateur $N({\bf w},\Pi,{\bf s})$ est holomorphe quand $\Re(s_1-s_2)> -1$, et est un isomorphisme quand $|\Re(s_1-s_2)| < 1$;
\item \label{it:L_GL_hol}
  la fonction $s\mapsto L(s,\pi_1\times \pi_2)$ définie pour $s\in {\mathbb C}$ n'a pas de pôle pour $\Re(s)>0$.
\end{enumerate}

Pour $i=1,2$ il existe une décomposition $n_i=h_im_i$ et une représentation unitaire cuspidale $\rho_i$ de $\GL_{m_i}(F)$ de sorte qu'en posant $t_i=\frac{h_i-1}{2}$, $\pi_i$ soit une sous-représentation de l'induite $\rho_i \vert.\vert^{t_i} \times  \rho_i \vert.\vert^{t_i-1} \times \dots \times \rho_i \vert.\vert^{-t_i}$.
Le lemme \cite{MoeglinWaldspurgerANENS89} I.4 dit que tout pôle ${\bf s}$ de $N({\bf w},\Pi,{\bf s})$ vérifie $\Re(s_1-s_2) \equiv t_1+t_2\, \mod \, {\mathbb Z}$ et $\Re(s_1-s_2)<-\vert t_1-t_2\vert$. Ces conditions impliquent $\Re(s_1-s_2)\leq -1$, d'où la première partie de l'assertion \eqref{it:N_GL_hol_iso}.
La deuxième partie s'en déduit grâce à la relation de cocycle satisfaite par les opérateurs d'entrelacements normalisés (cf.\ \cite[(1) p.607]{MoeglinWaldspurgerANENS89}) qui implique que la composée
\[ N(\mathbf{w}, {}^\mathbf{w} \Pi, {}^\mathbf{w} \mathbf{s}) \circ N(\mathbf{w}, \Pi, \mathbf{s}) \]
est égale à \(f(s) \mathrm{id}\) où \(f\) est une fonction entière ne s'annulant pas.
D'après \cite{JPSS_GL3_2} théorème 8.2, on a $L(s,\pi_1\times \pi_2)=\prod_{j}L(s+j,\rho_1\times \rho_2)$, où $j$ parcourt les éléments de $\frac{1}{2}{\mathbb Z}$ tels que $j\equiv t_1+t_2\, \mod\, {\mathbb Z}$ et $\vert t_1-t_2 \vert \leq j \leq t_1+t_2$.
Puisque tout pôle de la fonction $ L(s,\rho_1\times \rho_2)$ vérifie $\Re(s)=0$, on en déduit \eqref{it:L_GL_hol}.

\subsection{Preuve de la proposition \ref{pro:entrelac_hol_nonnul}}
\label{preuve}

La représentation $\rho$ étant quasi-tempérée et autoduale, on peut l'écrire sous la forme
\[ \rho = \rho'_1 \times \dots \times \rho'_v \times \rho_1 \vert.\vert^{b_1} \times \check{\rho}_1 \vert.\vert^{-b_1} \times \dots \times \rho_u \vert.\vert^{b_u} \times \check{\rho}_u \vert.\vert^{-b_u} \]
où les $\rho'_i$ et les $\rho_j$ sont des représentations de la série discrète de groupes $\GL_{m_i}(F)$ et $\GL_{m_j}(F)$ et les $b_j$ sont des réels tels que $0 < b_j < \frac{1}{2}$.
On pose \(\rho' = \rho'_1 \times \dots \times \rho'_v\) (on a \(v>0\)) et \(\rho_\mathrm{nt} = \rho_1 \vert.\vert^{b_1} \times \dots \times \rho_u \vert.\vert^{b_u}\) (si \(u=0\) alors \(\rho_\mathrm{nt}\) est la représentation triviale du groupe trivial \(\GL_0(F)\)).
On a donc \(\rho \simeq \rho' \times \rho_\mathrm{nt} \times \check{\rho}_\mathrm{nt}\).
D'après Arthur, il correspond à $\rho$ un $L$-paquet $\Sigma(\rho)$ de représentations de $G_0(F) = Sp_{2m}(F)$, auquel appartient $\sigma$ et il correspond à $\rho'$ un $L$-paquet $\Sigma(\rho')$ d'un groupe $G_0'(F)=Sp_{2m'}(F)$ où $m'=m-m'_{v}$.
Un paquet d'Arthur est caractérisé par le transfert endoscopique de son caractère à un groupe linéaire tordu convenable (ici les groupes $\GL_{2m+1}(F) $  et $\GL_{2m'+1}(F)$ tordus).
Le transfert endoscopique étant compatible à l'induction, il en résulte que $\Sigma(\rho)$ est formé des sous-quotients irréductibles des induites $\rho_\mathrm{nt} \rtimes \sigma'$ quand $\sigma'$ décrit $\Sigma(\rho')$.
Il résulte du théorème 2.11 de \cite{MoeglinWaldspurger_GGP_SO_gen} et de l'hypothèse $0<b_j<\frac{1}{2}$ pour tout $j$ que ces représentations induites sont toutes irréductibles.
Ainsi il existe $\sigma' \in \Sigma(\rho')$ telle que $\sigma \simeq \rho_\mathrm{nt} \rtimes \sigma'$.
La représentation \(\pi\) étant également quasi-tempérée, on peut écrire
\[ \pi \simeq \pi_1 \vert.\vert^{a_1} \times \dots \times \pi_t \vert.\vert^{a_t} \]
où \(1/2 > a_1 \geq \dots \geq a_t > -1/2\) et les \(\pi_i\) sont de la série discrète.
Les opérateurs d'entrelacement étant fonctoriels en la représentation induite on peut pour la preuve supposer que ces isomorphismes sont des égalités.
On a écrit
\[ \Pi \ =\  \pi \rtimes \sigma \ =\  \pi \times \rho_\mathrm{nt} \rtimes \sigma' \ =\  \pi_1 \vert.\vert^{a_1} \times \dots \times \pi_t \vert.\vert^{a_t} \times \rho_1 \vert.\vert^{b_1} \times \dots \times \rho_u \vert.\vert^{b_u} \rtimes \sigma'. \]

Le facteur de normalisation se décompose, en développant le carré extérieur d'une somme directe,
\begin{align*}
  &\ r(w, \pi \rtimes \sigma,s) \\
  = &\ \frac{L(s, \pi \times \rho) L(2s, \pi, \wedge^2)}{L(s+1, \pi \times \rho) L(2s+1, \pi, \wedge^2)} \\
  = &\ \prod_{1 \leq i \leq t} \frac{L(s+a_i, \pi_i \times \rho')}{L(s+1+a_i, \pi_i \times \rho')} \\
  &\ \times \prod_{\substack{1 \leq i \leq t \\ 1 \leq j \leq u}} \frac{L(s+a_i-b_j, \pi_i \times \check{\rho}_j) L(s+a_i+b_j, \pi_i \times \rho_j)}{L(s+1+a_i-b_j, \pi_i \times \check{\rho}_j) L(s+1+a_i+b_j, \pi_i \times \rho_j)} \\
  &\ \times \prod_{1 \leq i<j \leq t} \frac{L(2s+a_i+a_j, \pi_i \times \pi_j)}{L(2s+1+a_i+a_j, \pi_i \times \pi_j)} \\
  &\ \times \prod_{1 \leq i \leq t} \frac{L(2s+2a_i, \pi_i, \wedge^2)}{L(2s+1+2a_i, \pi_i, \wedge^2)}
\end{align*}
et pour \(\Re(s) \geq 1/2\) tous ces facteurs, sauf éventuellement les \(L(s+a_i-b_j, \pi_i \times \check{\rho}_j)\), sont holomorphes non nuls car de la forme \(L(z,\varphi)\) pour une représentation \(\varphi\) du groupe de Weil-Deligne \(W_F \times \SU(2)\) tempérée (i.e.\ d'image bornée) et \(\Re(z) > 0\).
En particulier seul le produit
\[ \frac{L(s, \pi \times \check{\rho}_\mathrm{nt})}{L(s+1, \pi \times \check{\rho}_\mathrm{nt})} = \prod_{\substack{1 \leq i \leq t \\ 1 \leq j \leq u}} \frac{L(s+a_i-b_j, \pi_i \times \check{\rho}_j)}{L(s+1+a_i-b_j, \pi_i \times \check{\rho}_j)} \]
joue un rôle dans la normalisation (on pourrait même se passer des dénominateurs).

\begin{remark}
  Nous avons justifié l'holomorphie des autres facteurs du facteur de normalisation en passant par le côté galoisien.
  On pourrait donner un argument directement du côté automorphe.
  On a rappelé en \ref{sec:rappels_GL} \eqref{it:L_GL_hol} que les fonctions \(L\) de Rankin-Selberg \(L(s, \tau_1 \times \tau_2)\) n'ont pas de pôle pour \(\Re(s)>0\) (pour \(\tau_1\) et \(\tau_2\) de la série essentiellement discrète, donc plus généralement pour \(\tau_1\) et \(\tau_2\) tempérées).
  La factorisation \(L(s, \tau \times \tau) = L(s, \tau, \wedge^2) L(s, \tau, \Sym^2)\) et le fait qu'aucun de ces deux facteurs ne s'annule implique qu'ils sont tous deux holomorphes pour \(\tau\) tempérée et \(\Re(s)>0\).
\end{remark}

Pour $\ul{s} \in {\mathbb C}^{t+u}$ et $\ul{w} \in W_{t+u}$, on a introduit l'opérateur $M(\ul{w}, \Pi, \ul{s}): \Pi(\ul{s})\to (^{\ul{w}}\Pi)(^{\ul{w}}\ul{s})$.
Considérons \(\ul{w} \in W_{t+u}\) défini par \(\ul{w}(i) = -(t+1-i)\) pour \(1 \leq i \leq t\) et \(\ul{w}(t+i) = t+i\) pour \(1 \leq i \leq u\).
Il est de longueur \(t(t-1)/2 + 2tu + t\).
En posant \(\ul{s} = (s, \dots, s, 0, \dots, 0) \in \C^{t+u}\) (\(t\) fois \(s\) et \(u\) fois \(0\)), on voit que $M(w,\Pi,s)$ coïncide avec l'opérateur d'entrelacement
\begin{align*}
  &\ \pi_1 \vert.\vert^{a_1+s} \times \dots \times \pi_t \vert.\vert^{a_t+s} \times \rho_1 \vert.\vert^{b_1} \times \dots \times \rho_u \vert.\vert^{b_u} \rtimes \sigma' \nonumber \\
  \xrightarrow{M(\ul{w}, \Pi, \ul{s})} &\ \theta(\pi_t) \vert.\vert^{-a_t-s} \times \dots \times \theta(\pi_1) \vert.\vert^{-a_1-s} \times \rho_1 \vert.\vert^{b_1} \times \dots \times \rho_u \vert.\vert^{b_u} \rtimes \sigma';
\end{align*}
plus exactement, ces deux opérateurs sont équivalents via l'identification \(\theta(\pi) \simeq \theta(\pi_t) \vert.\vert^{-a_t} \times \dots \times \theta(\pi_1) \vert.\vert^{-a_1}\), \(f \mapsto f \circ \theta^{-1}\).
Décomposons \(\ul{w} = \ul{w}_1 \ul{w}_2\) où
\begin{itemize}
\item \(\ul{w}_1(i) = u+i\) pour \(1 \leq i \leq t\) et \(\ul{w}_1(t+i) = i\) pour \(1 \leq i \leq u\),
\item \(\ul{w}_2(i) = t+i\) pour \(1 \leq i \leq u\) et \(\ul{w}_2(u+i) = -(t+1-i)\) pour \(1 \leq i \leq t\).
\end{itemize}
L'élément \(\ul{w}_1\) est de longueur \(tu\) et \(\ul{w}_2\) est de longueur \(tu + t(t-1)/2 + t\), en particulier \(\lng(\ul{w}) = \lng(\ul{w}_1) + \lng(\ul{w}_2)\).
Ainsi \(M(\ul{w},\Pi,\ul{s})\) se décompose en
\begin{align*}
  &\ \pi_1 \vert.\vert^{a_1+s} \times \dots \times \pi_t \vert.\vert^{a_t+s} \times \rho_1 \vert.\vert^{b_1} \times \dots \times \rho_u \vert.\vert^{b_u} \rtimes \sigma' \\
  \xrightarrow{M(\ul{w}_1, \Pi, \ul{s})} &\ \rho_1 \vert.\vert^{b_1} \times \dots \times \rho_u \vert.\vert^{b_u} \times \pi_1 \vert.\vert^{a_1+s} \times \dots \times \pi_t \vert.\vert^{a_t+s} \rtimes \sigma' \\
  \xrightarrow{M(\ul{w}_2, {}^{\ul{w}_1} \Pi, {}^{\ul{w}_1} \ul{s})} &\ \theta(\pi_t) \vert.\vert^{-a_t-s} \times \dots \times \theta(\pi_1) \vert.\vert^{-a_1-s} \times \rho_1 \vert.\vert^{b_1} \times \dots \times \rho_u \vert.\vert^{b_u} \rtimes \sigma'.
\end{align*}
L'opérateur \(M(\ul{w}_1, \Pi, \ul{s})\) est égal à \(M(\mathbf{w}, \pi \times \rho_\mathrm{nt}, s)\) et se décompose en produit de \(tu\) opérateurs simples obtenus en induisant les opérateurs (pour \(1 \leq i \leq t\) et \(1 \leq j \leq u\))
\[ M(\mathbf{w}, \pi_i \times \rho_j, (a_i+s,b_j)): \pi_i \vert.\vert^{a_i+s} \times \rho_j \vert.\vert^{b_j} \longrightarrow \rho_j \vert.\vert^{b_j} \times \pi_i \vert.\vert^{a_i+s}. \]
Les opérateurs normalisés correspondant
\[ N(\mathbf{w}, \pi_i \times \rho_j, (a_i+s,b_j)) = \frac{L(s+1+a_i-b_j, \pi_i \times \rho_j^\vee)}{L(s+a_i-b_j, \pi_i \times \rho_j^\vee)} M(\mathbf{w}, \pi_i \times \rho_j, (a_i+s,b_j)) \]
sont d'après \ref{sec:rappels_GL} \eqref{it:N_GL_hol_iso} holomorphes pour \(\Re(s) \geq 1/2\) car alors
\[ \Re(a_i+s-b_i) > -1/2 + 1/2 - 1/2 = -1/2. \]
Ce sont même des isomorphismes pour \(\Re(s)=1/2\) car sous cette hypothèse on a également
\[ \Re(a_i+s-b_i) < 1/2+1/2-0 = 1. \]
L'opérateur \(M(\ul{w}_2, {}^{\ul{w}_1} \Pi, {}^{\ul{w}_1} \ul{s})\) se décompose en \(tu + t(t-1)/2 + t\) opérateurs simples obtenus en induisant les opérateurs d'entrelacement suivants:
\begin{itemize}
\item \(\pi_i \vert.\vert^{a_i+s} \rtimes \sigma' \to \theta(\pi_i) \vert.\vert^{-a_i-s} \rtimes \sigma'\) pour \(1 \leq i \leq t\)
\item \(\pi_i \vert.\vert^{a_i+s} \times \theta(\pi_j) \vert.\vert^{-a_j-s} \to \theta(\pi_j) \vert.\vert^{-a_j-s} \times \pi_i \vert.\vert^{a_i+s}\) pour \(1 \leq i < j \leq t\),
\item \(\rho_j \vert.\vert^{b_j} \times \theta(\pi_i) \vert.\vert^{-a_i-s} \to \theta(\pi_i) \vert.\vert^{-a_i-s} \times \rho_j \vert.\vert^{b_j}\).
\end{itemize}
Pour \(\Re(s) \geq 1/2\) on a respectivement \(\Re(a_i+s)>0\), \(\Re(a_i+s+a_j+s)>0\) et \(\Re(b_j+a_i+s)>0\), donc ces opérateurs d'entrelacement sont tous holomorphes (voir \cite[Proposition IV.2.1]{Waldspurger_Plancherel}).
Leur composée \(M(\ul{w}_2, {}^{\ul{w}_1} \Pi, {}^{\ul{w}_1} \ul{s})\) est donc holomorphe pour \(\Re s \geq 1/2\).

Finalement notre opérateur d'entrelacement normalisé \(N(w, \pi \rtimes \sigma, s)\) est composé de
\begin{itemize}
\item une fonction méromorphe
  \[ r(w, \pi \rtimes \sigma, s)^{-1} \frac{L(s, \pi \times \rho_\mathrm{nt}^\vee)}{L(s+1, \pi \times \rho_\mathrm{nt}^\vee)} \]
  qui est holomorphe non nulle pour \(\Re s \geq 1/2\),
\item un opérateur d'entrelacement normalisé induit de \(M(\mathbf{w}, \pi \times \rho_\mathrm{nt}, s)\), qui est holomorphe pour \(\Re s \geq 1/2\) et un isomorphisme pour \(\Re s = 1/2\),
\item un opérateur d'entrelacement (non normalisé) \(M(\ul{w}_2, {}^{\ul{w}_1} \Pi, {}^{\ul{w}_1} \ul{s})\) qui est holomorphe pour \(\Re s \geq 1/2\), et pour toute spécialisation de \(s\) non identiquement nul (voir \cite[\S IV.1]{Waldspurger_Plancherel} (12) et (10)).
\end{itemize}
Ainsi \(N(w, \pi \rtimes \sigma, s)\) est holomorphe pour \(\Re s \geq 1/2\) et non identiquement nul pour \(\Re s = 1/2\).

\begin{remark}
  On peut certainement généraliser la preuve pour montrer la non-nullité de l'opérateur d'entrelacement normalisé pour \(\Re s \geq 1/2\), en décomposant \(\ul{w}_1\) comme produit de deux permutations, la première faisant ``passer à droite des \(\rho_j \vert.\vert^{b_j}\)'' les \(\pi_i \vert.\vert^{a_i+s}\) tels que \(0 < \Re(a_i+s) < 1/2\), la deuxième ceux tels que \(\Re(a_i+s) \geq 1/2\).
  La première donne un opérateur d'entrelacement normalisé holomorphe bijectif, la seconde est en position ``de Langlands'' et en la joignant à \(\ul{w}_2\) donne un facteur de normalisation holomorphe non nul et un opérateur d'entrelacement (non normalisé) holomorphe.
\end{remark}

\subsection {Autres groupes classiques}
\label{sec:app_autres_gpes}

La même preuve marche aussi bien en remplaçant le groupe $G = \Sp_{2n}$ par des groupes $\SO_{2n+1}$ ou $\SO_{2n}$.
On n'a pas besoin que le groupe soit déployé: on n'a fixé de paire de Borel $(B,T)$ que pour définir la notion de Levi standard,  on peut aussi bien fixer une paire parabolique minimale.
Dans le cas de $\SO(2n+1)$, la notion d'opérateur normalisé nécessite de remplacer  les fonctions $L(2s,\pi,\wedge^2)$ par $L(2s,\pi,Sym^2)$, cela ne perturbe pas le raisonnement car on a vu dans la preuve que pour \(\pi\) quasi-tempérée et \(\Re(s) \geq 1/2\) ces deux fonctions \(L\) n'admettent pas de pôle.
Pour le groupe $\SO_{2n}$, il faut prendre garde que la représentation $\sigma$ à laquelle les opérateurs d'entrelacement ne touchent pas dans le cas de $\Sp_{2n}$ ou $\SO_{2n+1}$ peut être modifiée par un automorphisme extérieur de $\SO_{2m}$.
Mais cela ne change rien car les représentations autoduales $\rho$ de $\GL_{2m}(F)$ associées à ces deux représentations sont les mêmes.

Considérons un peu plus en détail le cas d'un groupe unitaire.
On fixe une extension quadratique $E$ de $F$.
On note $x\mapsto \bar{x}$ l'automorphisme galoisien non trivial de $E$.
Soit $n\geq1$ un entier, considérons le groupe $\GL_n$ défini sur $E$.
Le groupe $\GL_n(E)$ est muni de l'automorphisme $g\mapsto {^cg}$: $(^cg)_{i,j}=\bar{g}_{i,j}$  pour tous $i,j\in \{1,\dots,n\}$.
On note $^c{\theta}$ l'automorphisme $g\mapsto \theta(^cg)$.
Pour une représentation admissible $\pi$ de $\GL_n(E)$, on définit les représentations $^c\pi$ et $^c{\theta}(\pi)$ par $^c\pi(g)=\pi(^cg)$ et $^c{\theta}(\pi)(g)=\pi(^c{\theta}(g))$.
Une représentation irréductible $\pi$ est dite $c$-autoduale si $\pi$ est équivalente à $^c{\theta}(\pi)$.

Considérons le groupe unitaire $G$ d'un espace hermitien de dimension $N$ sur $E$ (relatif bien sûr à l'extension $E/F$).
On en fixe une paire parabolique définie sur $F$ et minimale.
Pour un Levi standard $M$, on a une décomposition
\[ M(F) = \GL_{r_1}(E) \times \dots \times \GL_{r_t}(E) \times G_0(F), \]
où $G_0$ est le groupe unitaire d'un espace hermitien de dimension $N_0$ sur $E$ et où $N=N_0+2\sum_{i=1,\dots,t}r_i$.
Fixons des représentations admissibles irréductibles $\sigma$ de $G_0(F)$ et $\pi_i$ de $\GL_{r_i}(E)$ pour $i=1,\dots,t$.
On pose $\Pi = \pi_1 \times \dots \pi_t \rtimes \sigma$.
Soit  $w \in W_t$.
On définit la représentation $^w\Pi$ en remplaçant  $\theta$ par $^c{\theta}$ dans les définitions du paragraphe \ref{generalites}.
On définit alors l'opérateur $M(w,\Pi,s):\Pi(s)\to (^w\Pi)(^ws)$.

Considérons le cas où $M$ est propre maximal.
Alors $\Pi$ s'écrit simplement $\pi \rtimes \sigma$.
On conserve l'hypothèse

 (Hyp1) $\pi$ est quasi-tempérée

\noindent du paragraphe \ref{sec:enonce_appendice}.
En utilisant les résultats de Mok, on modifie la seconde hypothèse en

(Hyp2) $\sigma$ appartient à un paquet d'Arthur (qui est d'ailleurs un paquet de Langlands) associé à une représentation admissible irréductible $\rho$ de $\GL_{N_0}(F)$, laquelle est c-autoduale et quasi-tempérée.

On sait définir les fonctions $L$ d'Asai de $\pi$.
La définition est donnée dans \cite{GanGrossPrasad} p.26 pour les représentations du groupe de Weil-Deligne et on utilise ensuite la correspondance de Langlands locale.
Cette définition coïncide avec celles précédemment connues, cf. \cite{HenniartIMRN2003} théorème 1.5.
En fait, il y a deux fonctions $L$ d'Asai, notées $\Asai^+$ et $\Asai^-$, cf. \cite{GanGrossPrasad} p. 26.
On identifiera les signes $+$ et $-$ à des éléments de $\{\pm 1\}$.
Pour l'unique élément non trivial $w$ de $W_1$, on pose
\[ r(w, \pi \rtimes \sigma, s) = \frac{L(s, \pi \times \check{\rho}) L(2s, \pi, \Asai^{(-1)^{N_0}})}{L(s+1,\pi \times \check{\rho}) L(2s+1, \pi, \Asai^{(-1)^{N_0}})}. \]
On définit ensuite l'opérateur normalisé $N(w, \pi \rtimes \sigma, s) = r(w, \pi \rtimes \sigma, s)^{-1} M(w, \pi \rtimes \sigma, s)$.
Modulo ces modifications, on a la proposition suivante, analogue de la Proposition \ref{pro:entrelac_hol_nonnul}.

\begin{proposition} \label{pro:entrelac_hol_nonnul_unitaire}
    Sous les hypothèses (Hyp1) et (Hyp2), l'opérateur $N(w,\pi \rtimes \sigma,s)$ est holomorphe pour $\Re(s) \geq \frac{1}{2}$, et non identiquement nul pour \(\Re(s) = \frac{1}{2}\).
\end{proposition}

La preuve est essentiellement la même que dans le cas d'un groupe symplectique, et nous nous contentons de pointer les différences.
Notre représentation $\rho$ s'écrit maintenant
\[ \rho = \rho'_1 \times \dots \times \rho'_v \times \rho_1 \vert.\vert^{b_1} \times {}^c \theta(\rho_1) \vert.\vert^{-b_1} \times \dots \times \rho_u \vert.\vert^{b_u} \times {}^c \theta(\rho_u) \vert.\vert^{-b_u}. \]
On doit prouver l'irréductibilité de l'induite \(\rho_\mathrm{nt} \rtimes \sigma'\).
Il n'y a pas de doute que la preuve de \cite[Théorème 2.11]{MoeglinWaldspurger_GGP_SO_gen} dans le cas symplectique ou spécial orthogonal vaut aussi pour les groupes unitaires (elle repose seulement sur des calculs de modules de Jacquet).
Mais l'irréductibilité dans le cas unitaire résulte aussi de la combinaison des propositions 9.1 et B.1 de \cite{GanIchino} (qui reposent en partie sur des résultats d'Heiermann).

Pour décomposer le facteur de normalisation on a utilisé dans le cas d'un groupe symplectique l'égalité
$$ L(2s,\pi,\wedge^2) = \big( \prod_{i=1}^t L(2s + 2a_i, \pi_i, \wedge^2) \big) \big(\prod_{1\leq i<j \leq t} L(2s + a_i + a_j, \pi_i \times \pi_j) \big).$$
On utilise maintenant l'égalité
$$ L(2s, \pi, \Asai^{\epsilon}) = \big( \prod_{i=1}^t L(2s + 2a_i, \pi_i, \Asai^{\epsilon}) \big) \big(\prod_{1\leq i<j\leq t} L(2s + a_i + a_j, \pi_i \times {^c\pi}_j)\big) $$
pour $\epsilon=\pm$.
On conclut comme dans le cas symplectique.

\bibliographystyle{plain}
\bibliography{biblio.bib}



\end{document}